\documentclass[a4paper,11pt, reqno]{amsart}
\usepackage[utf8]{inputenc}
\usepackage{amssymb,amsmath, mathabx, stmaryrd}
\usepackage{amscd}
\usepackage[all]{xy}
\usepackage[onehalfspacing]{setspace}
\usepackage{graphicx}
\usepackage{url}  
\usepackage[a4paper, left=2cm, right=2cm, top=3.5cm, bottom=3cm]{geometry}
\usepackage{tikz-cd}
\usepackage{tikz}
\usetikzlibrary{matrix,arrows,calc,snakes,patterns,decorations.markings}
\usepackage{array}
\usepackage{verbatim}
\usepackage{enumitem}
\usepackage{mathtools}
\usepackage{extarrows}
\usepackage{mathrsfs}
\usepackage[normalem]{ulem}
\usepackage{hyperref, xspace}
\usepackage[bb=libus]{mathalpha}

\newcommand{\KSODs}{KSODs\xspace}

\newcommand{\SODs}{SODs\xspace}

\def\C{\mathbb C}
\def\E{\mathbb E}
\def\F{\mathbb F}

\def\H{\mathbb H}

\def\L{\mathbb L}
\def\P{\mathbb P}

\def\Z{\mathbb Z}
\def\AA{\mathbf A}
\def\BB{\mathbf B}
\def\CC{\mathbf C}
\def\DD{\mathbf D}
\def\cd{\mathbf D}

\def\cf{\mathscr F}

\def\KK{\mathbf K}

\def\OO{\mathcal O}
\def\PP{\mathbf P}

\def\TT{\mathbf T}
\def\ct{\mathscr T}

\def\mm{\mathfrak m}

\def\sA{\mathscr{A}}
\def\sB{\mathscr{B}}
\def\sC{\mathscr{C}}
\def\sD{\mathscr{D}}
\def\sE{\mathscr{E}}
\def\sF{\mathscr{F}}
\def\sG{\mathscr{G}}
\def\sH{\mathscr{H}}

\def\sK{\mathscr{K}}
\def\sL{\mathscr{L}}
\def\sM{\mathscr{M}}

\def\sP{\mathscr{P}}
\def\sS{\mathscr{S}}

\def\sT{\mathscr{T}}

\def\chr{{\operatorname{char}}}

\def\im{{\operatorname{im}}}

\def\Spec{{\operatorname{Spec}}}

\def\Proj{{\operatorname{Proj}}}

\def\Spec{{\operatorname{Spec}}}

\def\Cl{{\operatorname{Cl}}}

\def\Pic{{\operatorname{Pic}}}

\def\id{{\operatorname{id}}}
\def\Bl{{\operatorname{Bl}}}

\def\Coh{{\operatorname{Coh}}}
\def\Qcoh{{\operatorname{Qcoh}}}

\def\Sing{{\operatorname{Sing}}}

\def\dg{{\operatorname{dg}}}

\def\gldim{{\operatorname{gl.dim}\ }}
\def\projdim{{\operatorname{pr.dim}}}
\def\rqdim{{\operatorname{rq.dim}}}

\def\Mod{{\operatorname{Mod}}}
\def\mod{{\operatorname{mod}}}
\def\ad{{\operatorname{ad}}}

\newcommand{\opname}[1]{\operatorname{\mathsf{#1}}}

\newcommand{\Hom}{\opname{Hom}}
\newcommand{\End}{\opname{End}}
\newcommand{\Ext}{\opname{Ext}}
\newcommand{\Cone}{\opname{Cone}}
\newcommand{\op}{^{op}}

\newcommand{\thick}{\opname{thick}\nolimits}

\def\wh{\widehat}

\def\Dsg{\DD^{\mathrm{sg}}}

\def\Db{\DD^b}

\def\Dperf{\DD^{\mathrm{perf}}}
\def\Perf{\DD^{\mathrm{perf}}}

\newcommand{\bH}{\mathsf{H}}

\newcommand{\bK}{\mathsf{K}}

\newcommand{\bQ}{\mathsf{Q}}
\newcommand{\bR}{\mathsf{R}}
\newcommand{\bP}{\mathsf{P}}
\newcommand{\bS}{\mathsf{S}}
\newcommand{\bT}{\mathsf{T}}

\newcommand{\proj}{\opname{proj}\nolimits}

\theoremstyle{plain}
\newtheorem{dummy}{dummy}[section]

\newtheorem{theorem}[dummy]{Theorem}
\newtheorem{proposition}[dummy]{Proposition}
\newtheorem{prop}[dummy]{Proposition}
\newtheorem{lemma}[dummy]{Lemma}
\newtheorem{corollary}[dummy]{Corollary}

\theoremstyle{definition}
\newtheorem{example}[dummy]{Example}
\newtheorem{definition}[dummy]{Definition}
\newtheorem{remark}[dummy]{Remark}

\newtheorem{setup}[dummy]{Setup}

\numberwithin{equation}{section}

\newlength{\arrow}
\newtheorem{assumption}[dummy]{Assumption}

\usepackage{setspace}
\newcommand{\marginparstretch}{0.6}
\let\oldmarginpar\marginpar
\renewcommand\marginpar[1]{\-\oldmarginpar[\framebox{\setstretch{\marginparstretch}\begin{minipage}{\marginparwidth}{\raggedleft\tiny #1}\end{minipage}}]{\framebox{\setstretch{\marginparstretch}\begin{minipage}{\marginparwidth}{\raggedright\tiny #1}\end{minipage}}}}

\def\bal{\begin{aligned}}
\def\eal{\end{aligned}}

\newcolumntype{P}[1]{>{\centering\arraybackslash}p{#1}}
\newcolumntype{M}[1]{>{\centering\arraybackslash}m{#1}}

\title{Categorical absorptions of cone singularities 
}

\author{Martin Kalck} \author{Nebojsa Pavic}
\address{University of Graz,
Institut für Mathematik und Wissenschaftliches Rechnen,
Heinrichstraße 36,
8010 Graz,
Austria}
\email{martin.kalck@uni-graz.at}

\address{Graduate School of Mathematical Sciences,
University of Tokyo,
Komaba,
Meguro,
Tokyo,
153-8914,
Japan
}
\email{kawamata@ms.u-tokyo.ac.jp}

\address{University of Graz,
Institut für Mathematik und Wissenschaftliches Rechnen,
Heinrichstraße 36,
8010 Graz,
Austria}
\email{nebojsa.pavic@uni-graz.at}

\begin{document}
\begin{abstract}
    We study Kuznetsov--Shinder's categorical absorption for certain cone singularities, generalizing their results for nodal singularities. In particular, we give explicit descriptions of the endomorphism algebras of tilting objects for categorical absorptions of cones over certain Fano varieties admitting a geometric exceptional sequence, in the sense of Bridgeland \& Stern. In the simplest ``split case'', these algebras are truncations of certain Calabi--Yau completions in the sense of Keller. The split case occurs for anticanonical projective cones over many Fano varieties (like projective spaces, smooth quadrics, del Pezzo surfaces of degree greater than $4$, smooth del Pezzo threefolds of degree five, and finite products of these varieties). 
    In general, the algebras are deformations of the split case. 

    As a consequence, we  obtain triangle equivalances between singularity categories of finite dimensional algebras and singularity categories of certain cone singularities, which also yields vanishing results in negative $\bK$-theory.

    In a joint appendix with Yujiro Kawamata, we give an explicit description of tilting objects for weighted projective spaces $\P(1^d, m)$.
\end{abstract}
\sloppy

\dedicatory{With an appendix by Martin Kalck, Yujiro Kawamata and Nebojsa Pavic}

\maketitle
\tableofcontents

\section{Introduction}

Derived categories of algebraic varieties capture their homological essence. For many smooth varieties, a lot is known about the structure of these categories. The starting point was Beilinson's seminal construction of full exceptional sequences on projective space.
Building on this, full exceptional sequences have been constructed for many smooth projective varieties, e.g. by Hille \& Perling, Kapranov, Kawamata, Kuznetsov and Orlov, respectively, cf. e.g. Kuznetsov's ICM surveys \cite{KuzICM1, KuzICM2} . 

Singular varieties appear naturally in many contexts like the minimal model program in birational geometry and theoretical physics. However, the study of their derived categories has only started comparatively recently and much less is known about their structure.
For singular projective (Gorenstein) varieties $X$, full exceptional sequences cannot exist, cf.\ \cite[Cor. 4.15]{ks2}, \cite{KPS}! 

Together with recent works of Kawamata \cite{kawamata2} and Karmazyn--Kuznetsov--Shinder \cite{kks} this motivates the study of \emph{Kawamata type semiorthogonal decompositions (\KSODs)}. These are admissible \SODs that naturally generalize exceptional sequences to the singular case \cite{KPS}. More precisely, we consider admissible \SODs of $\Db(X) \coloneqq \Db(\Coh \, X)$ of the form
\begin{align}\label{eq:KSODIntro}
\Db(X) = \langle \PP, \Db(R_1), \ldots, \Db(R_n) \rangle, \text{ where  $\PP \subset \Dperf(X)$ and $\Db(R_i) \coloneqq \Db(R_i\mbox{-}\mod)$}
\end{align}
for some finite-dimensional $\mathbb{k}$-algebras $R_i$, ``capturing the singular information of $X$''. Their derived categories of finite dimensional left $R_i$-modules $\Db(R_i)$ are categorical absorptions of singularities in the sense of Kuznetsov \& Shinder \cite{ks, ks2, ks3}.
Here, $\Dperf(X) \subset \Db(X)$ is the subcategory of perfect complexes on $X$. If $\PP=0$ and all $R_i \cong \mathbb{k}$, \KSODs specialize to full exceptional sequences. 
On the other hand, if $n=1$ and $\PP=0$, \KSODs specialize to \emph{tilting equivalences}
\begin{align}\label{def:tilting}
    \Db(X) \cong \Db(\End_{\Db(X)}(\ct)),
\end{align}
where $\ct \in \Perf(X)$ is a \emph{tilting object}, that is $\smash{\Hom_{\Db(X)}(\ct, \ct[i])=0}$ for all $i \neq 0$ and for all $\cf \in \Db(X)$, if $\smash{\Hom_{\Db(X)}(\ct, \cf[i])=0}$ for all $i \in \Z$ then $\cf=0$, see for example \cite[Theorem 7.6(2)]{hille-van-den-bergh}.

Examples of \KSODs and categorical absorptions for nodal varieties and varieties with quotient singularities are studied in \cite{kuznetsov-sextics, kks, kawamata2,  KPS,xiedelPezfib, xieNdp, pavic-shinderDelPezzo, kawamata3, kawamata4, TU, ks, ks2}. This has been used by Kuznetsov \& Shinder \cite{ks3}, to explain a relationship between 
residual categories (a.k.a Kuznetsov components) of del Pezzo threefolds and the corresponding prime Fano threefolds, respectively. 

 The only known examples of \KSODs in dimensions $d \geq 3$ are nodal varieties \cite{ks} and the three dimensional weighted projective space $\P(1, 1, 1, 3)$ \cite{kawamata3}.

In this paper, we generalize these examples to many  cone singularities in arbitrary dimensions. Our starting point were the examples of weighted projective spaces $\P(1^d, m)$ treated in the appendix (Section \ref{sec:appendix}). Our focus on cone singularities  takes the  obstructions to KSODs  for many odd-dimensional non-nodal isolated hypersurface singularities including ADE-singularities and certain compound Du Val singularities \cite{KPS, kalck-klapproth-pavic} into account.

Describing the finite dimensional algebras $R_i$ in \eqref{eq:KSODIntro} explicitly is challenging in general. For cyclic quotient singularities in dimension two, this has been achieved via noncommutative and categorical resolutions of singularities and a reduction to Wemyss's description of reconstruction algebras \cite{kalck-karmazyn}. Alternatively, Lekili \& Tevelev obtain a description of these algebras using mirror symmetry for singular curves \cite{LT}.

Remarkably, in this article, we are able to avoid almost all explicit computations and still get explicit descriptions of the algebras in many cases. The key insight is that the algebras (in the ``split case'') have a conceptual interpretation as ``truncations'' of Keller's Calabi--Yau completions \cite{KellerCY} (cf. Theorem \ref{T:Algebras}). Calabi--Yau completions are known to have explicit descriptions in many cases. In particular, in this setting, all the ``interesting'' relations can be conveniently derived from a single non-commutative polynomial (the so-called potential), by taking ``partial derivatives'', allowing us to express the algebras in a rather compressed way. 

\subsection{Main results}
Throughout the rest of the introduction, we assume that $\mathbb{k}$ is an algebraically closed field of characteristic $0$. 
Following \cite{BridgelandStern},
for any $k \in \Z$ an exceptional sequence $\E=(\sE_1, \ldots, \sE_n)$ in a triangulated category $\DD$ with Serre functor $\bS$ generates a \emph{helix} $\H=(\sE_i)_{i \in \Z}$ of type $(n, k)$
by setting 
    \begin{align}
        \sE_{i-n}:=\bS(\sE_{i})[1-k] \qquad \quad \text{ for all } \ i \in \Z.
    \end{align}
    A helix $\H$ is called \emph{geometric} if $\Hom(\sE_i, \sE_j[l])=0$ for all $l \neq 0$ and all $i<j$. In this case, we also call the correspoding exceptional sequence $\E$ \emph{geometric (of type $(n,k)$)}\footnote{For example, the Beilinson exceptional sequence on $\P^n$ generates a geometric helix of type $(n+1,n+1)$.}.
To explain our results in more detail, we describe the type of Fano varieties 
we are mostly concerned with.

\begin{definition}\label{def:verystrong-Fano}
A smooth $(d-1)$-dimensional Fano variety $Z$ over $\mathbb{k}$ is called \emph{strong}, if the following hold: 
\begin{enumerate}
    \item\label{def:verystrong-Fano1} $Z$ has a \emph{projectively normal anticanonical embedding},
    i.e. the anticanonical bundle $\omega_Z^{-1}$ is very ample, and the corresponding anticanonical embedding is projectively normal  (we recall the definition of projective normality in the beginning of subsection \ref{subsec:cone-and-localization}). 
    \item\label{def:verystrong-Fano2} The admissible subcategory $\mathcal{O}_Z^\perp$ of $\Db(Z)$ has a full geometric exceptional sequence $\L$ of type $(m, d-1)$.
\end{enumerate}
A strong Fano variety $Z$ is called \emph{very strong} if there is an exceptional sequence $\L$ as in \eqref{def:verystrong-Fano2} such that  $\E=( \mathbb{L} , \OO_Z )$ is a geometric exceptional sequence of type $(m+1 , d)$ in $\Db(Z)$.
\end{definition}

\begin{example}\label{ex:intro-projcones}
Examples of very strong Fanos include projective spaces, smooth quadrics, del Pezzo surfaces of degree (strictly) greater than $4$ and smooth del Pezzo threefolds of degree $5$ (see Example \ref{ex:delpezzo-projnormal} regarding condition \eqref{def:verystrong-Fano1} of Definition \ref{def:verystrong-Fano} and see the discussion in Example \ref{ex:del-pezzo} and \ref{ex:projspace-quadric} regarding condition \eqref{def:verystrong-Fano2}.
Moreover, finite products of very strong Fano varieties are again very strong Fano varieties, 
see  Lemma \ref{L:productsofgeom}.
\end{example}

As a main result of  this paper, we construct many new examples of KSODs.

\begin{theorem}[Corollary \ref{eq:projcone-absorption}]\label{C:ProjCone}
Let $X$ be a projective cone over a  strong Fano variety $Z$ given by the anticanonical embedding $Z \subset \P^r$. 
Then 
there is a Kawamata type semiorthogonal decomposition 
\begin{align}
\Db(X)=\langle \Db(A), \mathcal{O}_X , \mathcal{O}_X (1) \rangle \cong \langle\Db(A), \Db(\mathbb{k}), \Db(\mathbb{k}) \rangle
\end{align}
for a finite dimensional $\mathbb{k}$-algebra $A$,
where $\OO_X (1)$ is the restriction of $\OO_{\P^{r+1}} (1)$ via the embedding $X \subset \P^{r+1}$.
\end{theorem}

\noindent
Only for $Z=\P^1$ respectively $Z=\P^2$, Theorem \ref{C:ProjCone} was known before, cf.  \cite{kks} respectively \cite{kawamata3}. 

\begin{example}\label{Ex:algebras-intro}
  The finite dimensional algebra $A$ in Theorem \ref{C:ProjCone} is isomorphic to
\begin{enumerate}[label=(\alph*)]
    \item $\mathbb{k}[x]/(x^2)$ if $Z=\P^1$.
    \item $\mathbb{k}\overline{Q}_3/(J+I^2)$ if $Z=\P^2$, where 
    \begin{equation*}\begin{tikzpicture}[description/.style={fill=white,inner sep=2pt}]
\matrix (n) [matrix of math nodes, row sep=20em,
                 column sep=2em, text height=1.5ex, text depth=0.25ex,
                 inner sep=20pt, nodes={inner xsep=0.3333em, inner
ysep=0.3333em}] at (0, 0)
    {   1  && 2 \\};

    \draw[->] ($(n-1-1.east) + (-2mm,3mm)$) .. controls +(3.7mm,6mm) and
+(-3.7mm,+6mm) .. node[fill=white, scale=0.65] [midway]{$x_{1}$} ($(n-1-3.west) + (2mm,3mm)$);

  \draw[->, color=red] ($(n-1-3.west) + (2mm,-3mm)$) .. controls +(-3.7mm,-6mm)
and +(+3.7mm,-6mm) .. node[fill=white, scale=0.65, yshift=-5.5] [midway]{$x_{1}^*$} ($(n-1-1.east) + (-2mm,-3mm)$);

    \draw[->] ($(n-1-1.east) + (-1mm,2mm)$) .. controls +(3.7mm,4mm) and
+(-3.7mm,+4mm) .. node[fill=white, scale=0.65] [midway]{$x_{2}$} ($(n-1-3.west) + (1mm,2mm)$);

  \draw[->, color=red] ($(n-1-3.west) + (1mm,-2mm)$) .. controls +(-3.7mm,-4mm)
and +(+3.7mm,-4mm) .. node[fill=white, scale=0.65, yshift=-2] [midway]{$x_{2}^*$} ($(n-1-1.east) + (-1mm,-2mm)$);

    \draw[->] ($(n-1-1.east) + (0mm,1mm)$) .. controls +(4.5mm,2mm) and
+(-4.5mm,+2mm) .. node[fill=white, scale=0.65] [midway]{$x_{3}$} ($(n-1-3.west) + (0mm,1mm)$);

    \draw[->, color=red] ($(n-1-3.west) + (0mm,-1mm)$) .. controls +(-4.5mm,-2mm)
and +(+4.5mm,-2mm) .. node[fill=white, scale=0.65, yshift=2] [midway]{$x_{3}^*$} ($(n-1-1.east) + (0mm,-1mm)$);

\node at (-1.7, .04) {$\overline{Q}_3=$};

\node at (6, .04) {$J=\left(\sum_{i=1}^3 x_ix_i^* - x_i^* x_i \right)$ and $I=\left(x_1^*, x_2^*, x_3^*\right).$};

\end{tikzpicture}\end{equation*}
    \item $\mathbb{k}\overline{Q}/(J+I^2)$ if $Z=\P^1 \times \P^1$, here the algebra $A$ is not unique. We give two possible choices.
    \begin{equation*}
\small
\begin{tikzpicture}[description/.style={fill=white,inner sep=2pt}]
 \matrix (n) [matrix of math nodes, row sep=2em,
                 column sep=2em, text height=1.5ex, text depth=1ex,
                 inner sep=2pt, nodes={inner xsep=0.3333em, inner
ysep=0.3333em}] at (-6, 0)
    {   (0,-1)  && (-1,-1) && (-1,0) \\
          };
    \draw[->] ($(n-1-1.east) + (0,1mm)$)  to  node[fill=white, yshift=0.7mm, scale=1] [midway]{$x_{0}$}($(n-1-3.west) + (0mm,1mm)$) ;
    \draw[->] ($(n-1-1.east) + (0,-1mm)$) to node[fill=white, yshift=-0.7mm, scale=1] [midway]{$x_{1}$} ($(n-1-3.west) + (0mm,-1mm)$);
 \draw[<-] ($(n-1-3.east) + (0,1mm)$)  to  node[fill=white, yshift=0.7mm, scale=1] [midway]{$y_{0}$}($(n-1-5.west) + (0mm,1mm)$) ;
    \draw[<-] ($(n-1-3.east) + (0,-1mm)$) to node[fill=white, yshift=-0.7mm, scale=1] [midway]{$y_{1}$} ($(n-1-5.west) + (0mm,-1mm)$);

    \draw[<-, red] ($(n-1-5.south west) + (5mm,-0mm)$) .. controls +(-3.7mm,-4.5mm)
and +(+3.7mm,-4.5mm) .. node[fill=white, scale=1, yshift=0.5mm] [midway]{$y_1^*$} ($(n-1-3.south east) + (-3mm,0mm)$);

\draw[<-, red] ($(n-1-5.south west) + (6.5mm,-0mm)$) .. controls +(-3.7mm,-9.5mm)
and +(+3.7mm,-9.5mm) .. node[fill=white, scale=1, yshift=-1.5] [midway]{$y_0^*$} ($(n-1-3.south east) + (-4.5mm,0mm)$);

   \draw[->, red] ($(n-1-3.south west) + (5mm,-0mm)$) .. controls +(-3.7mm,-4.5mm)
and +(+3.7mm,-4.5mm) .. node[fill=white, scale=1, yshift=0.5mm] [midway]{$x_1^*$} ($(n-1-1.south east) + (-3mm,0mm)$);

\draw[->, red] ($(n-1-3.south west) + (6.5mm,-0mm)$) .. controls +(-3.7mm,-9.5mm)
and +(+3.7mm,-9.5mm) .. node[fill=white, scale=1, , yshift=-1.5] [midway]{$x_0^*$} ($(n-1-1.south east) + (-4.5mm,0mm)$);

 \matrix (n) [matrix of math nodes, row sep=2em,
                 column sep=2em, text height=1.5ex, text depth=1ex,
                 inner sep=2pt, nodes={inner xsep=0.3333em, inner
ysep=0.3333em}] at (2, 0)
    {   (-2,-1)  && (-1,-1) && (-1,0) \\
          };
    \draw[<-] ($(n-1-1.east) + (0,1mm)$)  to  node[fill=white, yshift=0.7mm, scale=1] [midway]{$x_{0}$}($(n-1-3.west) + (0mm,1mm)$) ;
    \draw[<-] ($(n-1-1.east) + (0,-1mm)$) to node[fill=white, yshift=-0.7mm, scale=1] [midway]{$x_{1}$} ($(n-1-3.west) + (0mm,-1mm)$);
 \draw[<-] ($(n-1-3.east) + (0,1mm)$)  to  node[fill=white, yshift=0.7mm, scale=1] [midway]{$y_{0}$}($(n-1-5.west) + (0mm,1mm)$) ;
    \draw[<-] ($(n-1-3.east) + (0,-1mm)$) to node[fill=white, yshift=-0.7mm, scale=1] [midway]{$y_{1}$} ($(n-1-5.west) + (0mm,-1mm)$);

    \draw[<-, red] ($(n-1-5.south west) + (5mm,-0mm)$) .. controls +(-3.7mm,-4.5mm)
and +(+3.7mm,-4.5mm) .. node[fill=white, scale=1, yshift=0.5mm] [midway]{$y_1^*$} ($(n-1-3.south east) + (-3mm,0mm)$);

\draw[<-, red] ($(n-1-5.south west) + (6.5mm,-0mm)$) .. controls +(-3.7mm,-9.5mm)
and +(+3.7mm,-9.5mm) .. node[fill=white, scale=1, , yshift=-1.5] [midway]{$y_0^*$} ($(n-1-3.south east) + (-4.5mm,0mm)$);

   \draw[<-, red] ($(n-1-3.south west) + (5mm,-0mm)$) .. controls +(-3.7mm,-4.5mm)
and +(+3.7mm,-4.5mm) .. node[fill=white, scale=1, yshift=0.5mm] [midway]{$x_1^*$} ($(n-1-1.south east) + (-3mm,0mm)$);

\draw[<-, red] ($(n-1-3.south west) + (6.5mm,-0mm)$) .. controls +(-3.7mm,-9.5mm)
and +(+3.7mm,-9.5mm) .. node[fill=white, scale=1, yshift=-1.5] [midway]{$x_0^*$} ($(n-1-1.south east) + (-4.5mm,0mm)$);

\end{tikzpicture}\end{equation*}
In both cases, $I=(x_0^*, x_1^*, y_0^*, y_1^*)$ and $J=\sum_{i=0}^1 ([x_i, x_i^*] + [y_i, y_i^*])$.  
\item Examples \ref{ex:del-pezzo} and \ref{ex:P2xP1-cone}
contain further explicit descriptions of algebras $A$. Moreover, in Section \ref{sec:algebra-description}, we explain in detail how to obtain such descriptions more generally, cf. also Section \ref{subsec:descriptionIntro}. 
\end{enumerate}  
\end{example}

\begin{remark}
 Note that a projective cone $X$ as in Theorem \ref{C:ProjCone} is a Fano variety of index $2$ with canonical bundle $\omega_X \cong \OO_X (-2)$.
 It follows that $\OO_X , \OO_X(1)$ is a Lefschetz exceptional collection (see e.g. \cite{kuznetsov-hpd, kuznetsov-isotropic} for the notion).
 The derived category $\Db(A)$ in Theorem \ref{C:ProjCone} is thus the corresponding residual category (see e.g. \cite{kuznetsov-smirnov}).  
  The proof shows that it is a Verdier localization of the residual category $\tilde{\AA}$ on the crepant resolution $Y$.
\end{remark}

In Theorem \ref{T:IntroMain}, we give a generalization of Theorem \ref{C:ProjCone} to certain projective varieties with a singularity that is (analytically) isomorphic to a singularity of a cone as in in Theorem \ref{C:ProjCone}.
We start with some preparation.

\begin{definition}\label{def:ACS}
Let $X$ be a variety over $\mathbb{k}$.
    We say $s \in X$ is an \emph{$\mathrm{ACS}$-singularity}, if $s\in X$ is complete locally the singularity of the cone 
        over a strong Fano variety 
        $Z$ (see Definition \ref{def:verystrong-Fano})
        with respect to the anticanonical embedding. 
\end{definition}

In particular,
        an $\mathrm{ACS}$-singularity is analytically isomorphic to the singular point of the cone $\Spec \left(\bigoplus_{n\geq 0} \bH^0(\omega_Z^{-n})\right)$.

        Furthermore, in Theorem \ref{T:IntroMain} we need the following assumption for  singular varieties with a crepant resolution.
        In particular, by Proposition \ref{prop:full-exc-ProjCones} and Theorem \ref{thm:strong-adherence-equiv} this assumption holds for projective cones in Theorem \ref{C:ProjCone}, where the crepant resolution is given by the blow-up of the singular point.

\begin{assumption}\label{A:KeyAssumptions}
    Let $X$ be a projective variety  of dimension $d$ over $\mathbb{k}$ 
    with a unique singular point $s \in X$. 
    We assume there is a  crepant resolution $Y \to X$ with exceptional divisor $\iota \colon D \hookrightarrow Y$ and we suppose $\Db(Y)$ admits a geometric exceptional sequence\footnote{Which will never be full.} $\sE_1, \ldots, \sE_m$ of type 
        $(m, d-1)$, such that $\iota^!$ 
restricts to an equivalence 
$\langle \sE_1, \ldots, \sE_m \rangle \xrightarrow{\cong} \mathcal{O}_D^\perp$ of triangulated categories.
\end{assumption}

If $s\in X$ is an $\mathrm{ACS}$-singularity (with corresponding strong Fano $Z$), then the blow-up $Y = \Bl_s (X) \to X$ is a crepant resolution\footnote{because the embedding $Z \subset \P^r$ is projectively normal and because it is the anticanonical embedding.} with exceptional divisor $D\cong Z$.
Since $D $ is a strong Fano variety, there exists a geometric exceptional collection $\mathbb{L}$ of $\OO_D^\perp$.
This means that Assumption \ref{A:KeyAssumptions} asks about the lifting of the exceptional collection $\mathbb{L} \in \Db(D)$ to $\Db( Y )$.

\begin{theorem}[special case of Corollary \ref{eq:cone-absorption}] \label{T:IntroMain}
Let $X$ be a projective variety of dimension $d$ with a unique singular point $s \in X$. 
Assume $s\in X$ is an $\mathrm{ACS}$-singularity 
and suppose Assumption \ref{A:KeyAssumptions} holds, where $  \Bl_s ( X) \to X$ plays the role of $\pi\colon Y \to X$ in \ref{A:KeyAssumptions}.

Then there is a Kawamata type semiorthogonal decomposition \begin{align}\label{E:KSODIntro}
\Db(X)=\langle \Db(A), \CC \rangle\end{align} for a finite dimensional $\mathbb{k}$-algebra $A$ and $\CC \subset \Dperf(X)$. 
\end{theorem}

 We give conceptual descriptions of the algebras $A$ in Theorem \ref{T:Algebras}. 

\begin{remark}\label{rmk:silting}
    Theorem \ref{thm:strong-adherence-equiv} shows that for $\mathrm{ACS}$-singularities Assumption \ref{A:KeyAssumptions} can be viewed as a special case of ``adherence'' (see Definition \ref{def:strong-adherence}) between the exceptional collection $\mathbb{E}$ in $\Db(Y)$ and a certain autoequivalence $\bT$ on $\Db(Y)$. More precisely, $\bT$ is a composition of spherical twist functors that arise from a set of spherical objects that generate the kernel of $\pi_*\colon \Db(Y) \to \Db(X)$. This implies that $\pi_* \cong \pi_* \bT$.

    This notion of adherence generalizes Kuznetsov-Shinder's definition of adherence between an exceptional object and a spherical object \cite{ks}.
    Using our notion of adherence, we prove a generalization of Theorem \ref{T:IntroMain} for projective varieties $X$ with a unique isolated singularity  admitting a crepant resolution, and such that the corresponding exceptional collection $\mathbb{E}$ is geometric of type $(m, k)$ with $k \leq d-1$, see Corollary \ref{cor:veriderloc-negative-dg} and Example \ref{ex:nodal-ks}, which explains how to use Corollary \ref{cor:veriderloc-negative-dg} to recover Kuznetsov-Shinder's absorbtion of nodal threefolds admitting small resolutions.
    
    In the case when $X$ is the projective cone over a smooth $(d-1)$-dimensional Fano variety $Z$, such that $\OO_Z^{\perp}$ admits a geometric exceptional collection of type  $(m,k)$ with $k < d-1$ (instead of type $(m, d-1)$),  
     we obtain a finite dimensional $\dg$-algebra $A^\bullet$ with cohomologies concentrated in two non-positive degrees, and an admissible semiorthogonal decomposition  $\Db(X)=\langle \DD_{\mathrm{fg}} (A^\bullet), \CC\rangle$ with $\CC \subset \Dperf(X)$ (cf. also Corollary \ref{eq:projcone-absorption} for certain cone singularities, together with Remark \ref{rmk:k-geq-d-1}). 
     
    In particular, in the setting of 
    Corollary \ref{eq:projcone-absorption},
    $\Dperf(X)$ has a so-called silting object -- a generalization of tilting object due to Keller \& Vossieck that is currently very actively studied in representation theory.
\end{remark}

\subsection{Description of the algebras $A$ in Theorem \ref{T:IntroMain}}\label{subsec:descriptionIntro}

\begin{theorem}[Proposition \ref{prop:algebra-extensions}]\label{T:Algebras}
We keep the assumptions from Theorem \ref{T:IntroMain}. In particular, let $\mathbb{E} = ( \sE_1, \ldots, \sE_m )$ be the geometric exceptional sequence from Assumption \ref{A:KeyAssumptions}. Set
\begin{align}
E:=\End_Y\left(\bigoplus_{i=1}^m \sE_i\right).
\end{align} 
There is a ring epimorphism
\begin{align} \label{E:intro sq zero extension}
A \xrightarrow{\varphi} E,
\end{align}
whose kernel $I:=\ker \varphi$ satisfies $I^2=0$ (so $A$ is a \emph{square-zero extension} of $E$ by $I$).

Moreover, $I$ is an $E$-bimodule and there is an isomorphism of $E$-bimodules
\begin{align}\label{eq:bimodule-iso}
    I \cong \bS^{-1}_E(E)[d-2]
\end{align}
where $\bS^{-1}_E$ is the inverse Serre functor of $\Db(E)$.
\end{theorem}

\begin{proposition} \label{T:IntroSplit}
    We keep the notation and assumptions from Theorem \ref{T:Algebras}. Suppose additionally that the algebra homomorphism $\varphi$ in \eqref{E:intro sq zero extension} splits, i.e. there is an algebra homomorphism $\psi\colon E \to A$ such that $\varphi \psi =\id_E$. Then  there are algebra isomorphisms
\begin{align}\label{eq:introAquotientalgebra}
A \cong T_E(I)/(I \otimes_E I) \cong \Pi_{d-1}(E)/(I \otimes_E I),    
\end{align}
where $T_E(I)$ denotes the tensor algebra $T_E(I):=E \oplus I  \oplus (I \otimes_E I) \oplus \cdots$ and $(I \otimes_E I)$ denotes the two-sided ideal generated by $I \otimes_E I$. Moreover, $\Pi_{d-1}(E)$ denotes Keller's  $(d-1)$-Calabi--Yau completion\footnote{See \cite[Section 4]{KellerCY}.} of $E$, which is isomorphic to $T_E(I)$, by \eqref{eq:bimodule-iso} and the definition of $\Pi_{d-1}(E)$. 
\end{proposition}

\begin{theorem} \label{T:introalgebradesc}
     We keep the notation and assumptions from Theorem \ref{T:Algebras}. Assume additionally that one of the following conditions holds
     \begin{enumerate}
         \item $\dim X \leq 3$  (here $X$ is as in Theorem \ref{T:IntroMain} -- in particular, it need not be a projective cone)
         \item $X$ is a projective cone over a very strong Fano variety $Z$ given by the anticanonical embedding. Assume additionally that the  exceptional collection $\mathbb{L}$ of $\OO_Z^\perp$ in Definition \ref{def:verystrong-Fano} is given by coherent sheaves.
     \end{enumerate}
     
     Then the algebra homomorphism $\varphi$ in \eqref{E:intro sq zero extension} splits. In particular, $A$ can be described by \eqref{eq:introAquotientalgebra}.
\end{theorem}

The algebras $\Pi_{d-1}(E)$ have been described explicitly in many examples. In Example \ref{Ex:algebras-intro}, we combine this with Theorem \ref{T:introalgebradesc} (a) to describe the algebra $A$ in very simple cases. We refer to Corollary \ref{Cor:description_R_d} and Examples \ref{ex:del-pezzo}, \ref{ex:projspace-quadric} \& \ref{ex:P2xP1-cone} for more examples, where $A$ can be described.

\begin{remark}
    We remark that the finite dimensional $\mathbb{k}$-algebra $A$ in Theorem \ref{T:Algebras} might not only depend on the complete local singularity $s \in X$ and a fixed geometric exceptional collection $\E$.
    Namely, by work of Hochschild, the general algebras $A$ in Theorem \ref{T:Algebras} are deformations of the corresponding split extensions in Proposition \ref{T:IntroSplit}, cf. Subsection \ref{subs:non-split} -- in particular, Theorem \ref{T:Deformations}. 
    When $A$ is viewed as a path algebra of a quiver, this translates into ``deforming specific relations of this path algebra'', cf. e.g. Proposition \ref{prop:quotsing-nonsplit} for a special instance. 
    
    However, we do not know if such algebras are realized in geometric examples. 
    More concretely, we do not know if there are such examples $X$, with corresponding algebra $A \neq A_0 $, where $A_0 $ is the split algebra described in Proposition \ref{P:Hanihara} and Corollary \ref{Cor:description_R_d} (note that the split algebra $A_0$ appears in geometric examples by Theorem \ref{thm:projcone-split}). 
    We plan to investigate this further.
\end{remark}

\subsection{Application to  singularity categories}
Singularity categories were introduced  as a general framework for Tate cohomology \cite{Buchweitz}. More recently, they have been studied in birational geometry \cite{Wemyss18}, in relation with string theory \& homological mirror symmetry \cite{orlov-sing-1} and knot theory \cite{KR}.

\begin{definition}
    Let $X$ be a quasi-projective variety. The \emph{singularity category} of $X$ is the triangulated quotient category
    \begin{align}
        \Dsg(X):=\frac{\Db(\Coh \, X)}{\Dperf(X)}. 
    \end{align}
    Let $\Lambda$ be a left Noetherian ring (possibly noncommutative). The \emph{singularity category} of $\Lambda$ is the triangulated quotient category
    \begin{align}
        \Dsg(\Lambda):=\frac{\Db(\Lambda\mbox{-}\mod)}{\Dperf(\Lambda)}. 
    \end{align}
\end{definition}

There can be many non-trivial triangle equivalences $\Dsg(R) \cong \Dsg(\Lambda)$, where $R$ is a commutative ring  and $\Lambda$ is a non-commutative two-sided Noetherian ring.  For example, for every even dimensional ADE-singularity $R$ over $\C$ (except for $E_8$), one can construct \emph{uncountably} many algebras $\Lambda_\alpha$, such that
$\Lambda_\alpha\mbox{-}\mod \ncong \Lambda_\beta\mbox{-}\mod$ for $\alpha \neq \beta$ and $\Dsg(\Lambda_\alpha) \cong \Dsg(R)$ for all $\alpha$, see \cite{KIWY15}. 

In contrast, finding such $\Lambda$ that are \emph{finite dimensional} algebras over $\mathbb{C}$ is typically much harder and not possible in general (e.g for odd-dimensional non-nodal ADE-singularities, cf. \cite{kalck-klapproth-pavic}). 

Combining Theorem \ref{T:IntroMain}
with work of Orlov (\cite[Prop. 1.10]{orlov-hf}, \cite{OrlovIdemp}) and \cite[Corollary 4.5]{KPS} and using that $X$ has isolated Gorenstein singularities, yields many examples of such equivalences.

\begin{corollary}\label{C:Singulareq}
    In the setting of Theorem \ref{T:IntroMain}, there are equivalences of triangulated categories
    \begin{align}
         \Dsg(X) \cong \Dsg(A).
    \end{align}
In particular, the first negative $\bK$-group $\bK_{-1}(X)$ vanishes and thus
    \begin{align}\label{eq:singeqcompllocal}
    \Dsg\left(\widehat{\mathcal{O}}_{X, s}\right) \cong \Dsg(X).
    \end{align}
\end{corollary}

\begin{remark}
More generally, $\bK_{-1}(X)$ also vanishes for any $X$ as in Corollary \ref{eq:cone-absorption} by \cite[Proposition 6.12]{ks2}. 
Aspects of the $\bK$-theory of cone singularities are also discussed in \cite{Cortinasetal}.
\end{remark}

\begin{remark}
    Norihiro Hanihara kindly informed us about an alternative approach to the singular equivalences in \eqref{eq:singeqcompllocal} using \cite{Hanihara1}. Further, examples of singular equivalences (including non-Gorenstein singularities) are discussed in the appendix. 
\end{remark}

\subsection*{Notation} Throughout, $\mathbb{k}$ denotes a field. A variety is an integral, separated scheme of finite type over $\mathbb{k}$.

All triangulated categories are assumed to be $\mathbb{k}$-linear.
A triangulated category $\TT$ is cocomplete if it admits arbitrary direct sums.
Moreover, $\TT$ is said to be $\Hom$-finite if $\dim_{\mathbb{k}}\Hom_{\TT}(\sE,\sF)<\infty$ for all $\sE , \sF$ in $\TT$. 
 We denote $\Hom_{\TT}^{\bullet}( \sE , \sF ) : = \bigoplus_i \Hom_{\TT}(\sE , \sF[i])[-i]$.

For a $\dg$ $\mathbb{k}$-algebra $A^\bullet$, we denote by $\DD ( A^\bullet )$ the unbounded derived category of  left $\dg$-modules over $A^\bullet$.
Moreover, $\Dperf ( A^\bullet ) = \thick ( A^\bullet ) \subset \DD ( A^\bullet )$ denotes the subcategory of perfect $\dg$-modules, and $\DD_{\mathrm{fg}} ( A^\bullet ) \subset \DD ( A^\bullet )$ denotes the subcategory of $\dg$-modules with finite-dimensional total cohomology.
If the cohomology of $A^\bullet $ is concentrated in degree $0$ with cohomology $A = \bH^0 ( A^\bullet  )$, then we identify $\DD_{\mathrm{fg}} ( A^\bullet  )$ and $\Dperf(A^\bullet ) \subset \DD_{\mathrm{fg}} ( A^\bullet )$ with the bounded derived category of finitely generated left $A$-modules $\Db(A)$, and the subcategory of perfect complexes $\Dperf(A ) \subset \Db ( A )$, respectively.
Similarly, for a projective variety $X$ over $\mathbb{k}$ we write $\Db(X)$ for the
bounded derived category of coherent sheaves and $\Dperf(X) \subset \Db(X)$ stands for the subcategory of perfect complexes.

\subsection*{Acknowledgements}
 We are grateful to Severin Barmeier, Ben Briggs, Jonny Evans, Norihiro Hanihara, Osamu Iyama, Gustavo Jasso, Yujiro Kawamata, Sasha Kuznetsov, Evgeny Shinder, Michael Wemyss, Dong Yang for their help and interest in this project.
 M.K. also thanks David Ploog and Andreas Hochenegger for discussions in the initial phase of this work.

M.K. was partly supported by the Deutsche Forschungsgemeinschaft (DFG, German Research Foundation) –
Projektnummer 496500943.
Y.K. was supported by JSPS grant 21H00970.
N.P. would like to thank the GK1821 at the Universit\"at Freiburg for funding several shorter research stays where part of this work was done. 
M.K. and N.P. would also like to thank Hausdorff Research Institute for Mathematics in Bonn for the excellent working conditions during the JTP ``Algebraic geometry: derived categories, Hodge theory, and Chow groups''\footnote{Funded by the Deutsche Forschungsgemeinschaft (DFG, German Research Foundation) under Germany's Excellence Strategy – EXC-2047/1 – 390685813.} where part of this work has been done. M.K. worked on this article during ``writing retreats'' in 2024 and 2025 supported by RCC at Uni Graz. 

\subsection*{Open Access} 
For the purpose of open access, the authors have applied a Creative Commons Attribution (CC:BY) licence to any Author Accepted Manuscript version arising from this submission.

\section{Adherence and crepant resolutions}\label{sec:adherence}

Throughout this section, $\mathbb{k}$ is an arbitrary field.
Let $X$ be a projective variety over $\mathbb{k}$ and let $\pi : Y \to X $ be a resolution of singularities.
    Let $\pi_\ast : \Db(Y) \to \Db(X)$ be the induced pushforward and let  $\bT : \Db (Y) \to \Db (Y)$ be an autoequivalence.
    We say that \emph{$\pi_\ast $ is invariant under $\bT$},  if there is a natural transformation $\eta : \id \to \bT$, such that the transformation $\pi_\ast ( \eta )$ induces an isomorphism of functors $\pi_\ast \cong \pi_\ast \circ \bT$.

\begin{definition}\label{def:adherence}
    Let $\pi : Y  \to X$ be a resolution of singularities. Let $\bT : \Db(Y) \to \Db(Y) $ be an autoequivalence such that $\pi_\ast$ is invariant under $\bT$.
    We say that an exceptional collection $\E = ( \sE_1 , \ldots , \sE_m )$ in $\Db(Y)$ and $\bT$ are \emph{pre-adherent}, if 
\begin{align}\label{E:InitialExceptionalSequence}
\sE_1, \ldots, \sE_m, \bT( \sE_1 ), \ldots, \bT( \sE_m ) 
\end{align}
is an exceptional sequence in $\Db(Y)$ and such that $\ker \pi_\ast $ is contained in the thick admissible subcategory $\tilde{\AA}$ generated by \eqref{E:InitialExceptionalSequence}. 
\end{definition}
    
    In the situation of Definition \ref{def:adherence}, let us consider the cone of the morphisms $\eta_{\sE_i} \colon \sE_i\to \bT( \sE_i )$ in $\Db (Y)$ for all $1 \leq i \leq m$, and denote it by $\sK_i$.
    In particular, there are distinguished triangles
\begin{equation}\label{eq:main-distinguished-triangle}
    \sE_i \xrightarrow{\eta_{\sE_i} } \bT ( \sE_i  ) \xrightarrow{\varphi_i } \sK_i \xrightarrow{\delta_i} \sE_i [1]
\end{equation}
where $\varphi_i$ and $\delta_i$ are induced morphisms.
    
    \begin{proposition}\label{prop:sph-obj}
        Let $\pi : Y  \to X$, $\bT : \Db(Y) \to \Db(Y) $ and $\E$ be as in Definition \ref{def:adherence}.
        Assume in addition that $X$ has only isolated singularities and that  $\pi : Y \to X$ is a crepant resolution.
        Then the $\sK_i$ are contained in $\ker \pi_\ast$ and they are $d$-spherical objects in $\Db(Y)$.
    \end{proposition}

\begin{remark}\label{rmk:adherence-old-def}
After extending the definition of pre-adherence to categorical resolutions $\pi_\ast \colon \DD \to \Db(X)$, 
one can show that the analogous statement of Proposition \ref{prop:sph-obj} still holds, granted that the corresponding kernel category $\ker \pi_\ast$ has a Serre functor $\bS$ with $S \cong [n]$, for some integer $n$.
A straightforward computation shows in this case that the definition of pre-adherence with only one exceptional object $\E = \sE_1$ coincides with the defintion given by Kuznetsov-Shinder \cite[Definition 3.9]{ks}.

We will have two refinements of this definition, see Definition \ref{def:strong-adherence} below.
For nodal singularities, all $3$ notions will coincide.
In the general case, it is unclear to us in which generality these $3$ notions will agree.
\end{remark}

 We deduce Proposition \ref{prop:sph-obj} from the following two lemmas.

\begin{lemma}\label{lem:dCY-kernel}
    Let $X$ be a $d$-dimensional (not necessarily projective) Gorenstein variety with only isolated singularities and assume there is a crepant resolution of singularities $\pi : Y  \to X$. 
    Then the Serre functor $\bS_Y$ of $\Db(Y)$ restricts to the Serre functor $\bS_{\ker \pi_\ast }$ on $\ker  \pi_\ast  $ and $\ker \pi_\ast $ is a $d$-Calabi Yau category. 
    In other words, there is a natural isomorphism of functors $\bS_{\ker \pi_\ast } \cong [d]$ on $\ker  \pi_\ast $.
\end{lemma}

\begin{proof}
    Recall first that, since $\pi : Y \to X$ is crepant, the Serre functor $\bS_Y$ on $\Db(Y)$ is of the form $\bS_Y (-) \cong (-) \otimes \pi^\ast \omega_X[d]$.
    Hence, by projection formula 
    we see immediately that for any $\sK$ in $\ker \pi_\ast $ we have that
    \[
    \pi_\ast \bS_Y ( \sK ) \cong \pi_\ast (  \sK \otimes \pi^\ast \omega_X[d] ) \cong 0 .
    \]
    In other words, $\bS_Y ( \sK )$
    is contained in $\ker  \pi_\ast $.
    Similarly, the inverse Serre functor $\bS^{-1}_Y$ preserves $\ker \pi_\ast$.
    By uniqueness of the Serre functor it follows then that there is a natural isomorphism $\bS_Y \cong \bS_{\ker  \pi_\ast }$, where $\bS_Y$ denotes now the Serre functor of $\Db(Y)$ restricted to $\ker\pi_\ast$, by abuse of notation.

    To see that $\bS_{\ker  \pi_\ast } \cong [d]$, we note first that $\ker \pi_\ast$ is contained in $\Db_E (Y)$, the full triangulated subcategory of $\Db(Y)$ with supports on the exceptional locus $E$ of $\pi \colon Y \to X$.
    This follows directly from flat base-change applied to the fiber diagram obtained from $\pi \colon Y \to X$ and the open embedding $X \setminus \Sing(X) \hookrightarrow X$.
    Moreover, we have an equivalence $\kappa_Y^\ast \colon \Db_E ( Y ) \xrightarrow{\cong} \Db_E ( \widehat{Y} )$ (see e.g. \cite[Corollary 2.9]{OrlovIdemp}), where $\widehat{Y}$ and $\kappa_Y$ into the fiber product diagram
    \begin{equation*}\label{diag:formal-base-change}
\xymatrix{
    \wh{Y}\ar@{->}[d]_{\wh{\pi}}\ar@{->}[r]^{\kappa_Y} & Y\ar@{->}[d]^{\pi}\\
 \wh{X} \ar@{->}[r]^{\kappa_X} & X} ,
\end{equation*}
where  $\widehat{X} $ denotes the formal completion of $X$ at its singular points and $\kappa_X$ is the corresponding natural flat morphism.
Note that the inverse of $\kappa_Y^\ast$ is given by the push-forward ${\kappa_Y}_\ast \colon \Db_E ( \widehat{Y} )  \xrightarrow{\cong} \Db_E ( Y )$.
    Flat base-change with respect to this diagram, i.e. the isomorphism of functors $\widehat{\pi}_\ast \kappa_Y^\ast \cong \kappa_X^\ast \pi_\ast$ on $\Db_E (Y) \to \Db_{\Sing(\widehat X)} (\widehat{X})$, implies that the equivalence $\kappa_Y^\ast$ above restricts to a functor of kernel categories $\kappa_Y^\ast \colon \ker \pi_\ast \to \ker \widehat{\pi}_\ast $.
    Since also the isomorphism of functors $\pi_\ast {\kappa_Y}_\ast \cong {\kappa_X }_\ast \widehat{\pi}_\ast $ on $\Db_E ( \widehat{Y} ) \to \Db_{\Sing (X) }( X )$ holds,
    we see that ${\kappa_Y}_\ast $ restricts to a functor ${\kappa_Y}_\ast \colon \ker \widehat{\pi}_\ast \to  \ker \pi_\ast$.
    We conclude that $\kappa_Y^\ast \colon \ker \pi_\ast \xrightarrow{\cong} \ker \widehat{\pi}_\ast $ is an equivalence.
    
    Finally, we have a chain of isomorphisms
    \[
    \kappa_Y^\ast\omega_Y \cong \kappa_Y^\ast\pi^\ast \omega_X \cong \widehat{\pi}^\ast\kappa_X^\ast \omega_X \cong \OO_{\wh{Y}} 
    \]
    in $\Db( \wh{Y} )$,
    where we used that $\wh{X}$ is a disjoint union of complete local Gorenstein rings and hence there is an isomorphism $\omega_{\wh{X}}\cong \OO_{\wh{X}}$ of line bundles.
    Thus we obtain another chain of isomorphisms
    \[
    \kappa^\ast_Y\bS_{\ker\pi_\ast} ( \sK ) \cong \kappa^\ast_Y ( \sK \otimes \pi^\ast \omega_X ) [d] \cong \kappa_Y^\ast \sK [d] 
    \]
    in $\ker \wh{\pi}_\ast$ for all $\sK$ in $\ker \pi_\ast$.
    As $\kappa_Y^\ast \colon \ker\pi_\ast  \cong \ker \wh{\pi}_\ast$ is an equivalence,
    we have that
    \[
    \bS_{\ker\pi_\ast} ( \sK )  \cong \sK [d] 
    \]
    for all $\sK$ in $\ker \pi_\ast$ and thus that there is a natural isomorphism $\bS_{\ker\pi_\ast} \cong [d]$.
    This finishes the proof of the lemma.
\end{proof}

\begin{lemma}\label{lem:hom-computation}
    Let $\pi : Y  \to X$, $\bT : \Db(Y) \to \Db(Y) $ and $\E$ be as in Definition \ref{def:adherence}.
        Assume in addition that $X$ has only isolated singularities and that  $\pi : Y \to X$ is a crepant resolution.
        Then we have the following isomorphisms of $\Hom^\bullet$-spaces for all $1 \leq i , j \leq m$:
        \begin{align}
            \Hom^\bullet ( \sE_i , \sE_j  )  \xrightarrow[\cong]{(-)\cdot \delta_i [-1]} &\Hom^\bullet ( \sK_i , \sE_j [1] ) , \label{eq:hom-computation-iso1}\\
             \Hom^\bullet (   \sE_i    ,   \sE_j   )  \xrightarrow[\cong]{\varphi_j \cdot \bT(-)} & \Hom^\bullet ( \bT  \sE_i  , \sK_j )  ,  \label{eq:hom-computation-iso1_2}
        \end{align}
        where $\delta_i $ and $\varphi_i$ are as in \eqref{eq:main-distinguished-triangle}. 
        Moreover, by Serre duality these isomorphisms imply the following  isomorphisms of $\Hom^\bullet$-spaces
        \begin{align}
            \Hom^\bullet ( \sE_j , \sK_i ) & \cong  \Hom^\bullet ( \sE_i , \sE_j [d-1] )^\ast , \label{eq:hom-computation-iso2_2}\\
             \Hom^\bullet (  \sK_j , \bT  \sE_i  ) & \cong \Hom^\bullet (  \sE_i ,  \sE_j [d] )^\ast ,  \label{eq:hom-computation-iso2}
        \end{align}
        for all $1 \leq i , j \leq m$.
\end{lemma}

\begin{proof}
        We apply first $\Hom^\bullet ( - , \sE_j )$ on the distinguished triangle \eqref{eq:main-distinguished-triangle} and we obtain a distinguished triangle 
        \[
        \Hom^\bullet ( \sK_i , \sE_j ) \to \Hom^\bullet ( \bT  \sE_i  , \sE_j ) \to \Hom^\bullet ( \sE_i , \sE_j ) \xrightarrow{(-)\cdot \delta_i [-1]} \Hom^\bullet ( \sK_i , \sE_j [1] ) 
        \]
        in $\Db (k)$.
        Since $\E$ and $\bT$ are pre-adherent, we see that the second term of this triangle vanishes.
        Hence, we have an isomorphism of graded vector spaces
        \[
         \Hom^\bullet ( \sK_i , \sE_j ) \cong \Hom^\bullet (\sE_i , \sE_j ) \cdot \delta_i [-1]    
        \]
        Moreover, by Serre duality on $Y$ and by using the isomorphism $\bS_Y ( \sK_i ) \cong \sK_i [d]$ (Lemma \ref{lem:dCY-kernel}), together with this isomoprhism we get an isomomorphism
        \[
        \Hom^\bullet (  \sE_j ,   \sK_i  ) \cong \Hom^\bullet ( \sE_i ,  \sE_j [d-1] )^\ast .
        \]
        This shows \eqref{eq:hom-computation-iso1} and \eqref{eq:hom-computation-iso2_2}.
        Similarly, applying $\Hom^\bullet ( \bT  \sE_j  , - )$ on  \eqref{eq:main-distinguished-triangle} we obtain a distinguished triangle
        \[
        \Hom^\bullet ( \bT  \sE_j  , \sE_i ) \to \Hom^\bullet ( \bT  \sE_j  , \bT  \sE_i   ) \xrightarrow{\varphi_i \cdot (-)} \Hom^\bullet ( \bT  \sE_j  , \sK_i ) 
        \]
        in $\Db(\mathbb{k})$.
        Again since $\E$ and $\bT$ are pre-adherent, we obtain that the first term of this distinguished triangle vanishes.
        We conclude a graded isomorphisms of $\Hom$-spaces
\[
    \Hom^\bullet ( \bT  \sE_j  , \sK_i ) \cong  \varphi_i \cdot  \Hom^\bullet ( \bT  \sE_j    , \bT  \sE_i   )  \quad \text{and} \quad  \Hom^\bullet (  \sK_i , \bT  \sE_j  ) \cong \Hom^\bullet ( \bT \sE_j , \bT \sE_i [d] )^\ast .
\]
The second isomorphism follows again by Serre duality.
This shows \eqref{eq:hom-computation-iso1_2} and \eqref{eq:hom-computation-iso2} and finishes the proof of the lemma.
\end{proof}

    \begin{proof}[Proof of Proposition \ref{prop:sph-obj}] 
        By applying the functor $\pi_\ast : \Db(Y) \to \Db(X)$ to the distinguished triangle \eqref{eq:main-distinguished-triangle} and using the assumption that $\pi_\ast$ is invariant under $\bT$, we see immediately that the $\sK_i$ are contained in $\ker  \pi_\ast $.
        Moreover, by Lemma \ref{lem:dCY-kernel} there is an isomorphism $\bS_Y ( \sK_i ) \cong \sK_i [d]$ in $\Db(Y)$, where $\bS_Y$ is the Serre functor on $\Db(Y)$.
        Further, let us compute the graded $\Hom$-spaces of $\sK_i $.
        Applying $\Hom^\bullet ( \sK_i , - )$ to  
        \[
        \sE_i \xrightarrow{\eta_{\sE_i} } \bT ( \sE_i  ) \xrightarrow{\varphi_i } \sK_i \xrightarrow{\delta_i} \sE_i [1]
        \]
        we obtain a distinguished triangle
        \[
        \Hom^\bullet ( \sK_i , \sE_i  ) \to \Hom^\bullet ( \sK_i , \bT  \sE_i    ) \to \Hom^\bullet ( \sK_i , \sK_i  ) 
        \]
        in $\Db(\mathbb{k})$.
        Using Lemma \ref{lem:hom-computation} \eqref{eq:hom-computation-iso1} and \eqref{eq:hom-computation-iso2} for $i=j$ together with this distinguished triangle, we get an isomorphism of vector spaces
        \begin{align*}
\Hom  ( \sK_i  , \sK_i [p] ) \cong \begin{cases} \mathbb{k} , \quad & \text{if } p =   0 , d ,  \\ 
0 \quad & \text{else. }\end{cases}
\end{align*}
This shows that $\sK_i$ is $d$-spherical for all $1\leq i \leq m$.
    \end{proof}

    \begin{remark}\label{rmk:delta-commut}
    After a $\dg$-enhancement of the natural transformation $\eta \colon \id \to \bT$, we obtain a distinguished triangle of functors on $\Db(Y)$
    \[
    \id \xrightarrow{\eta} \bT \xrightarrow{\varphi} \bK \xrightarrow{\delta} [1] .
    \]
    Moreover, similar arguments as in the proof of Proposition \ref{prop:sph-obj} show that the image of the induced functor $\bK$ restricted to $\thick( \E) $ is contained in $\ker\pi_\ast$.
    Furthermore, we have the following commutative diagram if distinguished triangles for any morphism $e \colon \sE' \to \sE''$ is in $\thick ( \E )$
    \begin{equation*}\label{eq:E-TE-K}
\xymatrix{
    \sE'  \ar@{->}[d]^{e} \ar@{->}[r]^{\eta_{\sE'}} &  \bT \sE' \ar@{->}[d]^{\bT (e) } \ar@{->}[r]^{\varphi_{\sE'}} & \bK (\sE') \ar@{->}[d]^{\bK (e) } \ar@{->}[r]^{\delta_{\sE'}} & \sE [1] \ar@{->}[d]^{e[1]} \\
 \sE'' \ar@{->}[r]^{\eta_{\sE''}} & \bT \sE'' \ar@{->}[r]^{\varphi_{\sE''}} & \bK (\sE'') \ar@{->}[r]^{\delta_{\sE''}} & \sE'' [1] }.
\end{equation*}
\end{remark}

\begin{proposition}\label{prop:spherical}
Let $X$ be a $d$-dimensional projective Gorenstein variety with only isolated singularities and assume $X$ admits a  crepant resolution $\pi : Y \to X$.
\begin{enumerate}[label=(\alph*)]
    \item\label{prop:spherical-a} Let $\sA_1 , \ldots, \sA_l$ in $\ker \pi_\ast$ be a collection of $d$-spherical objects. 
    Set
    \[
       \bT_{\sA} : = \bT_{\sA_1} \circ \ldots \circ \bT_{\sA_l} ,
    \]
    where the functors $\bT_{\sA_i}$ denote the spherical twists induced by the spherical objects $\sA_i$, for all $i = 1 , \ldots , l$.
    Then $\pi_\ast$ is invariant under $\bT_{\sA}$.

\item\label{prop:spherical-b} Assume further that there is an exceptional collection $\sM = ( \sM_1 , \ldots , \sM_l )$ in $\Db(Y)$, such that there are unique (up to scalar) morphism $\epsilon_j \colon \sA_j \to \sM_j [1]$, such that there are isomorphisms of $\Hom$-spaces
\begin{align}
    \Hom^\bullet ( \sA_i , \sM_j ) & = \Hom^\bullet (\sM_i , \sM_j ) \cdot \epsilon_i [-1] , \quad \text{for all $i\leq j$}\label{eq:assump1} \\
    \Hom^\bullet ( \sA_j , \sM_i ) & = \epsilon_i \cdot \Hom^\bullet ( \sA_j , \sA_i ) , \quad \text{for all $j<i$} .\label{eq:assump2} 
\end{align}
        Then there are isomorphisms
        \begin{equation}\label{eq:j-bigger-i}
    \sM_i \cong \bT_{\sA_j} \sM_i ,
\end{equation}
        \begin{equation}\label{eq:j-lower-i}
    \bT_{\sA_i} \sM_i \cong \bT_{\sA_k} \bT_{\sA_i} \sM_i ,
\end{equation}
\begin{equation}\label{eq:j-equal-i}
    \bT_{\sA_i } \sM_i \cong \bT_{\sA} \sM_i
\end{equation}
        in $\Db(Y)$ for all $1\leq k<i<j \leq l $.
\end{enumerate}
\end{proposition}

\begin{proof}
    \ref{prop:spherical-a}: It follows from the definition of spherical twists that there are natural transformations 
    \[
    \eta_j \colon \id \to \bT_{ \sA_j} .
    \]
    We define the natural transformation $\eta_{\sA} \colon \id \to \bT_{\sA}$ as the composition of the transformations
\[
\id \xrightarrow{\eta_{1}} \bT_{ \sA_1} \xrightarrow{\bT_{ \sA_1}( \eta_2 )} \bT_{ \sA_1}  \bT_{ \sA_2} \xrightarrow{\bT_{ \sA_1} \bT_{ \sA_2} ( \eta_3 )} \ldots \to \bT_{\sA} .
\]
We show that $\pi_\ast (\eta_{\sA} ) : \pi_\ast \to \pi_\ast \bT_{\sA}$ is an isomorphism of functors.
    Indeed, by definition of the spherical twist $\bT_{ \sA_j}$, the cone of $\id \to \bT_{ \sA_j}$ evaluated at any object of $\Db(Y)$ is equal to a direct sum of shifts of $\sA_j$, hence this cone is contained in $\ker  \pi_\ast $ by assumption.
    In particular, evaluating an object of $\ker  \pi_\ast $ at $\bT_{\sA_j}$ is again contained in $\ker \pi_\ast $.
    It follows that the induced composition of natural transformations
    \[
\pi_\ast \xrightarrow[\cong]{\pi_\ast \eta_{1}} \pi_\ast \bT_{\sA_1} \xrightarrow[\cong]{\pi_\ast\bT_{ \sA_1}( \eta_2 )} \pi_\ast\bT_{ \sA_1}  \bT_{ \sA_2} \xrightarrow[\cong]{\pi_\ast \bT_{\sA_1} \bT_{\sA_2} ( \eta_3 )} \ldots \xrightarrow[\cong]{} \pi_\ast\bT_{\sA} 
\]
is a composition of isomorphisms of functors. This finishes \ref{prop:spherical-a}.
    
\ref{prop:spherical-b}: By construction of the spherical twist, the object $\bT_{{\sA_j}}\sM_i$ fits into the natural distinguished triangle
\[
\Hom^\bullet ( \sA_j , \sM_i ) \otimes \sA_j \to \sM_i \to \bT_{\sA_j} \sM_i .
\]
If $j>i$, we obtain by assumption \eqref{eq:assump1} that the first term of the distinguished triangle vanishes. 
In other words, we get an isomorphism $\sM_i \cong \bT_{\sA_j} \sM_i$ given by the natural morphism, which shows \eqref{eq:j-bigger-i}.
If $i=j$, then the only non-trivial morphism from $\sA_i  $ to $\sM_i [1]$ is (up to a scalar) equal to $\epsilon_i$ and hence by assumption \eqref{eq:assump1} the object $\bT_{{\sA_i}}\sM_i$ fits into the natural triangle
\begin{equation}\label{eq:main-in-proof}
    \sM_i \xrightarrow{\eta_i} \bT_{\sA_i} \sM_i \to \sA_i \xrightarrow{\epsilon_i }  \sM_i [1] .
\end{equation}
If $k<i$, we observe that the objects $\bT_{\sA_k} \bT_{\sA_i} \sM_i $ fit into the natural distinguished triangle 
\[
\Hom^\bullet ( \sA_k , \bT_{\sA_i} \sM_i ) \otimes \sA_j \to \bT_{\sA_i} \sM_i \to \bT_{\sA_k} \bT_{\sA_i} \sM_i .
\]
The first term of this distinguished triangle vanishes.
Indeed, applying $\Hom^\bullet (\sA_k , - )$ onto \eqref{eq:main-in-proof}, we see that assumption \eqref{eq:assump2} is equivalent to the vanishing of $\Hom^\bullet ( \sA_k , \bT_{\sA_i} \sM_i ) $.
Hence there is a natural isomorphism $\bT_{\sA_i} \sM_i \cong \bT_{\sA_k} \bT_{\sA_i} \sM_i$. This shows \eqref{eq:j-lower-i}.
Finally, combining the two isomorphisms \eqref{eq:j-bigger-i} and \eqref{eq:j-lower-i}, we obtain an isomorphism $\bT_{\sA_i } \sM_i \cong \bT_{\sA} \sM_i$ in $\Db(Y)$ for all $i$.
This finishes \ref{prop:spherical-b} and the proof of the proposition.
\end{proof}

We call $\sK_i$ the \emph{$d$-spherical object corresponding to $\sE_i$}.

\begin{corollary}\label{cor:T-vs-TK}
    If $\pi : Y  \to X$, $\bT : \Db(Y) \to \Db(Y) $ and $\E = ( \sE_1 , \ldots , \sE_m )$ are as in Proposition \ref{prop:sph-obj} and $\sK_1 , \ldots , \sK_m$ are  the $d$-spherical objects corresponding to $\sE_1 , \ldots  , \sE_m $ and if we set
    \[
       \bT_{\sK} : = \bT_{\sK_1} \circ \ldots \circ \bT_{\sK_m} ,
    \]
    then $\pi_\ast$ is invariant under $\bT_{\sK}$ and $\E$ and $\bT_{\sK}$ are pre-adherent.
\end{corollary}

\begin{proof}
    By Proposition \ref{prop:sph-obj} and Proposition \ref{prop:spherical} \ref{prop:spherical-a} we see  that $\pi_\ast$ is invariant under $\bT_{\sK}$.
It remains to see that $\E$ and $\bT_{\sK}$ are pre-adherent.
For this, we note that it is enough to show that there are isomorphisms 
\begin{equation}\label{eq:last-iso}
    \bT \sE_i \cong \bT_{\sK} \sE_i
\end{equation}
in $\Db(Y)$ for all $i$.
To prove this we observe first that by Lemma \ref{lem:hom-computation} \eqref{eq:hom-computation-iso1} the assumption \eqref{eq:assump1} of Proposition \eqref{prop:spherical} \ref{prop:spherical-b} is satisfied.
Moreover, applying $\Hom^\bullet ( \sK_j , -  )$ onto \eqref{eq:main-distinguished-triangle} and using Lemma \ref{lem:hom-computation} \eqref{eq:hom-computation-iso2} that also \eqref{eq:assump2} of Proposition \eqref{prop:spherical} \ref{prop:spherical-b} is satisfied.
It follows thus by Proposition \eqref{prop:spherical} \ref{prop:spherical-b} that there are isomorphisms $\bT_{\sK_i} \sE_i \cong \bT_{\sK} \sE_i $ in $\Db(Y)$ for all $i = 1 , \ldots , m$, see \eqref{eq:j-equal-i}.
Moreover, comparing the distinguished triangle 
\[
\sE_i \xrightarrow{\eta_i} \bT_{\sK_i} \sE_i \to \sK_i \xrightarrow{\delta_i} \sE_i [1]
\]
(see also \eqref{eq:main-in-proof}) to the distinguished triangle \eqref{eq:main-distinguished-triangle}, we see that there is an isomorphism $\bT_{\sK_i} \sE_i \cong \bT \sE_i$ in $\Db(Y)$ for all $i$.
Combining these two isomorphsim, we get an isomorphism \eqref{eq:last-iso} for all $i$.
This completes the proof of the corollary.
\end{proof}

\begin{remark}\label{rmk:T-vs-TK}
    Corollary \ref{cor:T-vs-TK} implies that we can replace the autoequivalence $\bT$ on $\Db(Y)$ by a composition of spherical twists $\bT_{\sK}$.
    We do not know if there is an isomorphism of functors between $\bT$ and $\bT_{\sK}$.
    Note here that in all our examples the corresponding $\bT$ is set to be a composition of spherical twists.
\end{remark}

\begin{definition}\label{def:strong-adherence}
    Let $X$ be a $d$-dimensional projective variety with a unique isolated Gorenstein singularity. Assume
    there is a crepant resolution of singularities
     $\pi \colon Y \to X$.
    Let $\bT : \Db(Y) \to \Db(Y) $ be an autoequivalence such that $\pi_\ast$ is invariant under $\bT$.
    Let $\E = ( \sE_1 , \ldots , \sE_m )$ be an exceptional collection in $\Db(Y)$, such that $\E$ and $\bT$ are  pre-adherent (cf. Definition \ref{def:adherence}).
    \begin{enumerate}
        \item\label{def:strong-adherence-1} 
        We say that $\E$ and $\bT$ are \emph{adherent}, if the spherical objects $\sK_i$ corresponding to the $\sE_i$ classically generate $\ker \pi_\ast$,
        i.e. the smallest thick triangulated subcategory containing $\sK_1 , \ldots , \sK_m$ is equal to $\ker\pi_\ast$.
        \item\label{def:strong-adherence-2} 
        Assume additionally that the exceptional locus $\iota \colon E \hookrightarrow Y$ is of codimension $1$.
        We say that $\E$ and $\bT$ are \emph{strongly adherent}, if $\E$ and $\bT$ are adherent  and if there is a full exceptional collection 
        \[
    \sL_1 , \ldots , \sL_m
\]
    of $\langle \OO_E \rangle^{\perp} \subset \Db(E)$, such that there are isomorphisms $\sK_i \cong \iota_\ast \sL_i$ in $\Db(Y)$ for all $i$.
    \end{enumerate}
\end{definition}

Note that the exceptional divisor 
$E$ as in Definition \ref{def:strong-adherence} \eqref{def:strong-adherence-2} has 
a priori locally hypersurface singularities and thus only Gorenstein singularities, as $E$ is an effective Cartier divisor inside a smooth variety $Y$.
Moreover, since $\Db(E)$ has a full exceptional sequence, it follows that $E$ is smooth (see  \cite[Corollary 4.15]{ks2}).
It follows that $\Db(E)$ has a Serre functor.

\begin{proposition}\label{prop:defs-agree}
    Le $X$ be a projective $d$-dimensional variety with a unique isolated Gorenstein singularity.
    Assume there is a crepant resolution $\pi \colon Y \to X $ and assume the exceptional locus $\iota \colon E \hookrightarrow Y$ is of codimension $1$.
    Let $\bT : \Db(Y) \to \Db(Y) $ be an autoequivalence such that $\pi_\ast$ is invariant under $\bT$.
    Assume there is an exceptional collection $\E$ in $\Db(Y)$, such that $\E$ and $\bT$ are strongly adherent.
    If the right adjoint $\iota^! \colon \Db(Y) \to \Db(E)$ of $\iota_\ast \colon \Db(E) \to \Db(Y)$ restricts to a functor $\iota^! \colon \thick (   \E   ) \to  \langle \OO_E \rangle^{\perp} $, then this functor is an equivalence
    \[
    \iota^! \colon \thick (  \E  ) \xrightarrow{\cong} \langle \OO_E \rangle^{\perp} .
    \]
\end{proposition}
\begin{proof}
We use the same notation as in Definition \ref{def:strong-adherence}.
    It is enough to show that there are isomorphisms $\iota^!  \sE_i [1] \cong \sL_i$ in $\langle \OO_E \rangle^\perp$, for all $i$ and that for all $i$ and $j$ we have isomorphisms $\Hom^\bullet (   \sE_j   ,   \sE_i   ) \cong \Hom^\bullet ( \sL_j , \sL_i )$ given by $\iota^! (-) [1]$.

    Let us first show that there are isomorphisms $\iota^!  \sE_i \cong \sL_i$ for all $i$.
    By assumption, we have isomorphisms $\sK_i \cong \iota_\ast \sL_i$.
    It further follows from \eqref{eq:hom-computation-iso1} in the case $i=j$ that there are isomorphisms
    \begin{equation}\label{eq:restr-hom}
            \Hom^\bullet ( \sL_i , \iota^!  \sE_i  ) \cong \Hom^\bullet (   \sK_i , \sE_i  ) = \mathbb{k}\cdot \delta_i [-1].
    \end{equation}
    In other words, there is a non-trivial morphism $\phi_i \colon \sL_i \to  \iota^! \sE_i [1] $ corresponding to $\delta_i$ via adjunction (i.e. $\ad ( \delta_i ) = \phi_i $).
    The distinguished triangle 
    \[
    \sL_i \xrightarrow{\phi_i } \iota^! \sE_i [1] \to C ( \phi_i ) ,
\]
shows that $C ( \phi_i )$ is contained in $\langle \OO_E \rangle^{\perp}$ for all $i$, as the $\sL_i$ and $\iota^!  \sE_i$ are contained in $\langle \OO_E\rangle^{\perp}$ by assumption.
Applying $\Hom^\bullet (  \sL_j , - )$ onto this triangle, we obtain a distinguished triangle
\begin{equation}\label{eq:cone-zero}
  \Hom^\bullet (  \sL_j , \sL_i    )\xrightarrow{ \phi_i \cdot (-)}  \Hom^\bullet (  \sL_j , \iota^!\sE_i [1] ) \to \Hom^\bullet ( \sL_j , C ( \phi_i  )   )   
\end{equation}
in $\Db(\mathbb{k})$.
We are going to show that $\Hom^\bullet (  \sL_j , C ( \phi_i  ) )$ vanishes for all $j$.
If $j> i$ then the first term of this triangle vanishes by the assumption that $\L = ( \sL_1 , \ldots , \sL_m )$ is an exceptional sequence and the second  term vanishes by adjunction and by Lemma \ref{lem:hom-computation} \eqref{eq:hom-computation-iso1}.
Moreover, if $i=j$  then the morphism $(-) \cdot \phi_i $ in \eqref{eq:cone-zero} is an isomorphism by \eqref{eq:restr-hom} and since $\sL_i$ is an exceptional object.
Hence, we have that $\Hom^\bullet ( C ( \phi_i ) , \sL_j )$ vanishes if $j \geq i$.
Assume now that $j< i$ holds. 
Here we need some preparatory work.
We first note that by applying $\Hom^\bullet ( \sK_j , - )$ onto \eqref{eq:main-distinguished-triangle} and using Lemma \ref{lem:hom-computation} \eqref{eq:hom-computation-iso2}, we obtain isomorphisms
\begin{equation}\label{eq:Ki-to-Kj}
    \Hom^\bullet ( \sK_j , \sK_i  ) \xrightarrow[\cong]{ \delta_i \cdot (-)} \Hom^\bullet ( \sK_j  ,  \sE_i [1] )
\end{equation} 
for all $j<i$.
Second,
we have a distinguished triangle 
\[
\sL_i \to \iota^! \iota_\ast \sL_i \to \sL_i \otimes \OO_E (E) [-1] 
\]
Applying $\Hom^\bullet ( \sL_j , - )$ onto this triangle, we get a triangle 
\[
\Hom^\bullet ( \sL_j , \sL_i ) \to \Hom^\bullet ( \sL_j , \iota^! \iota_\ast \sL_i ) \cong \Hom^\bullet ( \sK_j , \sK_i  ) \to \Hom^\bullet ( \sL_j , \sL_i \otimes \omega_E [-1] )
\]
in $\Db(\mathbb{k})$, where we used adjunction in the second component and the isomorphism of line bundles $\OO_E(E) \cong \omega_E$ in the third component (adjunction formula for canonical bundles, together with crepancy of $\pi$).
Using Serre duality on $E$ and the assumption that $\L$ is an exceptional collection, we have that
\[
 \Hom^\bullet_E ( \sL_j , \sL_i \otimes \omega_E ) \cong \Hom^\bullet_E ( \sL_i , \sL_j [d-1] )^\ast = 0 
\]
holds. 
Hence the third component of this triangle vanishes and we have that the isomorphism a chain of isomorphisms
\begin{equation}\label{eq:K-L-iso}
    \Hom^\bullet_E ( \sL_j ,  \sL_i ) \cong \Hom^\bullet_E (  \sL_j , \iota^! \iota_\ast \sL_i ) \cong \Hom^\bullet ( \sK_j , \sK_i  ) ,
\end{equation}
 where the first isomorphism is given by the composition of the morphism $\ad_{\sL_i} \colon \sL_i \to  \iota^! \iota_\ast\sL_i$ corresponding to  the identity $\iota_\ast \sL_i \to \iota_\ast \sL_i$ via adjunction.
Let us now show that the third term in \eqref{eq:cone-zero} vanishes for $j<i$.
By general properties of adjoint functors, the morphism $\phi_i = \ad (\delta_i )  $  is given as the composition of morphisms $\phi_ i = \iota^! \delta_i \cdot \ad_{\sL_i } $.  
Moreover, by adjointness of $\iota^!$ and $\iota_\ast$, there is a commutative diagram
\begin{equation}
\xymatrix{
    \Hom^\bullet (  \sL_j , \iota^! \iota_\ast \sL_i ) \ar@{<-}[d]^{\ad(-)}_\cong \ar@{->}[rr]^{ \iota^! \delta_i \cdot (-) } & & \Hom^\bullet ( \sL_j , \iota^! \sE_i [1] ) \ar@{<-}[d]^{\ad(-)}_{\cong} \\
 \Hom^\bullet ( \sK_j , \sK_i ) \ar@{->}[rr]^{\delta_i \cdot (-)}_{\cong} & & \Hom^\bullet ( \sK_j , \sE_i [1] )},
\end{equation}
and  the bottom horizontal morphism in this diagram is an isomorphism by \eqref{eq:Ki-to-Kj}.
It follows that the top vertical morphism $  \iota^! \delta_i  \cdot (-)$ between graded $\Hom$-spaces is an isomorphism.
Furthermore,  $ \ad_{\sL_i } \cdot (-) : \Hom^\bullet ( \sL_j , \sL_i  ) \to \Hom^\bullet (  \sL_j , \iota^! \iota_\ast\sL_i )$ is an isomorphism by \eqref{eq:K-L-iso}.
Thus by the equality of morphisms $\phi_ i =   \iota^! \delta_i \cdot \ad_{\sL_i }$, it follows that 
$ \phi_i \cdot (- ) $ in \eqref{eq:cone-zero}
is an isomorphism.
Summarizing, we have that $\Hom^\bullet (  \sL_j , C ( \phi_i )  ) = 0$, for all $j$.
Since $\bigoplus  \sL_j $ is a generator of 
$\langle\OO_E \rangle^\perp$ and since $C ( \phi_i )$ is contained in $\langle\OO_E \rangle^\perp$, we obtain that $C ( \phi_i ) = 0$ for all $i$.
This finishes the proof that $\phi_i \colon \sL_i \to \iota^! \sE_i$ is an isomorphism for all $i$.

Finally, let us show that the isomorphism $\Hom^\bullet ( \sE_j , \sE_i  ) \cong \Hom^\bullet (\sL_j , \sL_i )$ from above is given by $\iota^! [1]$. 
Let us consider the composition of morphisms
\begin{equation}\label{eq:morphism-311}
 \iota_\ast \iota^! \sE_j [1] \xrightarrow[\cong]{\iota_\ast \phi_i^{-1}}   \iota_\ast \sL_j \xrightarrow{\delta_j} \sE_j [1] .
\end{equation}
By Lemma \ref{lem:hom-computation} \eqref{eq:hom-computation-iso1} we get the following chain of isomorphisms
\begin{equation}\label{eq:i-shreek-iso}
    \Hom^\bullet (  \sE_j ,  \sE_i ) \xrightarrow[{\cong}]{(-) \cdot ( \delta_j \cdot \iota_\ast\phi_j^{-1} ) } \Hom^\bullet ( \iota_\ast \iota^!  \sE_j  ,  \sE_i ) \xrightarrow[\cong]{\ad} \Hom^\bullet ( \iota^!  \sE_j  , \iota^!  \sE_i ) .
\end{equation}
Moreover, since there is only one non-zero morphism from  $\iota_\ast \iota^! \sE_j [1] \cong \iota_\ast \sL_j$ to $\sE_j [1]$ up to scalar (see \eqref{eq:restr-hom}),
we see that the morphism in \eqref{eq:morphism-311} agrees (up to a scalar) with $\ad^{-1} ( \id_{\iota^!\sE_j} ) \colon \iota_\ast \iota^! \sE_j \to \sE_j$, the morphism corresponding to the identity morphism of $\iota^!\sE_j $ via adjunction.
Thus, for any morphism $e$ in $\Hom^\bullet ( \sE_j , \sE_i )$ we have the following chain of equalities (after possibly rescaling):
\[
 \ad ( e \cdot ( \delta_j \cdot \iota_\ast\phi_j^{-1} ) ) =  \ad ( e \cdot \ad^{-1} ( \id_{\iota^!\sE_j} ) ) = \iota^!(e) \cdot \ad (  \ad^{-1} ( \id_{\iota^!\sE_j} ) ) = \iota^!(e) \in \Hom^\bullet ( \iota^! \sE_j , \iota^! \sE_i ) ,
\]
where we used functoriality of adjunction in the second equality.
This shows that the isomorphism of $\Hom$-spaces in \eqref{eq:i-shreek-iso} agrees with $\iota^!$ and finishes the proof.
\end{proof}

\section{Adherence implies categorical absorption}

\noindent
In this chapter, $\mathbb{k}$ is an arbitrary field.
We first give a compatibility result for
semiorthogonal decompositions and Verdier localizations.

\begin{proposition}\label{lem:sod-and-quotient}
Let $\BB \subset \CC \subset \TT$ be triangulated categories, $\BB$ a thick triangulated subcategory of $\CC$ and $\CC$ a thick triangulated category of $\TT$. Assume that there is a semiorthogonal decomposition
\[
\TT = \langle \CC ,  {^{\perp}\CC} \rangle .
\]
Then $\CC / \BB$ and $^{\perp}\CC$ are full triangulated subcategories of $\TT / \BB$ and there is a semiorthogonal decomposition
\[
\TT / \BB = \langle \CC / \BB ,  \bQ ( {^{\perp}\CC} ) \rangle , 
\]
where $\bQ \colon \TT \to \TT / \BB$ is the quotient functor.
Moreover, the restriction functor $\ {^{\perp}\CC} \to \bQ ( {^{\perp}\CC} )$ is an equivalence.
\end{proposition}

\begin{proof}
We know by \cite[Corollary 4.3]{verdier} that $\CC / \BB $ is a full (and thick) triangulated subcategory of $\TT / \BB$. 
Next we are going to show that $^{\perp}\CC$ is a full triangulated subcategory of $^{\perp}(\CC / \BB) \subset \TT / \BB$ via the quotient functor $\bQ : \TT \to \TT /  \BB$. 
To do this, we are going to show that the quotient functor $\bQ : \TT \to \TT / \BB$ restricts to a functor $\bar{\bQ} : {^{\perp}\CC} \to {^{\perp}(\CC / \BB)}$ and that $\bar{\bQ}$ is fully faithful.
We observe the following. 
For any $\sE$ in $^{\perp}\CC$, we have that $\Hom_{\TT}( \sE , \BB ) = 0$, as $\BB \subset \CC$. 
By \cite[Proposition 5.3]{verdier}, this shows that there is an isomorphism 
\begin{equation}\label{eq:verdier-trick}
  \Hom_{\TT} ( \sE , \sF ) \xrightarrow{\cong} \Hom_{\TT / \BB }( \bQ \sE  , \bQ \sF  )  
\end{equation}
for all $\sF$ in $\TT$, given by the canonical morphism induced by $\bQ$. 
Now, for any object $\bar{\sA}$ in $\CC  /\BB$, we can find an object $\sA $ in $\CC$ such that $\bQ \sA  = \bar{\sA}$ in  $\CC / \BB$. 
Then setting $\sF = \sA$ in \eqref{eq:verdier-trick}, we get 
\[
  \Hom_{\TT / \BB}( \bQ \sE  , \bar{\sA} ) = \Hom_{\TT / \BB} (\bQ \sE  , \bQ \sA  ) \cong \Hom_{\TT} ( \sE , \sA ) = 0 ,
\]
for all $\sE$ in $^{\perp}\CC$.
In other words, $\bQ( \sE)$ is contained in $^{\perp}( \CC / \BB )$ and hence the functor $\bQ$ restricts to a functor $\bar{\bQ} : {^{\perp}\CC} \to {^{\perp}(\CC / \BB)}$.
On the other hand, if $\sE'$ is another object contained in $^{\perp}\CC$, then set $\sF = \sE'$ in \eqref{eq:verdier-trick} and we obtain an isomorphism of $\Hom$-spaces
\[
\Hom_{^{\perp}\CC} ( \sE , \sE' ) \cong \Hom_{^{\perp} ( \CC / \BB )}( \bar{\bQ} \sE  , \bar{\bQ} \sE'  ) .
\]
Equivalently, the functor $\bar{\bQ} : {^{\perp}\CC} \to {^{\perp}(\CC / \BB)}$ is fully faithful. 
This concludes that $\bQ ( {^{\perp}\CC})$ is a full subcategory of $^{\perp}(\CC / \BB)$.

To prove the decomposition $\TT / \BB = \langle \CC / \BB , {^{\perp}\CC} \rangle$ it remains to show that $\CC / \BB$ and $^{\perp}\CC$ generate $\TT / \BB$. 
But this follows just by essentially surjectivity of $\bQ : \TT \to \TT / \BB$. 
More concretely, let $ \bar{\sF} \in \TT / \BB$ and let $\sF$ be an object of $\TT$ such that $\bQ( \sF ) = \bar{\sF}$. 
By the decomposition $\TT = \langle \CC , {^{\perp}\CC} \rangle$, we can find a distinguished triangle 
\[
\sE \to \sF \to \sG 
\]
in $\TT$ with $\sG$ in $\CC$ and $\sE$ in $^{\perp}\CC$. 
Applying $\bQ$ to this triangle, we get that $\bar{\sF}$ is generated by $\bQ \sG \in \CC / \BB$ and by $\bQ( \sE )$, which by the above argument is an object in $^{\perp}\CC \subset \TT / \BB$. 
We conclude the decomposition $\TT / \BB = \langle \CC / \BB , {^{\perp}\CC} \rangle$.
Finally, we recall that $\bar{\bQ} : {^{\perp}\CC} \to {^{\perp}(\CC / \BB)}$ is fully faithful, and by the above we have that ${^{\perp}(\CC / \BB)} $ and $\bQ ( {^{\perp}\CC} )$ coincide.
In particular, $\bar{\bQ} : {^{\perp}\CC} \to {^{\perp}(\CC / \BB)}$ is an equivalence.
This finishes the proof of the proposition.
\end{proof}

Proposition \ref{lem:sod-and-quotient} and Theorem \ref{thm:bo-localization-and-kernel} yields the following.

\begin{corollary}\label{cor:sod-quotient}
Let $X$ be projective variety over $\mathbb{k}$ with a unique Gorenstein singularity at $p\in X$ and assume that a crepant resolution $\pi\colon Y \to X$ exists. 
Let $\bT \colon \Db( Y ) \to \Db ( Y )$ be an autoequivalence which is invariant under $\pi_\ast$.
Assume there is an exceptional sequence $\E = ( \sE_1, \ldots, \sE_{m} )$, such that $\E$ and $\bT$ are pre-adherent and set $\tilde{\AA}=\langle \E, \bT ( \E ) \rangle$.
Then there is a semiorthogonal decomposition
\begin{equation}\label{eq:sod-quot1}
    \Db (Y) /\ker \pi_\ast = \langle \tilde{\AA} / \ker \pi_* ,  {^{\perp}\tilde{\AA}} \rangle  .
\end{equation}
Moreover, if there is an equivalence $\Db (Y) /\ker \pi_\ast \cong \Db(X)$ induced by the functor $\pi_\ast \colon \Db(Y) \to \Db(X)$, then the decomposition
\begin{equation}\label{eq:sod-quot2}
    \Db (X) = \langle \tilde{\AA} / \ker \pi_* ,  \pi_\ast {^{\perp}\tilde{\AA}} \rangle 
\end{equation}
is admissible and ${^{\perp}\tilde{\AA}} \cong \pi_\ast {^{\perp}\tilde{\AA}}$
is contained in $ \Dperf(X)$.
\end{corollary}

\begin{proof}
A direct application of Proposition \ref{lem:sod-and-quotient} to the chain of inclusions $\ker \pi_\ast \subset \tilde{\AA} \subset \Db (Y)$ yields \eqref{eq:sod-quot1}.   
To see that $  \pi_\ast{^{\perp}\tilde{\AA}}$ in \eqref{eq:sod-quot2} is contained in $  \Dperf(X)$, we note first that 
there is an equality of $\Hom$-spaces
\[
\Hom ( \sP ' , \ker \pi_\ast ) = 0
\]
for all $\sP '$ in $^{\perp}\tilde{\AA}$, because $\ker \pi_\ast $ is contained in $\tilde{\AA}$ (as $\bT$ and $\E$ are pre-adherent).
Thus by \cite[Proposition 5.3]{verdier} there is an isomorphism
\begin{equation}\label{eq:verdier-trick2}
    \Hom_{\Db(Y)} (\sP' , \sB' ) \cong \Hom_{\Db(X) } ( \pi_\ast \sP'  , \pi_\ast \sB'  ) 
\end{equation}
for all $\sP '$ in $^{\perp}\tilde{\AA}$ and all $\sB '$ in $\Db (Y)$ (see also \eqref{eq:verdier-trick}). 
Using \eqref{eq:verdier-trick2}, we get the following chain of isomorphisms
\[
\Hom_{\Db(X)} ( \pi_\ast \sP' , \sB [N] ) \cong \Hom_{\Db(X)} ( \pi_\ast \sP'  , \pi_\ast  \sB'  [N] ) \cong \Hom_{\Db(Y)}(\sP' , \sB' [N] ) = 0 
\]
for $|N| $ big enough, where we use in the first isomorphism that $\pi_\ast$ is essentially surjective by assumption and thus that there exist a $\sB'$ in $\Db(Y)$ with  $\pi_\ast  \sB' \cong  \sB$.
Moreover, in the third isomorphism we use that $\Db (Y)$ is of finite type.
We have thus shown that any object in $\pi_\ast  {^\perp}\tilde{\AA} $ is homologically finite in $\Db(X)$ (see \cite[Definition 1.6]{orlov-hf} for the definition) and hence by \cite[Proposition 1.11]{orlov-hf} we conclude that $\pi_\ast  {^\perp}\tilde{\AA} $ is contained in $\Dperf (X)$.
Finally, admissibility follows from e.g. \cite[Lemma 2.15]{KPS}.
This finishes the proof of the corollary.
\end{proof}

Let now $\E$ and $\bT$ be adherent to each other.
Under certain conditions on $\E$, we will describe $\tilde{\AA} / \ker\pi_\ast$ as $\DD_{\mathrm{fg}} ( A^\bullet )$, for some $\dg$-algebra $A^\bullet$. 
To achieve this, we perform the following steps:

\begin{itemize}
    \item We start with the perfect generator $\sE \oplus \bT  \sE $ of $\tilde{\AA}$, where $\sE = \bigoplus_i \sE_i$, and we perform a mutation of the exceptional collection $\bT (\E )$, denoted $\F$.
    \item Under the assumption that there is an equality of $\Hom$-spaces 
    \[
\Hom ( \sE_i , \bS^{-1}_\E \sE_i [d] ) = 0
\]
    for all $i$,
    we obtain morphisms $\nu_i \colon \sE_i \to \sF_i [d-1]$, and we call $\sG_i$ the cocone of this morphism by $\sG_i$.
     Here, $\bS_\E$ denotes the Serre functor on $\thick ( \E )$.
    \item We obtain thus a different perfect generator $\sE \oplus \sG$  of $\tilde{\AA}$.
    In particular, 
    there is an equivalence 
    \[
    \tilde{\AA} \cong \DD_{\mathrm{fg}} ( \tilde{A}^\bullet ) ,
    \]
    where $\tilde{A}^\bullet$ is a $\dg$-algebra with $p$-th cohomology isomorphic to $\Hom ( \sT , \sT [p] )$, where $\sT = \bigoplus_i \sG_i \oplus \bigoplus_j \sF_j [d-2]$.
    \item We show that the $\Hom$-spaces $\Hom ( \sG_i , \sG_j [p] )$ fit into a long exact sequence 
    \[
    \ldots \to \Hom ( \sE_i , \bS^{-1}_\sE \sE_j [p+d-2] ) \to \Hom ( \sG_i , \sG_j [p] ) \to \Hom ( \sE_i , \sE_j [p] ) \to \ldots
    \]
    for all $i , j$ and $p$.
    \item If $\E$ is geometric of type $(m,k)$ with $k\leq d-1$, we obtain an induced equivalence $\tilde{\AA} / \ker\pi_\ast \cong \DD_{\mathrm{fg}} ( A^\bullet  )$, where $A^\bullet $ is a $\dg$-algebra with $p$-th cohomology isomorphic to $\Hom ( \sG , \sG [p] )$ and concentrated in degree $0$ and in (negative) degree $k-d+1$.
\end{itemize}

\subsection{Generators of \texorpdfstring{$\tilde{\AA}$}{ A\textasciicircum tilde}}
We use the notation from the previous section.
That is, let $X$ be a projective $d$-dimensional variety over $\mathbb{k}$ with a unique isolated Gorenstein singularity at $p \in X$ and assume there exists a crepant resolution $\pi \colon Y \to X$ of $X$.
Moreover, let $\bT$ be an autoequivalence of $\Db(Y)$, such that $\pi_\ast$ is invariant under $\bT $.
Finally, assume there is an exceptional collection $\E = (  \sE_1 , \ldots , \sE_m )$ in $\Db(Y)$, such that $\E$ and $\bT$ are pre-adherent and denote $\tilde{\AA} = \langle \E , \bT ( \E ) \rangle$.

\begin{lemma}\label{lem:hom-K-F}
    Let $\E$ and $\bT$ be as above. 
    Let $\sK_i$ be the $d$-spherical object corresponding to $\sE_i$, for all $1\leq i \leq m$.
    Set
    \[
    \sF_i := \bS^{-1}_{\bT\E}  \bT ( \sE_i ) \in \thick \bT( \E ) ,
    \]
    where $\bS_{\bT\E}$ denotes the Serre functor on $\thick\bT ( \E )$.
    Then there are unique (up to a scalar) non-trivial morphisms $\xi_i \colon \sK_i \to \sF_i [d]$ in $\Db(Y)$ satisfying:
       \begin{enumerate}
          \item\label{lem:hom-K-F-a} There is an isomorphism of $\Hom$-spaces
        \[
        \Hom^\bullet ( \sK_i , \sF_j ) = \Hom^\bullet ( \sF_i , \sF_j  ) \cdot \xi_i 
        \]
        for all $1 \leq i , j \leq m$.
        \item\label{lem:hom-K-F-12} For all $1\leq i , j \leq m$ and all morphisms $e$ in $\Hom^\bullet  ( \sE_i , \sE_j )$ the equality
        \[
        \xi_j \cdot \bK ( e ) = \bS^{-1}_{\bT\E} \bT (e) [d] \cdot \xi_i   
        \]
        in $\Hom^\bullet ( \sK_i , \sF_j [d] )$
        is satisfied.
        Here, $\bK (e) \in \Hom^\bullet ( \sK_i , \sK_j  )$ is the induced morphism as defined in Remark \ref{rmk:delta-commut}.
    \end{enumerate}
\end{lemma}

\begin{proof}
    We start by constructing non-trivial morphisms $\xi_{i} \colon \sK_i \to \sF_i [d]$, for all $i$.
We note first that $\sF_j $ is contained in $\thick \bT ( \E)$.
In particular, as $\E$ and $\bT$ are pre-adherent, we see that the vanishing
\begin{equation}\label{eq:vanishing}
    \Hom^\bullet ( \sF_j , \sE_i ) =  0 
\end{equation}
holds for all $i$ and $j$ (as $\sE_i$  is in $\thick ( \E)$).
We recall the distinguished triangle 
\begin{equation}\label{eq:main-distinguished-triangle-proof}
    \sE_i \xrightarrow{\eta_i } \bT  \sE_i   \xrightarrow{\varphi_i } \sK_i ,
\end{equation}
see \eqref{eq:main-distinguished-triangle}.
We apply $\Hom^\bullet ( \sF_j , - )$ onto \eqref{eq:main-distinguished-triangle-proof} and we obtain a distinguished triangle
\[
\Hom^\bullet ( \sF_j  , \sE_i  ) \xrightarrow{\eta_i \cdot (-)} \Hom^\bullet ( \sF_j , \bT  \sE_i    ) \xrightarrow{\varphi_i \cdot (-)} \Hom^\bullet ( \sF_j , \sK_i ) 
\]
in $\Db (\mathbb{k})$.
The first term vanishes, see \eqref{eq:vanishing} above.
We remain with an isomorphism
\begin{equation}\label{eq:hom_F_and_K}
    \Hom^\bullet ( \sF_j , \bT  \sE_i    ) \xrightarrow[\cong]{\varphi_i \cdot (-)} \Hom^\bullet ( \sF_j , \sK_i  ) .
\end{equation}
Furthermore, we have the following chain of natural isomorphisms
\[
\Hom^\bullet ( \sF_j , \bT  \sE_i     ) = \Hom^\bullet ( \bS^{-1}_{\bT\E}\bT  \sE_j   , \bT \sE_i   ) \cong \Hom^\bullet (  \bT  \sE_i  , \bT  \sE_j   )^\ast \cong \Hom^\bullet ( \sF_i , \sF_j  )^* 
\]
In the second equality, we used Serre duality on $\thick ( \bT ( \sE ) )$.
Combining this chain of isomorphism with \eqref{eq:hom_F_and_K}, we obtain
\begin{equation}\label{eq:hom_K_and_F}
    \Hom^\bullet ( \sK_i , \sF_j   ) \cong \Hom^\bullet ( \sF_j    , \sK_i [d] )^\ast \cong \Hom^\bullet ( \sF_i , \sF_j ) [-d]  . 
\end{equation}
Note that in the first isomorphism we used Serre duality on $\Db (Y)$ and the fact that $\sK_i$ is $d$-spherical by Proposition \ref{prop:sph-obj}.
In particular, for $i=j$ there is a unique non-trivial morphism $\xi_i \colon \sK_i \to \sF_i  [d]$ corresponding to the identity morphism $\id \colon \sF_i \to \sF_i $ via the above isomorphism.
In the following, we use isomorphism \eqref{eq:hom_K_and_F} to show  \eqref{lem:hom-K-F-a} and \eqref{lem:hom-K-F-12}.

\eqref{lem:hom-K-F-a}:
This follows by the bi-functoriality of the isomorphism \eqref{eq:hom_K_and_F}.
To show this, let us first recall what we mean by bi-functoriality.
Let $\DD$ be a $\Hom$-finite triangulated category with a Serre functor $\bS_\DD$, and denote by $\phi^{\DD} \colon \Hom ( \sB ,  \bS_\DD \sA  ) \xrightarrow{\cong} \Hom ( \sA ,   \sB  )^\ast$ the corresponding isomorphism.
Then for morphism $f \colon \sA \to \sC$ and $g\colon \sD \to \sB$ we have that the following two commutative  diagrams
\begin{equation}\label{eq:Serre-diag}
\xymatrix{
    \Hom (  \sB , \bS_\DD \sA ) \ar@{->}[d]^{\phi^{\DD}}_\cong \ar@{->}[rr]^{ \bS_\DD ( f ) \cdot (-) } & & \Hom ( \sB , \bS_\DD \sC ) \ar@{->}[d]^{\phi^{\DD}}_{\cong} \\
 \Hom ( \sA,   \sB  )^\ast \ar@{->}[rr]^{(\ast \cdot f )^\ast } & & \Hom ( \sC ,   \sB   )^\ast },
 \xymatrix{
    \Hom (  \sB , \bS_\DD \sA ) \ar@{->}[d]^{\phi^{\DD}}_\cong \ar@{->}[rr]^{ (-) \cdot g } & & \Hom ( \sD , \bS_\DD \sA ) \ar@{->}[d]^{\phi^{\DD}}_{\cong} \\
 \Hom ( \sA ,   \sB  )^\ast \ar@{->}[rr]^{( g \cdot \ast )^\ast } & & \Hom ( \sA ,   \sD   )^\ast },
\end{equation}
where $( \ast \cdot f )^\ast$ and $(g \cdot \ast )^\ast$ are the dual moprhisms of the morphisms of $\Hom$-spaces $(-) \cdot f \colon \Hom ( \sC , \sB ) \to \Hom ( \sA ,  \sB )$ and $g \cdot (-) \colon \Hom ( \sA ,  \sD ) \to \Hom ( \sA ,  \sB )$, respectively.

Now, for any morphism $f \colon \sF_i \to \sF_j [m]$, the commutativity of the second diagram in \eqref{eq:Serre-diag} with $\DD = \Db(Y)$, $ \sB = \bS_Y \sF_i $, $\sD = \sK_i \cong \bS_Y \sK_i [-d]$,  and $\sA = \sF_j [m]$
implies  that there is the following equality
\begin{equation}\label{eq:Serre-equality-1}
    \phi^Y ( \bS_Y ( f  \cdot \xi_i [-d] ) ) = ( \bS_Y (\xi_i [-d] ) \cdot \ast  )^\ast \phi^Y( \bS_Y (f ) ) \in \Hom ( \sF_j [m]  , \sK_i  )^\ast .
\end{equation}
Note that we omit to write out the identification $\bS_Y^{-1}\sK_i [d] \cong \sK_i$ equality above.
By abuse of notation, we will continue to omit this isomorphism throughout the proof.
Anyway, applying the isomorphism \eqref{eq:hom_F_and_K} to the previous equality, we obtain the equality
\begin{equation}\label{eq:Serre-equality-2}
    (\varphi_i \cdot \ast )^\ast \cdot \phi^Y ( \bS_Y (f \cdot \xi_i [-d] ) ) = ( \bS_Y (\xi_i [-d] ) \cdot \varphi_i  \cdot \ast  )^\ast \phi^Y( \bS_Y (  f  ) ) \in \Hom ( \sF_j [m]  , \bT \sE_i  )^\ast .
\end{equation}
In particular, for $i = j$ and  $f= \id_{\sF_i }$ we get the equality
\begin{equation}\label{eq:Serre-equality-23}
    (\varphi_i \cdot \ast )^\ast \cdot \phi^Y ( \bS_Y (  \xi_i [-d]) ) = ( \bS_Y (\xi_i [-d]) \cdot \varphi_i  \cdot \ast  )^\ast \phi^Y( \id_{\bS_Y(\sF_i ) } ) \in \Hom ( \sF_i    , \bT \sE_i  )^\ast .
\end{equation}
Similarly, 
by commutativity of the first diagram in \eqref{eq:Serre-diag} with $\DD = \thick \bT \E$, $\sB = \bT \sE_i$, $\sA = \bT\sE_i \cong  \bS_{\bT\E}\sF_i$, and $\sC = \sF_j [m]$,
we obtain the equality
\begin{equation}\label{eq:Serre-equality-3}
    \phi^{\bT\E} ( \bS_{\bT\E}(f)) = \phi^{\bT\E} ( \bS_{\bT\E}(f) \cdot \id_{\bT\sE_i} ) = ( \ast \cdot f)^\ast \phi^{\bT\E} ( \id_{\bT\sE_i} ) \in \Hom ( \sF_j [m]  , \bT\sE_i )^\ast .
\end{equation}
By definition of $\xi_i$, we have that the equality
\begin{equation}\label{eq:id}
    \phi^{\bT\E}(\id_{\bT\sE_i}) = (\varphi_i \cdot \ast )^\ast  \phi^Y (  \bS_Y ( \xi_i [-d] ) ) \in \Hom ( \sF_i   , \bT\sE_i )^\ast .
\end{equation}
Inserting this equality in \eqref{eq:Serre-equality-3} and using the equality \eqref{eq:Serre-equality-23} we get
\begin{align*}
    \phi^{\bT\E} ( \bS_{\bT\E}(f)  ) & = (\ast \cdot f)^\ast (\varphi_i \cdot \ast )^\ast \cdot \phi^Y ( \bS_Y ( \xi_i [-d] ) ) \\
    & = ((\ast \cdot f)^\ast ( \bS_Y (\xi_i [-d]) \cdot \varphi_i \cdot \ast  )^\ast \phi^Y( \id_{\bS_Y ( \sF_i ) } )  \\
    & =  ( \bS_Y (\xi_i [-d] ) \cdot \varphi_i \cdot \ast  )^\ast ( \ast \cdot f)^\ast \phi^Y( \id_{\bS_Y ( \sF_i ) } ) \\
    & = ( \bS_Y (\xi_i [-d]) \cdot \varphi_i \cdot \ast  )^\ast \phi^Y( \bS_Y ( f ) ) 
\end{align*}
in $ \Hom ( \sF_j [m]  , \bT\sE_i )^\ast $, where we used commutativity of the first diagram in \ref{eq:Serre-diag} with $\DD = \Db(Y)$, $\sB= \bS_Y \sF_i $, $\sA = \sF_i$ and $\sC = \sF_j [m]$.
Combining this equality with \eqref{eq:Serre-equality-2}, we get 
\[
\phi^{\bT\E} ( \bS_{\bT\E}(f)  ) 
= ( \bS_Y (\xi_i [-d]) \cdot \varphi_i \cdot \ast  )^\ast \phi^Y( \bS_Y (f ) ) 
= (\varphi_i \cdot \ast )^\ast \cdot \phi^Y ( \bS_Y  ( f \cdot \xi_i [-d] ) ) \in  \Hom ( \sF_j  , \bT\sE_i )^\ast,
\]
which shows that the isomorphism \eqref{eq:hom_K_and_F} sends $f \cdot \xi_i [-d]$ to $f$.
This finishes the proof of \eqref{lem:hom-K-F-a}.

\eqref{lem:hom-K-F-12}:
To show \eqref{lem:hom-K-F-12}, 
we note first that  there is an isomorphism of functors $\sigma \colon \id \cong   \bS_Y [-d]$ on $\ker \pi_\ast$ by Lemma \ref{lem:dCY-kernel}.
In particular,  the commutativity $\sigma_j \cdot \bK (e) = \bS_Y \bK(e) [-d] \cdot \sigma_i$ holds, where $\sigma_j = \sigma_{\sK_j }$ and $\sigma_i = \sigma_{\sK_i }$.
Similarly as in the proof of \eqref{lem:hom-K-F-a}, we will omit writing out the transformation $\sigma$, by abuse of notation.
Next, we use that the second diagram in \ref{eq:Serre-diag} with $\DD = \Db(Y)$, $\sB = \sK_j $, $\sA = \sF_j $ and $\sD = \sK_i$ 
to obtain the equality
\[
\phi^Y ( \bS_Y ( \xi_j \cdot \bK (e ) [-d] )   ) = \phi^Y ( \bS_Y ( \xi_j \cdot [-d] )  \cdot \bK (e ) )  = ( \bK (e) \cdot \ast  )^\ast   \phi^Y ( \bS_Y ( \xi_j ) [-d] )  \in \Hom^\bullet ( \sF_j   , \sK_i  )^\ast ,
\]
and the first equality holds by abuse of notation as above.
Anyway, applying the isomorphism \eqref{eq:hom_F_and_K} to this equality, we obtain 
\begin{equation*}
    ( \varphi_i  \cdot \ast  )^\ast \phi^Y ( \bS_Y (\xi_j \cdot \bK (e )) [-d]  )  = ( \varphi_i  \cdot \ast  )^\ast( \bK (e)\cdot \ast  )^\ast   \phi^Y ( \xi_j )= ( \bK (e)\cdot \varphi_i  \cdot \ast  )^\ast \phi^Y ( \bS_Y ( \xi_j [-d] ) )
\end{equation*}
in $\Hom ( \sF_j   , \bT \sE_i )^\ast$.
Using that $\varphi \colon \bT \to \bK$ is a natural transformation, we have the following equality
\[
( \varphi_i  \cdot \ast  )^\ast \phi^Y ( \bS_Y (\xi_j \cdot \bK (e ) )[-d]  )  = ( \bK (e)\cdot \varphi_i  \cdot \ast  )^\ast \phi^Y ( \bS_Y ( \xi_j [-d] ) ) 
=  ( \varphi_j \cdot \bT (e)  \cdot \ast  )^\ast \phi^Y ( \bS_Y ( \xi_j [-d] ) )
\]
in $\Hom ( \sF_j   , \bT \sE_i )^\ast$.
Thus, we conclude the equality
\begin{equation}\label{eq:last-eqn}
    ( \varphi_i  \cdot \ast  )^\ast \phi^Y ( \bS_Y (\xi_j \cdot \bK (e ) [-d]  )  = ( \bT (e)  \cdot \ast )^\ast ( \varphi_j  \cdot \ast  )^\ast \phi^Y ( \bS_Y ( \xi_j [-d] ) ) = ( \bT (e)  \cdot \ast  )^\ast \phi^{\bT\E } (\id_{\bT\sE_j} ),
\end{equation}
in $\Hom ( \sF_j   , \bT \sE_i )^\ast$,
where we used the equality \eqref{eq:id} in the second equation. 
Further, 
using commutativity of the second  diagram in \ref{eq:Serre-diag} with $\DD = \thick \bT \E$, $\sB = \bT \sE_j $, $\sA = \sF_j \cong \bS^{-1}_{\bT\E}\bT\sE_j$, and $\sD = \bT\sE_i$,
we get the equality
\[
\phi^{\bT\E} ( \bT(e) ) 
= ( \bT (e) \cdot \ast  )^\ast \phi^{\bT\E } ( \id_{\bT\sE_j} )  
=  ( \varphi_i  \cdot \ast  )^\ast \phi^Y ( \bS_Y  ( \xi_j \cdot \bK (e )  ) [-d] ) \in \Hom ( \sF_j   , \bT \sE_i )^\ast ,
\]
where we used the previous equality \eqref{eq:last-eqn} in the second equation.
This equality shows that the isomorphism \eqref{eq:hom_K_and_F} sends $( \xi_j \cdot \bK (e) ) [-d] = \xi_j [-d] \cdot \bK(e) [-d]$ to $\bS^{-1}_{\bT\E} \bT (e)$, but by \eqref{lem:hom-K-F-a} the inverse of \eqref{eq:hom_K_and_F} sends $\bS^{-1}_{\bT\E} \bT (e)$ to $\bS^{-1}_{\bT\E} \bT (e)  \cdot \xi_i [-d]$.
In other words, we have shown \eqref{lem:hom-K-F-12}.
This finishes the proof of the lemma.
\end{proof}

\begin{lemma}\label{lem:nu-iso}
    Let $\E$ and $\bT$ be as above.
    Let $\sF_1  , \ldots , \sF_m $ and $\xi_i\colon \sK_i \to \sF_i [d] $ in $\Db(Y)$ be as in Lemma \ref{lem:hom-K-F}.
    Assume that the equality of $\Hom$-spaces
    \[
    \Hom ( \sE_i , \bS^{-1}_{\E} \sE_i [d] ) =0 
    \]
    holds for all $1\leq i \leq m$.
    There are non-trivial morphisms $\nu_i \colon \sE_j \to \sF_j [d-1]$ in $\Db(Y)$ for all $1\leq j \leq m$,
    satisfying the equality 
        \begin{equation}\label{eq:nu-def}
    \nu_i \cdot \delta_i = \xi_i \in  \Hom ( \sK_i , \sF_i [d] ) \cong \mathbb{k} \cdot \xi_i  
    \end{equation}
    for all $i$, and such that there is an isomorphism of $\Hom$-spaces
    \begin{equation}\label{eq:nu-iso}
         \Hom^\bullet ( \sK_i , \sE_j ) \xrightarrow{\nu_j \cdot (-)} \Hom^\bullet ( \sK_i , \sF_j  [d-1] )
    \end{equation}
        for all $i $ and $j$.
\end{lemma}

\begin{proof}
 By assumption, the composition of the natural morphism $\varphi_i \colon \bT\sE_i \to \sK_i $ (see \eqref{eq:main-distinguished-triangle} for the definition) with $\xi_i \colon \sK_i \to \sF_i [d]$ vanishes for all $i$, as $\xi_i \cdot \varphi_i$ is contained in
\[
\Hom (\bT \sE_i , \sF_j [d] ) = \Hom ( \bT\sE_i , \bS^{-1}_{\bT \E} \bT \sE_i[d]  ) \cong \Hom ( \sE_i , \bS^{-1}_{\E}\sE_i [d]) =0 , 
\]
using the commutativity of the Serre functor with autoequivalences in the second equality, i.e. the natural isomorphism of functors $\bS_{\bT\E}^{-1} \bT \cong \bT \bS^{-1}_{\E}$.

Next, we show that there are non-trivial morphisms $\nu_i  \colon \sE_i \to \sF_i [d-1]$ satisfying \eqref{eq:nu-def}, for all $i$.
Applying $\Hom^\bullet ( - , \sF_i [d-1] )$ onto the distinguished triangle 
\eqref{eq:main-distinguished-triangle} 
yields a long exact sequence
\[
\ldots \xrightarrow{(-)\cdot \varphi_i} \Hom ( \bT  \sE_i   , \sF_i [d-1] ) 
\xrightarrow{(-) \cdot \eta_i } \Hom ( \sE_i , \sF_i [d-1] ) \xrightarrow{(-)\cdot \delta_i} \Hom ( \sK_i  , \sF_i [d] ) \cong \mathbb{k} \cdot \xi_i \xrightarrow{(-)\cdot \varphi_i} \ldots
\]
of $\Hom$-spaces.
Note that we used Lemma \ref{lem:hom-K-F} \eqref{lem:hom-K-F-a} for $i=j$ in the third term.
We see that the fourth morphism in this long exact sequence is a zero morphism by the explanation above
and thus 
there exists a morphism $\nu_i \colon \sE_i \to \sF_i  [d-1]$ satisfying the equality $\nu_i \cdot \delta_i = \xi_i$ in $\Hom (\sK_i , \sF_i  [d])$.\footnote{It is clear from this argument that $\nu_i$ depends on a choice of a morphism $\bT\sE_i \to \sF_i [d-k]$. We show in the following that the isomorphism of $\Hom$-spaces \eqref{eq:nu-iso} is independent of this choice}
Clearly $\nu_i $ is non-trivial, because $\xi_i$ is a non-trivial morphism.

It remains to show 
that the isomorphism of $\Hom$-spaces \eqref{eq:nu-iso} holds.
For this we observe the following equality of morphisms for any $e$ in $\Hom^\bullet ( \sE_i , \sE_j )$:
\begin{equation}\label{eq:nu-iso-comm}
    \nu_j \cdot  e \cdot \delta_i = \nu_j \cdot \delta_j \cdot \bK(e)= \xi_j \cdot \bK (e) =  \bS^{-1}_{\bT\E}  \bT (e) [d] \cdot \xi_i
    \in  \Hom^\bullet ( \sK_i  , \sF_j  [ d ] ) =\Hom^\bullet (\sF_i , \sF_j ) \cdot \xi_i ,
\end{equation}
where we used the commutativity in Remark \ref{rmk:delta-commut} in the first equation, Lemma \ref{lem:hom-K-F} \eqref{lem:hom-K-F-12} in the third equation and Lemma \ref{lem:hom-K-F} \eqref{lem:hom-K-F-a} for the isomorphism of $\Hom$-spaces.
It follows from this that the isomorphism of $\Hom$-spaces \eqref{eq:nu-iso} holds. 
I.e. surjectivity in \eqref{eq:nu-iso} follows immediately from \eqref{eq:nu-iso-comm} and for injectivity we note that if $\nu_j \cdot  e \cdot \delta_i$ is the zero morphism, then so is $\bS^{-1} \bT (e) [d] \cdot \xi_i$ by \eqref{eq:nu-iso-comm} and by the isomorphism of $\Hom$-spaces in Lemma \ref{lem:hom-K-F} \eqref{lem:hom-K-F-a} it follows that $\bS^{-1}\bT ( e ) [d]$ is the zero morphism.
We conclude that $e$ is the zero morphism and thus \eqref{eq:nu-iso} is also injective.
This finishes the proof of the Lemma.
\end{proof}

\begin{proposition}\label{prop:G}
    Let $X$ be a projective $d$-dimensional variety over $\mathbb{k}$ with a unique isolated Gorenstein singularity at $p \in X$ and assume there exists a crepant resolution $\pi \colon Y \to X$ of $X$.
Let $\bT$ be an autoequivalence of $\Db(Y)$, such that $\pi_\ast$ is invariant under $\bT $ and assume there is an exceptional collection $\E = (  \sE_1 , \ldots , \sE_m )$ in $\Db(Y)$ satisfying
\[
\Hom ( \sE_i , \bS^{-1}_\E \sE_i [d] ) = 0
\]
for all $1\leq i  \leq m$,
and such that $\E$ and $\bT$ are pre-adherent. Set $\tilde{\AA} = \langle \E , \bT ( \E ) \rangle$.
Then there exist objects $\sG_1 , \ldots , \sG_m$ in $\tilde{\AA}$ fitting into a distinguished triangle
    \begin{equation}\label{eq:tilting-triangle}
        \sF_j[d-2] \xrightarrow{\alpha_j} \sG_j \xrightarrow{\beta_j} \sE_j \xrightarrow{\nu_j } \sF_j [d-1]
    \end{equation}
    for all $j$, satisfying:
\begin{enumerate}
    \item\label{lem:G-a} The objects $\bigoplus_i \sE_i$ and $\bigoplus_j \sG_j$ generate $\tilde{\AA}$.
    \item\label{lem:G-b} For all $i$ and all $j$ we have that the equality of $\Hom$-spaces
    \begin{equation}\label{eq:orthogonality-K-G}
            \Hom^\bullet ( \sK_i , \sG_j ) = 0  \quad \text{and} \quad \Hom^\bullet ( \sG_j , \sK_i) = 0
    \end{equation}
    hold.
    If $\E$ and $\bT$ are in addition adherent, then the equalities
    \begin{equation}\label{eq:orthogonality-K-G-adherent}
        \Hom ( \ker\pi_\ast , \sG_j ) = 0  \quad \text{and} \quad \Hom ( \sG_j , \ker\pi_\ast ) = 0    
    \end{equation}
    hold for all $j$.
    \item\label{lem:G-c} Set $\sG = \bigoplus_j \sG_j$ and $\sE= \bigoplus_i \sE_i$.
    There is a long exact sequence of $\Hom$-spaces
    \begin{equation}\label{eq:les-G}
            \ldots \to \Hom( \sE , \bS^{-1}_{\E} \sE [p+d-2] ) \xrightarrow{\alpha \cdot \bT(-) \cdot \eta \cdot \beta} \Hom ( \sG , \sG [p] ) \xrightarrow{q : = ( \ast \cdot \beta )^{-1} (\beta \cdot  (-) )} \Hom ( \sE , \sE [p] ) \to \ldots ,
    \end{equation}
    where $\bS_{\E}$ denotes the Serre functor on $\thick ( \E )$ and $\eta$ and $\beta$ are the (graded)  $\bigoplus_i \eta_i $ and $  \bigoplus_i \beta_i$ and $(\ast \cdot \beta)^{-1}$ is the direct sum of the inverses of the isomorphisms $(-)\cdot \beta_i \colon \Hom^\bullet ( \sE_i , \sG_j ) \xrightarrow{\cong} \Hom^\bullet ( \sG_i , \sG_j  )$.
\end{enumerate}
\end{proposition}

\begin{proof}
By Lemma \ref{lem:nu-iso} we can always find non-trivial morphisms $\nu_j \colon \sE_j \to \sF_j [d-1]$, for all $j$.    
We define $\sG_j$ to be the cocone of $\nu_j $ in $\tilde{\AA}$, for all $j$.
In particular, $\sG_j$ fits into a distinguished triangle of the form \eqref{eq:tilting-triangle}.
Moreover, it is clear from the distinguished triangles \eqref{eq:tilting-triangle} that $\bigoplus_i \sE_i$ and $\bigoplus_j \sG_j$ generate $\tilde{\AA} = \langle \E , \bT ( \E ) \rangle$, i.e. \eqref{lem:G-a} holds.
Furthermore, applying $\Hom^\bullet ( \sK_i , - )$ onto the distinguished triangle \eqref{eq:tilting-triangle}, we obtain a distinguished triangle
\[
\Hom^\bullet ( \sK_i , \sG_j  ) \to \Hom^\bullet ( \sK_i , \sE_j  ) \xrightarrow{\nu_j \cdot (-)} \Hom^\bullet ( \sK_i , \sF_j [d-1]  ) 
\]
in $\Db(\mathbb{k})$.
Since $\nu_j$ is chosen as in Lemma \ref{lem:nu-iso}, the isomorphism of $\Hom$-spaces \eqref{eq:nu-iso} in Lemma \ref{lem:nu-iso}, together with this distinguished triangle, implies that the equality of $\Hom$-spaces
\[
\Hom^\bullet ( \sK_i , \sG_j ) = 0
\]
holds for all $i$ and $j$.
Since $\sK_i$ is $d$-spherical, this equality  also implies
\[
\Hom^\bullet ( \sG_j , \sK_i ) = 0
\]
for all $i$ in $j$.
We have thus shown \eqref{eq:orthogonality-K-G}.
Moreover, if $\E$ and $\bT$ are adherent, then, by definition of adherence, $\ker \pi_\ast$ is generated by $\sK_1 , \ldots , \sK_m$.
Thus \eqref{eq:orthogonality-K-G-adherent} follows from \eqref{eq:orthogonality-K-G} and we conclude \eqref{lem:G-b}.

It remains to show \eqref{lem:G-c}. 
For this, 
we apply and $\Hom^\bullet ( - , \sF_j  )$ onto the triangles \eqref{eq:tilting-triangle} (here we exchange $i$ and $j$) and 
onto \eqref{eq:main-distinguished-triangle}
and we obtain following commutative diagram of distinguished triangles
in $\Db(\mathbb{k})$:
\begin{equation}\label{eq:diag-G}
\xymatrix{
     &  \Hom^\bullet ( \bT \sE_i  , \sF_j  ) \ar@{->}[d]^{(-)\cdot \eta_i } \ar@{->}[r]^{(-) \cdot \eta_i \cdot \beta_i } & \Hom^\bullet ( \sG_i , \sF_j  ) \ar@{->}[d]^{\id}_{=} \\
    \Hom^\bullet (  \sF_i [d-1]  , \sF_j ) \ar@{->}[d]^{\xi_i [-1] \cdot (-)}_{\cong} \ar@{->}[r]^{(-)\cdot \nu_i} &  \Hom^\bullet ( \sE_i  , \sF_j  ) \ar@{->}[d]^{(-)\cdot \delta_i [-1]} \ar@{->}[r]^{(-)\cdot \beta_i } & \Hom^\bullet ( \sG_i , \sF_j  ) \ar@{->}[d]^{} \\
 \Hom^\bullet ( \sK_i [-1] , \sF_j  ) \ar@{->}[r]^{\id}_{=} &  \Hom^\bullet (   \sK_i [-1] , \sF_j  ) \ar@{->}[r]^{} & 0 }
\end{equation}
Note here that the left lower square 
is commutative by  Lemma \ref{lem:nu-iso}) \eqref{eq:nu-def}
and the left most morphism in \eqref{eq:diag-G}
is an isomorphism by Lemma \ref{lem:hom-K-F} \eqref{lem:hom-K-F-a} and the third term is the identity and thus the first morphism is an isomorphism as well.
It follows then by the octahedral axiom that the top most morphism $(-)\cdot \eta_i \cdot \beta_i$ is an isomorphism in $\Db ( \mathbb{k})$.
There is thus a chain of isomorphisms
\begin{equation}\label{eq:first-term}
   \Hom^\bullet ( \sE_i , \bS^{-1}_{\E} \sE_j [d-2] ) \xrightarrow[\cong]{\bT} \Hom^\bullet ( \bT \sE_i , \sF_j [d-2] ) \xrightarrow[\cong]{(-)\cdot \eta_i \cdot \beta_i} \Hom^\bullet ( \sG_i , \sF_j [d-2] )  ,
\end{equation}
where we used the isomorphism of functors $\bS^{-1}_{\bT\E}  \bT \cong \bT  \bS^{-1}_{\E}$ in the second term.
Furthermore, applying $\Hom^\bullet ( - , \sE_i )$ onto \eqref{eq:tilting-triangle} implies the distinguished triangle 
\[
\Hom^\bullet (\sE_j , \sE_i ) \xrightarrow{(-)\cdot \beta_j} \Hom^\bullet ( \sG_j , \sE_i ) \to \Hom^\bullet ( \sF_j [d-2] , \sE_i ) 
\]
in $\Db(\mathbb{k})$.
The third term of of this triangle vanishes, as $\sF_j$ is contained in $\thick\bT( \E )$ and as $\E$ and $\bT$ are pre-adherent.
We remain thus with an isomorphism of $\Hom$-spaces
\begin{equation}\label{eq:third-term}
    \Hom^\bullet (\sE_j , \sE_i ) \xrightarrow[\cong]{(-)\cdot \beta_j} \Hom^\bullet ( \sG_j , \sE_i ) .
\end{equation}
Finally, applying $\Hom^\bullet ( \sG_i , - )$ onto \eqref{eq:tilting-triangle} we obtain a distinguished triangle 
\[
\Hom^\bullet ( \sG_i , \sF_j [d-2] ) \xrightarrow{\alpha_j \cdot (-)} \Hom^\bullet ( \sG_i , \sG_j ) \xrightarrow{\beta_j \cdot (-)} \Hom^\bullet ( \sG_i , \sE_j )
\]
and using \eqref{eq:first-term} and \eqref{eq:third-term} we conclude \eqref{eq:les-G}.
This finishes the proof of the proposition.
\end{proof}

\begin{corollary}\label{cor:tilting}
If the exceptional collection in the setup of Proposition \ref{prop:G} is geometric of type $(m,k)$ with $k\neq d$, then the generator $\sT : = \bigoplus_i \sF_i [d-2] \oplus \bigoplus_j \sG_j $ of $\tilde{\AA}$  induces an equivalence 
\[
\tilde{\AA} \cong \DD_{\mathrm{fg}} ( \tilde{A}^\bullet ) ,
\]
where $\tilde{A}^\bullet$ is a $\dg$-algebra with cohomologies concentrated in degrees $0$ and $k+1-d$.
In particular, if $k=d-1$, then $\sT$ is tilting and induces an equivalence
\[
\tilde{\AA} \cong \Db ( \End( \sT ) ) .
\]
\end{corollary}

\begin{proof}
By a theorem of Keller \cite[Theorem in section 8.7]{keller} it is enough to show that $\Hom^\bullet ( \sT , \sT )$ is concentrated in degrees $0$ and $k+1-d$.
Since $\E$ is geometric, then so is $\F$ and thus $\Hom^\bullet ( \sF , \sF )$ is concentrated in degree $0$.
Moreover, by Proposition \ref{prop:G} \eqref{lem:G-c} the graded $\Hom$-space $\Hom^\bullet ( \sG , \sG )$ is concentrated in degrees $0$ and $k+1-d$.
Moreover, following the proof of Proposition \ref{prop:G} \eqref{lem:G-c} there is an isomorphism of $\Hom$-spaces
\[
\Hom^\bullet ( \sG_i , \sF_j [d-2] ) \cong \Hom^\bullet ( \sE_i , \bS^{-1}_{\E} \sE_j [d-2] ) 
\]
and since $\E$ is geometric, we have that the right-hand-side is concentrated in degree $k+1-d$ (recall here that the iterated right mutation $\bR^{m-1} ( \sE_j )$ agrees with $\bS^{-1}_{\E} \sE_j [k-1]$).
Finally, applying $\Hom^\bullet ( \sF_i [d-2] , - )$ to the distinguished triangle \eqref{eq:tilting-triangle}, we obtain 
\[
\Hom^\bullet ( \sF_i [d-2] , \sF_j [d-2] ) \to \Hom^\bullet ( \sF_i [d-2] , \sG_j  ) \to \Hom^\bullet (  \sF_i [d-2] , \sE_j )
\]
in $\Db(\mathbb{k})$.
Since $\sF_i $ is in $\thick\bT ( \E )$ and $\sE_j$ in $\thick\E$, we obtain by pre-adherence that the third term of this triangle vanishes.
We see in particular that $\Hom^\bullet ( \sF_i [d-2] , \sG_j  )$ is concentrated in degree $0$.
Summarizing, we have that 
\[
\Hom^\bullet ( \sT , \sT ) = \Hom^\bullet ( \sG \oplus \sF [d-2] , \sG \oplus \sF [d-2] ) 
\]
is concentrated in degrees $0$ and $k+1-d$.
\end{proof}

\subsection{The tilting case}
We specialize the previous subsections to the case where the exceptional collection $\E$ satisfying the adherence property (together with an autoequivalence $\bT$) is geometric of type $(m,k)$ (see Definition \eqref{def:geom-type}) and such that the equality $k = d-1$ holds.
We need the following well-known lemma:

\begin{lemma}\label{lem:equiv-bimod}
    Let $\CC$ and $\DD$ be additive categories and let $H\colon \CC \to\DD$ be an additive equivalence. Let $X, Y$ be objects in $\CC$. Then $H$ induces an isomorphism of $\End_\CC(Y)$-$\End_\CC(X)$-bimodules
    \begin{align*}
        \prescript{}{\End ( Y )}{\Hom_\CC(X, Y)}_{\End(X)} \xrightarrow{\cong} \prescript{}{\End ( Y )}{\Hom_\DD(HX, HY)}_{\End(X)} ,
    \end{align*}
    where $x \in \End_\CC(X)$ acts on $\Hom_\DD(HX, HY)$ by precomposition with $H(x)$ and similarly for $\End_\CC(Y)$. 
\end{lemma}

\begin{prop}\label{prop:algebra-extensions}
    Let $X$ be a projective $d$-dimensional variety over $\mathbb{k}$ with a unique isolated Gorenstein singularity at $p \in X$ and assume there exists a crepant resolution $\pi \colon Y \to X$ of $X$.
Let $\bT$ be an autoequivalence of $\Db(Y)$, such that $\pi_\ast$ is invariant under $\bT $ and assume there is a geometric exceptional collection $\E = (  \sE_1 , \ldots , \sE_m )$ of type $(m, d-1)$ in $\Db(Y)$, such that $\E$ and $\bT$ are adherent and set $\tilde{\AA} = \langle \E , \bT ( \E ) \rangle$.
Let $\sG $ be the object $\tilde{\AA}$ as defined in Proposition \ref{prop:G}.
Then $\End ( \sG )$ sits in the middle of the following short exact sequence
\begin{equation}\label{eq:ses-G}
    0 \to  \Hom(\sE, \bS^{-1}_\E\sE [d-2] ) \xrightarrow{\alpha \cdot \bT(-) \cdot \eta \cdot \beta} \End( \sG ) \xrightarrow{\varphi} \End( \sE ) \to 0 ,
\end{equation}
where $\sE = \bigoplus \sE_i$ and $\varphi$ is an algebra homomorphism.
\begin{enumerate}
\item\label{prop:algebra-extensions-3} Moreover, the image $I$ of $\Hom(\sE, \bS^{-1}_\E \sE [d-2])$ in $\End( \sG )$ satisfies 
\begin{equation}
    I^2=0 .
\end{equation}
\item\label{prop:algebra-extensions-2} 
Set $E:=\End ( \sE )$. Then the inclusion
\[
\Hom(\sE, \bS^{-1}_\E\sE [d-2] ) \xhookrightarrow{\alpha \cdot \bT(-) \cdot \eta \cdot \beta} \End( \sG )
\]
is an isomorphism of $E$-bimodules onto $I$.
Moreover,
this $E$-bimodule is isomorphic to  
\[
\Hom(E, \bS_E^{-1}(E)[d-2]) .
\]
\end{enumerate}
\end{prop}

\begin{proof}
The short exact sequence of $\Hom$-spaces \eqref{eq:ses-G} follows immediately from Proposition \ref{prop:G} \ref{lem:G-c}. 
Moreover, the map $\varphi$ is given by sending an endomorphism $\gamma$ of $\sG$ to the unique endomorphism $\epsilon$ of $\sE$ satisfying $\epsilon \beta = \beta \gamma$ (see Proposition \ref{prop:G} \ref{lem:G-c}).  
In particular, the uniqueness of $\epsilon$ can be used to show that this defines an algebra homomorphism, which is surjective by \eqref{eq:ses-G}.
Together with $\beta\alpha=0$ (see definition of $\alpha$ and $\beta$ in Proposition \ref{prop:G}) 
we see that $I^2=0$, which shows \eqref{prop:algebra-extensions-3}.

It remains to show \eqref{prop:algebra-extensions-2}.
The fact that $I = \Hom ( \sE , \bS^{-1}_\E\sE [d-2] )$ is an $\End(\sE)$-bimodule follows from $I^2 = 0$.
Namely, by the above we have that for any $\epsilon$ in $\End ( \sE )$ there is a $\gamma $ in $\End( \sG )$ with $\varphi ( \gamma ) = \epsilon$.
We then define the left action by $\epsilon$ on $I$ by $\epsilon z := \gamma z$ for any $z$ in $I$.
For another $\gamma '$ in $\End ( \sG )$ satisfying $\epsilon =\varphi ( \gamma ' )$, we have that $\gamma - \gamma'$ is in $I$.
Thus $(\gamma - \gamma ' ) z = 0$ and hence the left action $\epsilon$ on $I$ is well-defined.
The right action of $\epsilon$ on $I$ is similarly defined.
Next we show that the actions of $\End(\sE)$ on $I$ coincide with the actions of $\End (\sE)$ on $ \Hom (\sE , \bS^{-1}_{\bT\E}\bT \sE [d-2] )$ via $\alpha \cdot \bT(-) \cdot \eta \cdot \beta$, where
$\epsilon\in \End(\sE)$ acts on the right
by composition with $\bS^{-1}_{\E} (\epsilon ) [d-2]$.
Indeed, for the left action we have that for all $\epsilon = \varphi (\gamma )$ in $\End ( \sE )$ that the equalities 
\[
\alpha \bT ( \bar{z} ) \eta \beta \epsilon = \alpha \bT ( \bar{z} ) \eta \beta  \gamma = \alpha \bT ( \bar{z} ) \eta \epsilon \beta = \alpha \bT ( \bar{z} \epsilon ) \eta  \beta \in I
\]
hold for all $\bar{z}$ in $\Hom ( \sE , \bS^{-1}_\E \sE [d-2] )$.
Note that we used that $\eta \colon \id \to \bT $ is a natural transformation in the third equation.
For the right action, we  
first show that for each $\gamma$ in $\End ( \sG )$, 
there exists a $f$ in $\End (\sF [d-2] )$, such that 
\begin{equation}\label{eq:alpha-f}
    \gamma \alpha = \alpha f \in \Hom (\sF , \sG )
\end{equation}
holds. 
Indeed, applying $\Hom^\bullet (\sF_j , -)$ to the triangle \eqref{eq:tilting-triangle}
we get a distinguished triangle 
\[
\Hom^\bullet ( \sF_j , \sF_i [d-2] ) \xrightarrow{\alpha_i \cdot (-)} \Hom^\bullet ( \sF_j  , \sG_i ) \to \Hom^\bullet ( \sF_j , \sE_i )
\]
in $\Db(\mathbb{k})$.
Note that the third term of this distinguished triangle vanishes, as $\F$ is contained in $\thick \bT ( \E )$ and as $\E$ and $\bT ( \E )$ are (pre-)adherent.
In particular we obtain an isomorphism 
\begin{equation}\label{eq:f_1}
\End( \sF [d-2] )  \xrightarrow[\cong]{\alpha \cdot (-) }  \Hom( \sF [d-2] , \sG ).
\end{equation}
This isomorphism of $\Hom$-spaces immediately implies that such an $f \in \End ( \sF [d-2] )$ exists.
Next we show that an $f \in \End( \sF[d-2] )$ as above is equal to $\bS^{-1}_{\bT\E}\bT (\epsilon ) $, where $\epsilon = \varphi (  \gamma )$. 
To prove this, we note first that such $f$ and $\epsilon$ fit into a commutative diagram
\begin{equation*}\label{eq:E-G-F}
\xymatrix{
    \sF [d-2]  \ar@{->}[d]^{f} \ar@{->}[r]^{\alpha} &  \sG \ar@{->}[d]^{\gamma} \ar@{->}[r]^{\beta} & \sE \ar@{->}[d]^{\epsilon} \ar@{->}[r]^{\nu} & \sF [d-1] \ar@{->}[d]^{f[1]} \\
 \sF [d-2] \ar@{->}[r]^{\alpha} &  \sG \ar@{->}[r]^{\beta} & \sE \ar@{->}[r]^{\nu} & \sF [d-1] }.
\end{equation*}
In particular, we see from this diagram that the commutativity $\nu \epsilon = f [1] \nu$ in $\Hom ( \sE , \sF [d-1] )$ holds.
Precomposing this equation with the morphism $\delta \colon \sK \to \sE [1]$, we obtain
\begin{equation}\label{eq:f-STe-1}
    \nu \cdot \epsilon \cdot \delta = f [1] \cdot \nu \cdot \delta = f[1] \cdot \xi \in \Hom ( \sK , \sF [d]) .
\end{equation}
Note that we used here that by definition of $\nu$ the equality $\nu\cdot \delta = \xi$ holds.
Furthermore, the first term can be further rewritten as follows:
\begin{equation}\label{eq:f-STe-2}
\nu \cdot \epsilon \cdot \delta = \nu \cdot \delta \cdot \kappa (\epsilon ) = \xi \cdot \kappa ( \epsilon ) = \bS^{-1}_{\bT\E} \bT ( \epsilon ) \cdot \xi  \in \Hom (\sK , \sF [d] ) .
\end{equation}
Here, we used the commutativity in Remark \ref{rmk:delta-commut} in the first equation, the equation $\nu\cdot \delta = \xi$ of $\nu$ (see \eqref{eq:nu-def}) in the second equation, and Lemma \ref{lem:hom-K-F} \eqref{lem:hom-K-F-12} in the third equation.
Combining \eqref{eq:f-STe-1} and \eqref{eq:f-STe-2}, we obtain the equality
\[
f[1] \cdot \xi = \bS^{-1}_{\bT\E} \bT ( \epsilon ) \cdot \xi  \in \Hom (\sK , \sF [d] ) .
\]
We conclude by Lemma \ref{lem:hom-K-F} \eqref{lem:hom-K-F-a} that the equality $f = \bS_{\bT\E}^{-1} \bT ( \epsilon )$ in $\End(\sF [d-2])$ holds. 
Combining this equality with \eqref{eq:alpha-f}, we obtain a chain of equalities
\[
\epsilon \alpha ( \bar{z} ) \beta  = \gamma \alpha \bT ( \bar{z} ) \eta \beta   = \alpha \bS_{\bT\E}^{-1} \bT ( \epsilon )\bT ( \bar{z} ) \eta  \beta = \alpha \bT\bS_{\E}^{-1}  (\epsilon )\bT ( \bar{z} ) \eta  \beta = \alpha \bT ( \bS_{\E}^{-1}  ( \epsilon ) \bar{z} ) \eta  \beta  \in I ,
\]
where we used the natural isomorphism of functors $\bT\bS_{\E}^{-1}  \cong \bS^{-1}_{\bT\E}\bT$.
In other words, the right action of $\End(\sE)$ on $I$ coincides with the right action of $\End (\sE)$ on $\Hom ( \sE , \bS^{-1}_{\bT\E} \sE[d-2] )$ via $\alpha \cdot \bT(-) \cdot \eta \cdot \beta$.

Finally, using the equivalence $\thick ( \E ) \cong \Db( E )$ given by tilting, we obtain an isomorphism of $E  = \End( \sE )$ bimodules
\[
\prescript{}{\End ( \sE )}{\Hom (  \sE ,  \bS^{-1}_{\E}\sE [d-2] )}_{\End ( \sE )} 
    \cong \prescript{}{E}{\Hom ( E ,  \bS^{-1}_E ( E ) [d-2] )}_{E} ,
\]
where the left action on $\Hom ( E ,  \bS^{-1}_E ( E ) [d-2] )$ is given by pre-composing morphism $\epsilon \colon E\to E$, and the right action is given by composing morphism $\bS_E^{-1} (\epsilon )[d-2]$.
This finishes the proof of \eqref{prop:algebra-extensions-2} and the proof of the proposition.
\end{proof}

\subsection{Describing \texorpdfstring{$\tilde{\AA} / \ker \pi_*$}{Amod\textasciicircum ker}}

Let $B^\bullet $ be a $\dg$ $\mathbb k$-algebra with cohomology concentrated in non-positive degrees.
Let $B^\bullet_{\leq 0}$ be the good truncation of $B^\bullet$. It is a dg subalgebra and the inclusion $B^\bullet_{\leq 0} \to B^\bullet $ is a quasi-isomorphism. 
 Hence, there are induced equivalences
\begin{equation}\label{eq:dg-truncation}
    \DD ( B^\bullet ) \cong \DD 
( B^\bullet_{\leq 0} ) \quad \text{and} \quad  \DD_{\mathrm{fg}} ( B^\bullet ) \cong \DD_{\mathrm{fg}} ( B^\bullet_{\leq 0} ) .
\end{equation}
Since we are only interested in derived categories, we can (and will) assume that $B^\bullet$ is a \emph{non-positive} dg algebra, i.e. $B^i=0$ for all $i>0$.
The good truncation $M_{\leq 0}^\bullet$ of a $\dg$ $B^\bullet$-module $M^\bullet$, viewed as a complex over $\mathbb k$, inherits a $\dg$ submodule structure of $M^\bullet $ (note that this uses that $B^\bullet $ is non-positive).
Hence, good truncation induces a $t$-structure on $\DD ( B^\bullet )$ and by \cite[Theorem 1.3]{hoshino-et-al} the corresponding heart is equivalent to $\bH^0( B^\bullet)\mbox{-}\Mod$.
Moreover, this $t$-structure descends to a bounded $t$-structure of $\DD_{\mathrm{fg}} ( B^\bullet ) $ with heart equivalent to $\bH^0( B^\bullet  )\mbox{-}\mod$.

An idempotent $e$ of $B^\bullet $ is contained in the degree zero component $B^0$ (this follows from $e^2 = e $).
Thus, by the Leibniz rule  the equality $d_B(b\cdot e) = d_B(b) \cdot e + (-1)^i \cdot b  \cdot d_B(e) = d_B(b) \cdot e$ holds for all $b\in B^i$, and therefore $B^\bullet e$ is a $\dg$ submodule of $B^\bullet$. 

For given $\dg$ $B^\bullet $-modules $M^\bullet$ and $N^\bullet$, we denote by $\Hom_{B^\bullet } ( M^\bullet , N^\bullet )$
the complex whose degree $n$
component $\Hom_{B^\bullet} ( M^\bullet , N^\bullet )^n$ consists of homogeneous $B^\bullet$-linear maps $M^\bullet \to N^\bullet$ of degree $n$,
and whose diﬀerential takes a homogeneous map $f$ of degree $|f|$ to
$d_{N} \cdot f - (-1)^{|f|} d_{M} \cdot f$.

The following result has been communicated to us by Dong Yang, for which we are very grateful. 
A variant of this result will appear in a forthcoming work \cite{dong-et-al} by Jin-Yang-Zhou.

\begin{proposition}[Jin-Yang-Zhou \cite{dong-et-al}]\label{prop:dg-verdier}
    Let $B^\bullet$ be a non-positive $\dg$ $\mathbb k$-algebra with finite-dimensional total cohomology.
Let $e$ be an idempotent of $B^\bullet$, let $\bar{e} = \bH^0 (e)$ be the induced idempotent of $\bH^0 ( B^\bullet )$ and let $\bP_e$ be the functor 
\[
\bP_e:= \Hom_{B^\bullet}^\bullet ( B^\bullet e , -) \colon \DD_{\mathrm{fg}}( B^\bullet)\to \DD_{\mathrm{fg}}( e B^\bullet e ) .
\]
There is an equality of subcategories
\begin{equation}\label{eq:equality-subcat}
    \ker \bP_e = \thick (   (\bH^0( B^\bullet )/\bar{e})\mbox{-} \mod )
\end{equation}
in $\DD_{\mathrm{fg}} ( B^\bullet )$, and $\bP_e$
induces an equivalence 
\begin{equation}\label{eq:dg-verdier-localization}
    \DD_{\mathrm{fg}}( B^\bullet)  / \ker \bP_e \cong \DD_{\mathrm{fg}}( e B^\bullet e) .
\end{equation}
\end{proposition}

\begin{proof}
We first show the inclusion of subcategories``$ \supseteq$'' in \eqref{eq:equality-subcat}. 
Note that $(\bH^0( B^\bullet )/\bar{e})\mbox{-} \mod$ is generated by those simple  $\bH^0( B^\bullet )$-modules $\sS$ that satisfy $\bar{e} \sS  = 0$.
This implies 
\begin{align}
  \bP_e(\sS)=\Hom^\bullet( B^\bullet e , \sS ) \cong e \sS  = \bar{e}\sS= 0,  
\end{align}
where $e \sS  = \bar{e}\sS$ holds by definition of the $B^\bullet$-action on the $\bH^0( B^\bullet )$-module $\sS$, which is induced by the natural map $B^\bullet \to \bH^0( B^\bullet )$. Thus all generators of the triangulated category on the right hand side of \eqref{eq:equality-subcat} are contained in the category on the left, showing the inclusion ``$ \supseteq$''.    

To show the inclusion ``$\subseteq$'', assume that $\sK$ is contained in $\ker \bP_e$.
By the discussion above Proposition \ref{prop:dg-verdier}, its $i$-th cohomology $\bH^i ( \sK )$ has an $\bH^0 ( B^\bullet )$-module structure.
Note that 
\begin{align}
    \bar{e} \bH^i( \sK ) \cong \bH^i( e\sK ) \cong \bH^i(\bP_e ( \sK ))=0.
\end{align}
Hence, all $\bH^i ( \sK )$ are contained in $( \bH^0 ( B^\bullet ) / \bar{e}) \mbox{-} \mod $. Since $\sK \in \DD_{\mathrm{fg}} ( e B^\bullet e )$, it has finite dimensional total cohomology and thus only cohomology in finitely many degrees. Combining these statements implies that $\sK$ is contained in $\thick ( ( \bH^0 ( B^\bullet ) / \bar{e}) \mbox{-} \mod )$.
This concludes the proof of the equality \eqref{eq:equality-subcat}.

    The equivalence \eqref{eq:dg-verdier-localization}, is shown in \cite{dong-et-al}. 
    The strategy of the proof is similar to the proofs of \cite[Lemma 2.32 and Theorem 2.30]{pavic-shinder}.
\end{proof}

Recall that
$\Phi\colon \tilde{\AA} \xrightarrow{\cong} \DD_{\mathrm{fg}}(\tilde{A}^\bullet )$ denotes the equivalence appearing 
in Corollary \ref{cor:tilting}.
We apply the previous Proposition in Corollary \ref{cor:veriderloc-negative-dg} to obtain that $\tilde{\AA}/ \ker \pi_\ast \cong \DD_{\mathrm{fg}} ( e \tilde{A}^\bullet e)$ for some idempotent $e$ of $\tilde{A}^\bullet$, if $\tilde{A}^\bullet$ has finite-dimensional total cohomology and is non-positive.
We need the following lemma.

\begin{lemma}\label{L:SimplesGenKernel}
We assume the setup and notation of Proposition \ref{prop:G} and Corollary \ref{cor:tilting}.
Suppose  additionally that $\E$ and $\bT ( \E )$ are adherent and that $k-1\leq d$ holds (i.e. $\tilde{A}^\bullet$ has finite-dimensional total cohomology and is non-positive).
Then there is an  equality of subcategories of $\DD_{\mathrm{fg}} ( \tilde{A}^\bullet )$
\[
\Phi(\ker \pi_*) = \thick ( \bH^0( \tilde{A}^\bullet ) / \bar{e} \mbox{-} \mod   ),
\]
where $e$ is the idempotent of $\tilde{A}^\bullet$ corresponding to $\Phi ( \sG )$ and $\bar{e} = \bH^0 ( e ) \in \bH^0 ( \tilde{A}^\bullet )$.
\end{lemma}

\begin{proof}
    Recall that $\sK_i $ denotes the cone of $\eta_{i} \colon \sE_i \to \bT \sE_i$ for all $1\leq i \leq m$. Let $\KK = \thick ( \sK_1  , \ldots , \sK_{m})$, which agrees with $\ker \pi_\ast$ by the adherence assumption.
    Set $\sS_m := \sK_{m} [d-2] $ and
    \[
    \sS_i : = \bT_{\sK_{m}}^{-1} \circ \ldots \circ \bT^{-1}_{\sK_{i+1}} ( \sK_{i} ) [d-2] \in \KK ,
    \]
    for $1 \leq i \leq m -1$.
      It is clear that the $\sS_1 , \ldots , \sS_{m}$ generate $\KK$.
    To prove the lemma, it is enough to show  that $\Phi ( \sS_i ) \cong \bH^0 \Phi ( \sS_i )$ in $\DD_{\mathrm{fg}}( \tilde{A}^\bullet )$ holds, and that the $\bH^0 \Phi ( \sS_i )$ are the $\bH^0 ( \tilde{A}^\bullet )$-simples 
    satisfying $\bar{e} \bH^0 \Phi ( \sS_i )  = 0$, since these simples
    generate $\thick ( \bH^0( \tilde{A}^\bullet ) / \bar{e} \mbox{-} \mod)$.
    To show this, it suffices to show that 
 \begin{equation}\label{eq:simple-a}
     \Hom^\bullet ( \sG , \sS_i  ) = 0 \quad \text{and that} \quad 
         \Hom ( \sF_{j} [d-2] , \sS_i [p  ] ) \cong 
        \begin{cases}
\mathbb k, &\text{ $j= i$ and $p=0$}\\
0 , &\text{ else}
\end{cases}
 \end{equation}
 for all $1\leq j , i \leq m$.
 Indeed, \eqref{eq:simple-a} implies (by applying $\Phi$) that there are isomorphisms  
 \begin{equation}\label{eq:simple-Phi}
     \Hom^\bullet_{\DD( \tilde{A}^\bullet)} (  \tilde{A}^\bullet e , \Phi (\sS_i ) ) = 0 \quad \text{and} \quad 
         \Hom_{\DD ( \tilde{A}^\bullet )} (  \tilde{A}^\bullet e_{m+j} , \Phi ( \sS_i ) [p  ] ) \cong 
        \begin{cases}
\mathbb k, &\text{ $j= i$ and $p=0$}\\
0 , &\text{ else}
\end{cases}
 \end{equation}
 where the idempotents $e_{m+1} , \ldots , e_{2m}$ are the idempotents of $\tilde{A}^\bullet$ correspnding to $\Phi ( \sF_1 ) [d-2] , \ldots , \Phi ( \sF_m ) [d-2]$.
 Using \eqref{eq:simple-Phi}, we show that $\Phi ( \sS_i ) \cong \bH^0 \Phi ( \sS_i )$ in $\DD_{\mathrm{fg}}( \tilde{A}^\bullet )$ holds, and that  $\bar{e} \bH^0 \Phi ( \sS_i )  = 0$ is satisfied.
 Indeed, 
 since objects in $\DD ( \tilde{A}^\bullet  )$ of the form $ \tilde{A}^\bullet e'$ are $\KK$-projective, for any idempotent $e'$ (meaning the complex $\Hom^\bullet_{\tilde{A}^\bullet}(\tilde{A}^\bullet e' , N^\bullet ) \cong e' N^\bullet $ is acyclic, for all acyclic $N^\bullet$), we can replace $\Hom_{\DD ( \tilde{A}^\bullet )}$ by $\Hom_{\bK (\tilde{A}^\bullet )}$ in \eqref{eq:simple-Phi}, where $\bK (\tilde{A}^\bullet )$ denotes the homotopy category of $\tilde{A}^\bullet$.
 Furthermore, $\Hom_{\bK (\tilde{A}^\bullet )}$ is, by definition, just $\bH^0 \Hom_{\tilde{A}^\bullet}$, and so \eqref{eq:simple-Phi} can be rewritten as
 \begin{equation}\label{eq:simple-cohom}
     \bH^\bullet ( e \Phi (\sS_i )  ) = 0 \quad \text{and} \quad 
         \bH^{p } (  e_{m+j} \Phi ( \sS_i )  ) \cong 
        \begin{cases}
\mathbb k, &\text{ $j= i$ and $p=0$}\\
0 , &\text{ else}
\end{cases}
 \end{equation}
 and thus  $\Phi( \sS_i ) \cong \bH^0 ( \Phi \sS_i )$ in $\DD_{\mathrm{fg}} ( A^\bullet  )$ holds, and $\bar{e} \bH^0 \Phi ( \sS_i )  = 0$ is satisfied.
 In particular, $\Phi( \sS_i ) $ can be viewed as a $\bH^0 ( \tilde{A}^\bullet )$-module. 
 Moreover, by e.g. \cite[Section 4.4]{koenig-yang}, the equalities in \eqref{eq:simple-Phi} imply further that the $\Phi( \sS_i ) \cong \bH^0 ( \Phi \sS_i )$ are distinct $\bH^0 ( \tilde{A}^\bullet )$-simples.
 \sout{(see also} \cite[Section 4.4]{koenig-yang}).

  It remains to show the equalities in \eqref{eq:simple-a}.
     Since the $\sS_i$ are contained in $\KK$, we obtain from \eqref{eq:orthogonality-K-G-adherent} in Proposition \ref{prop:G} \eqref{lem:G-b} that the vanishing $\Hom^\bullet (\sG , \sS_i ) = 0$ holds, for all $1\leq i \leq m$.
    For the second set of isomorphisms of \eqref{eq:simple-a}, we note that 
    \[
    \Hom^\bullet ( \sE_j [-1] , \sS_i ) \cong \Hom^\bullet (  \sF_{j} [d-2] , \sS_i  ) , 
    \]
    for all $1\leq i, j \leq m$, by the distinguished triangle \eqref{eq:tilting-triangle} and using $\Hom^\bullet (\sG_j , \sS_i ) = 0$ (see above).
    Using this isomorphism, 
    we see that showing the second part of \eqref{eq:simple-a} is equivalent to showing that
    \begin{equation}\label{eq:simple2}
         \Hom ( \sE_{j} , \sS_i [p+1] ) \cong 
        \begin{cases}
\mathbb k, &\text{ $j=i$ and $p=0$}\\
0 , &\text{ else}
\end{cases}
    \end{equation}
    for all $1\leq j , i \leq m$.
    To show this, we first consider the case $i<j$.
    We have the following equality of $\Hom$-spaces:
    \[
    \Hom^\bullet (  \sE_j , \sS_i ) \cong \Hom^\bullet (  \bT_{\sK_{i+1}} \ldots \bT_{\sK_{m}} \sE_j , \sK_{i} [d-2] ) \cong \Hom^\bullet (  \bT \sE_j , \sK_{i} [d-2] ) = 0 .
    \]
    Note that we used the isomorphisms $\bT_{\sK_j} \sE_j \cong \bT_{\sK_k}\bT_{\sK_j} \sE_j \cong \bT_{\sK_k}\bT_{\sK_j} \bT_{\sK_l} \sE_j $ for all $k<j<l$ (see \eqref{eq:j-lower-i} and \eqref{eq:j-bigger-i}) and $  \bT_{\sK} \sE_j \cong \bT\sE_j $ (see \eqref{eq:last-iso}) in the second equality, and Lemma \ref{lem:hom-computation} \eqref{eq:hom-computation-iso1_2} in the third equality.
    On the other hand, in the case $i\geq j$ we have a chain of isomorphisms
    \[
    \Hom^\bullet ( \sE_j , \sS_i [1] ) \cong   \Hom^\bullet (  \bT_{\sK_{i+1}} \ldots \bT_{\sK_{m}} \sE_j , \sK_{i} [d-1] ) \cong \Hom^\bullet (   \sE_j , \sK_{i} [d-1] ) \cong \begin{cases}
\mathbb k[0], &\text{ $j= i$}\\
0 , &\text{ else}
\end{cases}
    \]
    where we used the isomorphism $\bT_{\sK_l} \sE_j \cong \sE_j $ for all $j<l$ in the second equality, and Lemma \ref{lem:hom-computation} \eqref{eq:hom-computation-iso2_2} in the third equality.
    Combining these two cases, we conclude \eqref{eq:simple2} and we finish the proof of the lemma.
\end{proof}

\begin{corollary}\label{cor:veriderloc-negative-dg}
Let $X$ be projective Gorenstein variety over $\mathbb k$ with a unique  singular point $p\in X$. Assume that a crepant resolution $\pi\colon Y \to X$ exists. 
Let $\bT \colon \Db( Y ) \to \Db ( Y )$ be an autoequivalence which is invariant under $\pi_\ast$.
Assume there is a geometric exceptional sequence $\E = ( \sE_1, \ldots, \sE_{m} )$ of type $(m,k)$, such that $k \leq d-1$ and $\sE$ and $\bT$ are adherent.
Set $\tilde{\AA}=\langle \E, \bT ( \E ) \rangle$.
Then there is an equivalence
\[
\tilde{\AA} / \ker\pi_\ast \cong \DD_{\mathrm{fg}} ( A^\bullet ) ,
\]
and a semiorthogonal decomposition
\begin{equation}\label{eq:sod-quot1-b}
    \Db (Y) /\ker \pi_\ast = \langle \DD_{\mathrm{fg}} ( A^\bullet ) ,   {^{\perp}\tilde{\AA}} \rangle  ,
\end{equation}
where $A^\bullet$ is a $\dg$ $\mathbb k$-algebra with cohomologies
\begin{align}
  \bH^i(A^\bullet) \cong \begin{cases}
      \End(\sE)  \quad \qquad \qquad \qquad \text{ if } \, \, i=0 \\
      \Hom ( \sE , \bS^{-1}\sE [k-1] ) \quad \text{ if } \, \, i=k+1-d \\
      0  \,\,\,\, \quad \qquad \qquad \qquad \qquad \text{ else.}
      \end{cases}  
\end{align}
Moreover, if there is an equivalence $\Db (Y) /\ker \pi_\ast \cong \Db(X)$ induced by the functor $\pi_\ast \colon \Db(Y) \to \Db(X)$, then the decomposition
\begin{equation}\label{eq:sod-quot2-b}
    \Db (X) = \langle \DD_{\mathrm{fg}} ( A^\bullet ) ,  \pi_\ast {^{\perp}\tilde{\AA}} \rangle 
\end{equation}
is admissible and $ {^{\perp}\tilde{\AA}} \cong \pi_\ast {^{\perp}\tilde{\AA}}$
is contained in $ \Dperf(X)$.
\end{corollary}

\begin{proof}
By Corollary \ref{cor:tilting}, there is a an equivalence
\[
\Phi \colon \tilde{\AA} \xrightarrow{\cong} \DD_{\mathrm{fg}} ( \tilde{A}^\bullet )
\]
given by the object $\sF[d-2] \oplus \sG $ in $\tilde{\AA}$ as defined in Proposition \ref{prop:G}.
Let $e$  be the idempotent of $\tilde{A}^\bullet$ corresponding to $\sG$ via $\Phi$.
By \eqref{eq:dg-truncation}, we can assume that $\tilde{A}^\bullet$ is non-positive. 

Set $A^\bullet = e\tilde{A}^\bullet e $.
By Corollary \ref{cor:tilting}, the cohomology of $A^\bullet $  in degrees $0$ and $k+1-d$ is isomorphic to $\Hom^\bullet ( \sG , \sG )$ (as $e$ is the idempotent of $\tilde{A}^\bullet $ corresponding to $\sG$). 
By Proposition \ref{prop:G} \ref{lem:G-c}, the cohomology of $A^\bullet$ in degrees $0$ and $k+1-d$ is then isomorphic to $\End ( \sE )$ and $\Hom ( \sE , \bS^{-1}\sE [k-1])$, respectively and vanishes everywhere else.
Moreover, there is an exact functor 
\[
\bP_e(-) = \Hom_{\tilde{A}^\bullet}^\bullet (\tilde{A}^\bullet e , -) \colon \DD_{\mathrm{fg}}(\tilde{A}^\bullet)\to \DD_{\mathrm{fg}}( A^\bullet ).
\]
By Proposition \ref{prop:dg-verdier} \eqref{eq:equality-subcat} and Lemma 
\ref{L:SimplesGenKernel} we have  the equality of subcategories
\[
\ker \bP_e = \Phi ( \ker \pi_\ast )
\]
in $\DD_{\mathrm{fg}} ( \tilde{A}^\bullet )$.
Thus, by Proposition \ref{prop:dg-verdier} \eqref{eq:dg-verdier-localization} there is an equivalence 
\[
\tilde{\AA} / \ker \pi_\ast \xrightarrow[\cong]{\Phi } \DD_{\mathrm{fg}} ( \tilde{A}^\bullet ) / \ker \bP_e  \cong \DD_{\mathrm{fg}}( A^\bullet  ).
\]
The equivalences \eqref{eq:sod-quot1-b} and \eqref{eq:sod-quot2-b} follow from Corollary \ref{cor:sod-quotient}, see \eqref{eq:sod-quot1} and \eqref{eq:sod-quot2}, respectively.
\end{proof}

\begin{example}\label{ex:nodal-ks}
Assume that the base field $\mathbb{k}$ is algebraically closed of characteristic not equal to $2$.
Then Corollary \ref{cor:veriderloc-negative-dg} can be applied to $1$-nodal projective threefolds $X$ admitting a small (hence crepant) resolution $\pi \colon Y \to X$. We now explain how this recovers the $1$-nodal threefold case of Kuznetsov-Shinder \cite[Theorem 6.17 (i)]{ks}.

 The subcategory $\ker\pi_\ast \subseteq \Db(Y)$ is generated by one $3$-spherical object. Let $\bT$ be the corresponding spherical twist.
Assume there exists an exceptional object $\sE$ in $\Db(Y)$, such that $\sE$ and $\bT$ are (pre-)adherent (note that our notions of pre-adherence and adherence coincide in this example, as $\ker\pi_\ast$ has a single generator. Moreover, this coincides with the notion of adherence in \cite[Definition 3.9]{ks}, see Remark \ref{rmk:adherence-old-def}).
Note that a single exceptional object forms a geometric exceptional sequence of type $(1,1)$.
Then using Corollary \ref{cor:veriderloc-negative-dg}, together with \cite[Corollary 5.6 and Lemma 5.7]{ks}, we obtain  an admissible semiorthogonal decomposition 
\[
\Db(X) \cong \langle \DD_{\mathrm{fg}} ( A^\bullet ) , {{^\perp}\tilde{\AA}} \rangle ,
\]
with ${{^\perp}\tilde{\AA}} \subset \Dperf ( X )$ and $A^\bullet$ is a $\dg$ $\mathbb k$-algebra with cohomologies 
\begin{align}
    \bH^i(A^\bullet)=\begin{cases}
        \mathbb{k} \quad \text{ if } i=0 \text{ or } i=k+1-d=-1  \\
        0  \quad \text{ else}.
    \end{cases}
\end{align}
This implies that $A^\bullet$ is quasi-isomorphic to the graded algebra $\mathbb{k}[z] / (z^2)$ with $\deg(z) =-1$ (viewed as a dg algebra with trivial differential), cf. \cite[Theorem 2.1.]{KellerYangZhou}.
\end{example}

\section{Explicit descriptions of the algebras }\label{sec:algebra-description}

In this section, we assume that the base field $\mathbb{k}$ is algebraically closed.
The algebras $A=\End(\sG)$ in Proposition \ref{prop:algebra-extensions} are closely related to so-called rolled-up helix algebras, which in our situation are also known as higher preprojective algebras, cf. Remark \ref{rem:rolled-up}. Indeed, in certain cases (including some projective cones $X$ -- in particular, weighted projective spaces of the form $X=\P(1^d, d)$), $A$ is a natural quotient of a rolled-up helix algebra. In general, $A$ is a deformation of such a quotient algebra. 

\subsection{Preparation: rolled-up helix algebras and higher preprojective algebras}

Our description of the algebras $A$, builds on known descriptions of higher preprojective algebras due to Keller \cite{KellerCY} and Hanihara \cite{Hanihara2}, which we recall in this subsection.

\begin{setup}\label{Set:fintype}
    Let $\DD$ be a $\mathbb k$-linear algebraic triangulated category of finite type, i.e. $ \sum_{s \in \Z} \dim_{\mathbb k} \Hom_\DD (A , B [s]) < \infty$ for all objects $A, B$ in $\DD$.
\end{setup}

Bounded derived categories of smooth projective varieties over a field $\mathbb k$ and of finite dimensional $\mathbb k$-algebras of finite global dimension satisfy the conditions in Setup \ref{Set:fintype}.

\begin{lemma}[Bondal--Kapranov {\cite[Corollary to Theorem 2.10 \& Corollary 3.5]{bondal-kapranov}}]
    Let $\DD$ be a category as in Setup \ref{Set:fintype}. Assume that $\DD$ admits a full exceptional sequence. Then $\DD$ has a Serre functor $\bS_\DD$.
\end{lemma}

\begin{definition}\label{def:geom-type}
     Let $\DD$ be a triangulated category as in Setup \ref{Set:fintype} with a full exceptional sequence $\E:=(\sE_1, \ldots, \sE_n)$. Let $k \in \Z$. The \emph{helix of type $(n,k)$  generated by $\sE$} is an infinite sequence $\mathbb{H}=(\sE_i)_{i \in \Z}$ extending $\sE$ by setting 
    \begin{align}
        \sE_{i-n}:=\bS_\DD(\sE_{i})[1-k] \qquad \quad \text{ for all } \ i \in \Z.
    \end{align}
    A helix $\H$ is called \emph{geometric} if $\Hom^l(\sE_i, \sE_j)=0$ for all $l \neq 0$ and all $i<j$. {In this case, we also call the exceptional sequence $\E$ \emph{geometric (of type $(n,k)$)}.} 
\end{definition}

We are particularly interested in geometric helices of type $(n, k)$ that are generated by special exceptional sequences that can be partitioned into $k$ ``blocks'' of pairwise orthogonal objects:  

\begin{definition} \label{D:d-block exceptional sequence}
    Let $k \in \Z_{>0}$ and let $\TT$ be a triangulated category. For each $1 \leq i \leq k$ let $\E_i=(\sE_{i1}, \ldots, \sE_{il_i})$ be a sequence of pairwise orthogonal exceptional objects, i.e. 
    \begin{align}
        \Hom_\TT(\sE_{ij}, \sE_{im}[s])=0 \text{ for all }  s \in \Z \text{ and all } j\neq m.
    \end{align}
 If $\E:=(\E_1, \ldots, \E_k)$ forms an exceptional sequence, it is called \emph{$k$-block exceptional sequence}. 
  \end{definition}

\begin{example}\label{ex:p1xp1}
    Let $X=\P^1 \times \P^1$. Then $\left(\OO, \OO(0, 1), \OO(1, 0), \OO(1, 1)\right)$ is a full strong $3$-block exceptional sequence. Indeed, $\E_1=\OO$, $\E_2=(\OO(0, 1), \OO(1, 0))$, $\E_3=\OO(1, 1)$. The $(4, 3)$-helix $\H$ generated by it is
    \begin{align}
        \ldots \OO(-2, -2), \OO(-2, -1), \OO(-1, -2), \OO(-1, -1), \OO, \OO(0, 1), \OO(1, 0), \OO(1, 1), \OO(2, 2), \ldots 
    \end{align}
    since $\bS_X=- \otimes \omega_X [\dim X]=- \otimes \mathcal{O}(-2,-2)[2]$ is the Serre functor. 
\end{example}

\begin{lemma}\label{L:dblockgldim}
Let $\DD$ be a triangulated category as in Setup \ref{Set:fintype}.
    Let $k \in \Z_{>0}$ and let $\E:=(\sE_1, \ldots, \sE_m)$ be a strong $k$-block exceptional sequence with endomorphism algebra $E=\End(\oplus_{i=1}^m \sE_i)$. Then $\gldim E \leq k-1$.
\end{lemma}
\begin{proof}
    The ``$k$-block condition'' implies that the quiver of $A$ has paths of length at most $k-1$. It is well-known that this implies $\gldim A \leq k-1$, cf. e.g. \cite[Corollary 2.17]{HuberKalck}.
\end{proof}

For geometric helices of type $(n, n)$ (i.e. those generated by $n$-block sequences, where each block consists of a single object) the following result is due to Bondal \& Polishchuk, \cite{bondal-polishchuk}. 

\begin{proposition}\label{P:dBlockGeometricRed}
        Let $l<k \in \Z_{>0}$ and let $\E:=(\E_1, \ldots, \E_k)$ be a  $k$-block exceptional sequence in a triangulated category $\DD$ as in Setup \ref{Set:fintype}. Let $1\leq i_1 < \cdots < i_l \leq k$.
        
        If the helix of type $(n, k)$ generated by $\E$ is geometric, then  the helix of type $(s, l)$ generated by the $l$-block exceptional sequence $(\E_{i_1}, \ldots, \E_{i_l})$ is also geometric, where $s$ is the number of pairwise non-isomorphic objects in the sequence $(\E_{i_1}, \ldots, \E_{i_l})$.
\end{proposition}
\begin{proof}
One can either adapt the proof of Bondal \& Polishchuk \cite[Prop. 2.6]{bondal-polishchuk} by replacing mutations by ``block mutations'', cf. e.g. \cite{BridgelandStern}. Or first  mutate the blocks that should be removed to the end of the sequence. This yields another geometric exceptional sequence by    \cite[Theorem 5.3.]{BridgelandStern}. Then Lemma \ref{L:GeomHelicesandhigherreprinfinite} shows that \cite[Theorem 2.3]{Hanihara2} can be applied iteratively to remove these blocks.
\end{proof}

\begin{remark}
    The statement in Proposition \ref{P:dBlockGeometricRed} can fail if only part of a block is removed. Indeed, let

    \begin{equation*}\begin{tikzpicture}[description/.style={fill=white,inner sep=2pt}]

    \matrix (n) [matrix of math nodes, row sep=2em,
                 column sep=2em, text height=1.5ex, text depth=0.25ex,
                 inner sep=2pt, nodes={inner xsep=0.3333em, inner
ysep=0.3333em}] at (0, 0)
    {   1  && 2 & \cdots & n \\
    &&& n+1 \\};

    \draw[->] ($(n-1-1.east) + (-1.7mm,-2mm)$) to  ($(n-2-4.west) + (2.3mm,2mm)$);

    \draw[->] ($(n-1-3.east) + (-1.7mm,-2mm)$) to  ($(n-2-4.north) + (0mm,0mm)$);

\draw[->] ($(n-1-5.west) + (+1.7mm,-2mm)$) to ($(n-2-4.east) + (-1mm,3mm)$);

\node at (-4, -0.2) {$Q:=S_{n}=$}; 

\end{tikzpicture}\end{equation*}

    The derived category $\DD:=\Db(\mathbb{k}Q\mbox{-}\mod)$ has a full strong $2$-block\footnote{One block consists of $n$ objects and the other block of $1$ object.} exceptional sequence consisting of indecomposable projective $\mathbb{k}Q$-modules. 
    By Lemma \ref{L:GeomHelicesandhigherreprinfinite}, this sequence is geometric if and only if $n\geq 4$ if and only if $S_{n}$ is a non-Dynkin quiver.
    In particular, in this case, the geometricity of the helix is not preserved if and only if we remove $n> t \geq n-3$ objects from the larger block.
\end{remark} 

The following notion was introduced in \cite{HIO}, cf. \cite[Definition 1.14]{HIMO} for the version below.

\begin{definition}\label{D:higherreprinfinite}
    Let $E$ be a finite dimensional $\mathbb k$-algebra with $\gldim E < \infty$. Let  $\bS_E$ be the Serre functor of $\Db(E)$. Let $n \in \Z_{\geq 0}$ and $\nu_{n}:=\nu_{n,E}:=\bS_E[-n]$.
    Then $E$ is  \emph{$n$-representation-infinite} if 
    \begin{align} \label{E:higher-repr-inf}
        \bH^i(\nu_{n,E}^{-j}(E))=0  \text{  for all } i \neq 0 \text{  and all } j\geq 0. 
    \end{align}
\end{definition}

\begin{remark} \label{R:GldimHigherHered}
    It is known that $\gldim E = n$ if $E$ is $n$-representation-infinite, cf. e.g. \cite[Lem. 1.4]{Tomonaga2510} combined with \cite[Definition 2.7]{HIO}. 
\end{remark}

The next result is well-known and, for example, implicit in works of Bondal--Polishchuk \cite{bondal-polishchuk} and Bridgeland--Stern \cite{BridgelandStern}.

\begin{lemma}\label{L:GeomHelicesandhigherreprinfinite}
    Let $\E:=(\sE_1, \ldots, \sE_{{m}})$ be a full strong exceptional sequence in a triangulated category $\DD$ as in Setup \ref{Set:fintype}.
    The following conditions are equivalent
    \begin{itemize}
    \item[(a)] The helix of type $({m}, k)$ generated by $\E$ is \emph{geometric}.
    \item[(b)] The finite dimensional $\mathbb k$-algebra $E:=\End_\DD(\oplus_{i=1}^{{m}} \sE_i)\op$ is $(k-1)$-representation-infinite.
    \end{itemize}
\end{lemma}
\begin{proof}
Since $\E$ is an exceptional sequence, the quiver of $E$ has no oriented cycles, which shows that $\gldim E<\infty$ holds.
By assumption, $T:=\oplus_{i=1}^m \sE_i$ is a tilting object in $\DD$. By work of Keller, this yields a triangle equivalence $\DD \cong \Db(E)$ sending $T$ to $E$. Using this and the fact that triangle equivalences commute with Serre functors, we see that
\begin{align*}
 \Hom_\DD^i(T, (\bS^{-1}_\DD[k-1])^j(T))=0 \ \text{  for all } i \neq 0 \text{  and all } j\geq 0  &\Longleftrightarrow \\
\Hom_E^i(E, \nu_{(k-1),E}^{-j}(E))=0 \ \text{  for all } i \neq 0 \text{  and all } j\geq 0 
&\Longleftrightarrow \eqref{E:higher-repr-inf}
\end{align*}
This shows that (a) implies condition \eqref{E:higher-repr-inf} for $E$ and thus (a) implies (b).   
For the converse, we have to show $\Hom(\sE_i, \sE_j [l])=0$ for all $l \neq 0$ and all $i<j$. This follows since we can replace $T$ by $(\bS_\DD[1-k])^s(T)$ for any $s\in \Z$ (by definition, this generates the same helix as $T$ (up to reindexing)). In this way, we can assume without loss of generality that $T$ contains $\sE_i$ as a direct summand. The claim follows from the discussion above since (by definition of a helix) $\sE_j$ is then a direct summand of $(\bS^{-1}_\DD[k-1])^j(T)$ for some $j \geq 0$.
\end{proof}

Our main source of geometric exceptional sequences are certain $k$-block exceptional sequences via the following known results.

\begin{proposition}\label{prop:BHI}
    Let $Z$
    be a smooth projective variety of dimension $d$. Let $\E:=(\sE_1, \ldots, \sE_m)$ be a full strong exceptional sequence of \emph{sheaves} in $\Db(Z)$ with endomorphism algebra $E=\End(\oplus_{i=1}^m \sE_i)\op$.    
    Then the following conditions are equivalent:
    \begin{enumerate}[label=(\alph*)]
        \item $\gldim E\leq d.$
        \item The helix of type $({m}, d+1)$ generated by $\E$ is \emph{geometric}.
    \end{enumerate}
    In particular, any full strong $(d+1)$-block exceptional sequence of \emph{sheaves} in $\Db(Z)$ generates a geometric helix of type $({m}, d+1)$.
\end{proposition}
\begin{proof}
  The first part follows from \cite[Lemma 1.4 and Proposition 1.5]{Tomonaga2510}, which is attributed to Buchweitz, Hille \& Iyama. Indeed, Proposition 1.5 (3) in loc. cit. translates into condition (b), since $(\mathsf{S}_Z[1-(d+1)])^{-1}=-\otimes_Z\omega_Z^{-1}[d-d]=-\otimes_Z\omega_Z^{-1}$, cf. also the proof of Lemma \ref{L:GeomHelicesandhigherreprinfinite} (b) $\Rightarrow$ (a). 
  
  The second part follows from the implication (a) $\Rightarrow$ (b) together with Lemma \ref{L:dblockgldim}.
\end{proof}

\begin{example}
    The helix $\H$ of type $(4,3)$ in Example \ref{ex:p1xp1} is generated by a full strong  $3$-block exceptional collection of sheaves on the surface $\P^1 \times \P^1$. Hence, it is geometric, by Proposition \ref{prop:BHI}.

    By Proposition \ref{P:dBlockGeometricRed}, the $2$-block exceptional sequences $\E:=(\mathbb{E}_1, \E_2), \E':=(\mathbb{E}_2, \E_3)$ and $\E'':=(\mathbb{E}_1, \E_3)$ are also geometric. This can also be seen directly via tilting theory and Lemma \ref{L:GeomHelicesandhigherreprinfinite} -- the corresponding algebras are of the form $\mathbb{k}Q$ for a non-Dynkin quiver $Q$, hence $1$-representation infinite. 
\end{example}

\begin{remark}
    Not all geometric helices of type $(m, d+1)$ come from $(d+1)$-block exceptional sequences. 
    For example, the blow-up $Z$ of $\P^2$ in a point admits a full geometric sequence of type $(4, 3)$ consisting of line bundles, cf. Example \ref{ex:del-pezzo}. This sequence is not a $3$-block exceptional sequence. In fact, there does not exist a full $3$-block exceptional sequence of sheaves on $Z$ according to \cite[Remark in Section 3.5]{karpov-nogin}.

    Here is a ``representation theoretic'' argument to show that there does not exist a full strong $3$-block exceptional collection $\E$ on $Z$ such that one block consists of a single line bundle $\sL$\footnote{Such full string block exceptional sequences exist in most of our examples, which simplifies our arguments in these cases.}. Indeed, after tensoring $\E$ with $\sL^{-1}$ we may assume that $\sL=\mathcal{O}_Z$. It follows from \cite[Proposition 6.1]{King} (cf. also Example \ref{ex:del-pezzo}), that 
    \begin{align}
        \mathcal{O}_Z^\perp \cong \Db(\mathbb{k}Q_Z)
    \end{align}
    where $Q_Z$ is the following quiver
    \begin{equation*}
\small
\begin{tikzpicture}[description/.style={fill=white,inner sep=2pt}]
 \matrix (n) [matrix of math nodes, row sep=2.5em,
                 column sep=2em, text height=1.5ex, text depth=1ex,
                 inner sep=2pt, nodes={inner xsep=0.3333em, inner
ysep=0.3333em}] at (-6, 0)
    {   (-1,-1)  && (0,-1) \\ 
     && (-1,0) \\
          };
    \draw[<-] ($(n-1-1.east) + (0,1mm)$)  to  node[fill=white, yshift=0.7mm, scale=0.8] [midway]{$x_{0}$}($(n-1-3.west) + (0mm,1mm)$) ;
    \draw[<-] ($(n-1-1.east) + (0,-1mm)$) to node[fill=white, yshift=-0.7mm, scale=0.8] [midway]{$x_{1}$} ($(n-1-3.west) + (0mm,-1mm)$);

 \draw[<-] ($(n-1-3.south) + (-1mm,0mm)$)  to  node[fill=white, xshift=0mm,  yshift=-0.5mm, scale=0.8] [midway]{$y$}($(n-2-3.north) + (-1mm,0mm)$);

  \draw[<-] ($(n-1-1.south) + (2.5mm,0mm)$)  to  node[fill=white, yshift=0mm, scale=0.8] [midway]{$z$}($(n-2-3.north) + (-2mm,0mm)$);

\end{tikzpicture}\end{equation*}
This implies that $\Db(\mathbb{k}Q_Z)$ does not admit a full strong $2$-block exceptional sequence $\F$. 

Indeed, $\End_Z(\F)\cong \mathbb{k}Q$, where $Q$ has paths of lengths at most $1$.  
 Since, by tilting, $\mathbb{k}Q_Z$ and $\mathbb{k}Q$ are derived equivalent, \cite[Corollary in Chapter I.5.7]{HappelBook} implies that $Q_Z$ and $Q$ are related by a sequence of reflections (in sinks or sources). This yields a contradiction, since $Q_Z$  and all its reflections have a path of length $2$.   
\end{remark}

\begin{lemma} \label{L:gldimdroptensor}
    Let $A=\mathbb{k}Q/I$ and $A'=\mathbb{k}Q'/I'$ be finite dimensional $\mathbb{k}$-algebras such that $Q$ and $Q'$ 
    both have  source vertices $s \in Q_0$ and $s' \in Q'_0$, with corresponding idempotents $e_s \in A$ and $e_{s'} \in A'$.
    
    If $\gldim A/e_s \leq \gldim A-1$ and $\gldim A'/e_{s'} \leq \gldim A'-1$, then  $\gldim (A \otimes_\mathbb{k} A'/e_s \otimes e_{s'}) \leq \gldim A \otimes_\mathbb{k} A'-1$. 
\end{lemma}
\begin{proof}
Let $S$ be a simple $(A \otimes_\mathbb{k} A'/e_s \otimes e_{s'})$-module. We can view it as an $A \otimes_\mathbb{k} A'$- module via $\pi\colon A \otimes_\mathbb{k} A' \to A \otimes_\mathbb{k} A'/e_s \otimes e_{s'}$, which is again simple and not supported at $e_s \otimes e_{s'}$. Every such simple $A \otimes_\mathbb{k} A'$-module arises from a simple $A/e_s \otimes_\mathbb{k} A'$ or a simple $A \otimes_\mathbb{k} A'/e_{s'}$-module via pullback along the ring epimorphisms $p\colon A \otimes_\mathbb{k} A' \to A/e_s \otimes_\mathbb{k} A'$ and $p'\colon A \otimes_\mathbb{k} A' \to A \otimes_\mathbb{k} A'/e_{s'}$. 

Since $s$ and $s'$ are source vertices, pulling back projective modules via $\pi$, $p$ and $p'$ yields projective $A \otimes_\mathbb{k} A'$- modules. Using this, minimal projective resolutions over the quotient algebra pull back to minimal projective resolutions over $A \otimes_\mathbb{k} A'$.  In particular, if $S$ is not supported on $e_s \otimes 1$ 
\begin{align}
    \projdim_{A \otimes_\mathbb{k} A'/e_s \otimes e_{s'}} S = \projdim_{A \otimes_\mathbb{k} A'} S = \\
    \projdim_{A/e_s \otimes_\mathbb{k} A'} S \leq \gldim A/e_s \otimes_\mathbb{k} A' \\= \gldim A/e_s + \gldim A' \leq \gldim A -1 + \gldim A' = \gldim A \otimes_\mathbb{k} A' -1
\end{align}
and analogously if $S$ is not supported on $1 \otimes e_{s'}$. This proves the claim.
\end{proof}

\begin{lemma} \label{L:productsofgeom}
    Let $Z_1, Z_2$ be very strong Fano
    varieties. Then their product $Z=Z_1 \times Z_2$ is a very strong Fano
    variety.

    Let $\E_i$ be the geometric exceptional sequence on $Z_i$, in the notation of Definition \ref{def:verystrong-Fano}. 
    If $\E_1$ consists of vector bundles and $\E_2$ consists of sheaves, then there is a geometric exceptional sequence $\E$ on $Z$ as in Definition \ref{def:verystrong-Fano} that consists of sheaves. If both $\E_1$ and $\E_2$ consist of vector bundles, then $\E$ consists of vector bundles.
\end{lemma}
\begin{proof} 
Remark \ref{rmk:projnormal-product} shows that $Z$ satisfies Definition \ref{def:verystrong-Fano} \eqref{def:verystrong-Fano1}. 

Let $d_i-1=\dim Z_i$.
Let $\E_i=(\sE_{i1}, \ldots, \sE_{i(m_i+1)}=\OO_{Z_i})$ be the geometric exceptional sequences on $Z_i$ of type $(m_i+1, d_i)$ as in  Definition \ref{def:verystrong-Fano}.

For $\sH_{ij} \in \Db(Z_i)$, we have (cf. e.g. \cite[Proof of Lemma 5.2]{Uehara}),
\begin{align} \label{eq:HomBox}
  \Hom_Z(\sH_{11} \boxtimes \sH_{21}, \sH_{12} \boxtimes \sH_{22}[s]) =\bigoplus_{k+l=s} \Hom_{Z_1}( \sH_{11}, \sH_{12} [k]) \otimes_{\mathbb{k}}  \Hom_{Z_2}( \sH_{21}, \sH_{22} [l]) 
\end{align}
This implies that the exterior products $\sE_{1i}\boxtimes\sE_{2j}$ form again a full strong exceptional sequence $\E=\E_1 \boxtimes \E_2$.

By \eqref{eq:HomBox}, there is an algebra isomorphism $\End_Z(\E_1 \boxtimes\E_2) \cong \End_{Z_1}(\E_1) \otimes_{\mathbb{k}} \End_{Z_2}(\E_2)$. Since the $\E_i$ are geometric, the algebras $E_i=\End_{Z_i}(\E_i)\op$ are $(d_i-1)$-representation-infinite by Lemma \ref{L:GeomHelicesandhigherreprinfinite}. Hence, $\End_Z(\E_1 \boxtimes\E_2)\op \cong E_1 \otimes_{\mathbb{k}} E_2$ is $(d_1+d_2 -2)$-representation-infinite by \cite[Theorem 2.10]{HIO}. By Lemma \ref{L:GeomHelicesandhigherreprinfinite}, $\E$ is a geometric exceptional sequence of type $(m+1, d=\dim Z +1)=((m_1+1)(m_2+1), (d_1-1)+(d_2-1)+1)$ with last term $\mathcal{O}_Z=\mathcal{O}_{Z_1} \boxtimes \mathcal{O}_{Z_2}$. 

In order to show that $Z$ is strongly Fano, it thus remains to show that the restriction $\L$ of $\E$ to $\OO_Z^\perp$ is geometric of type $(m, d-1)$. Let $\L_i$ be the restrictions of $\E_i$ to 
$\OO_{Z_i}^\perp$, which are geometric of type $(m_i, d_i-1)$ since $Z_i$ are strongly Fano by assumption.

The term $\mathcal{O}_Z$ of $\E$ corresponds to an idempotent $e:=e_{s_1} \otimes e_{s_2}$ in  
$E_1 \otimes_{\mathbb{k}}E_2=:E $ where (since we take opposite algebras) the $s_i$ are sources in the quivers of $E_i$. Restricting endomorphisms of $\E_i$ to the subsequence $\L_i$ yields an isomorphism $E_i/e_{s_i} \cong \End(\L_i)\op$. Combining this with Lemma \ref {L:GeomHelicesandhigherreprinfinite} our geometricity assumptions on $\E_i$ and $\L_i$ and Remark \ref{R:GldimHigherHered}, shows that $\gldim E_i/e_{s_i} = \gldim E_i-1$.

Together with Lemma \ref{L:gldimdroptensor} this shows that $\gldim E/e \leq (d_1 -1) + (d_2 -1) -1 =d-2$. By \cite[Theorem 2.3]{Hanihara2}, $E/e$ is $(d -2)$-representation infinite\footnote{Here, we use that for algebras $A=\mathbb{k} Q/I$ without oriented cycles in their quivers $Q$ (e.g. endomorphism algebras of exceptional sequences) \cite[Assumption 2.2(2)]{Hanihara2}
\begin{align}
\mathrm{inj.dim}_A A/e\leq d-1
\end{align}
follows from 
\begin{align}
\gldim A/e\leq d-1
\end{align}
where $e$ is an idempotent corresponding to a source of the quiver $Q$.
}
and thus $\End(\L)\op \cong E/e$ is also $(d-2)$-representation-infinite.  Lemma \ref{L:GeomHelicesandhigherreprinfinite} shows that $\L$ is a geometric exceptional sequence of type $(m, d-1)$.

The last statement, follows since $\sH \boxtimes \sH'$ is a sheaf if $\sH$ is a vector bundle and $\sH'$ is a sheaf. And it is a vector bundle if both $\sH$ and $\sH'$ are vector bundles.
\end{proof}

We will use the following result in Theorem \ref{thm:projcone-split}.

\begin{corollary}\label{cor:geometric-subcollection}
    Let $Z$ be a smooth projective variety of dimension $d$. Let $\E:=(\sE_1, \ldots, \sE_m)$ be a full geometric exceptional sequence of \emph{sheaves} of type $(m,d)$ in $\mathcal{O}_Z^\perp$.
    If $\E':=(\sE_1, \ldots, \sE_m, \mathcal{O}_Z)$ is strong, then $\E'$ is geometric of type $(m+1,d+1)$. 
\end{corollary}
\begin{proof}
    Let $E=\End_Z(\oplus_{i=1}^m \sE_i)\op$ and $E'=\End_Z((\oplus_{i=1}^m \sE_i) \oplus \mathcal{O}_Z)\op$. By Remark \ref{R:GldimHigherHered} combined with Lemma \ref{L:GeomHelicesandhigherreprinfinite} and our assumptions on $\E$, we have $\gldim E \leq d-1$.  This shows \begin{align}
      \projdim_{E'}\, S \leq 1 + \projdim_E\, \mathrm{rad}(P_{E'}(S)) \leq 1+(d-1)=d  
    \end{align}
    where $S$ is a simple $E'$-module with projective cover $P_{E'}(S)$. This shows $\gldim E' \leq d$ which, in combination with Proposition \ref{prop:BHI}, completes the proof.
\end{proof}

\begin{definition}\label{D:HigherpreprojAlg} 
    Let $E$ be an $n$-representation infinite algebra. The $(n+1)$-\emph{preprojective algebra} $\Pi_{n+1}(E)$ of $E$ is defined as the tensor algebra of the $E$-bimodule $\Hom_E(E, \nu_n^{-1}(E)) $ over $E$:
    \begin{align}
      \Pi_{n+1}(E):=T_E \Hom_E(E, \nu_n^{-1}(E))   
    \end{align}
\end{definition}

\begin{remark}\label{R:careful}
This definition of the $(n+1)$-preprojective algebra coincides with \cite[Definition 2.11]{IyamaOppermann}. Indeed,
 there are isomorphisms of $E$-$E$-bimodules 
 \begin{align}
  \Ext^n_E(DE, E) = \Hom_{\Db(E)}(\bS(E), E[n]) \cong  \Hom_E(E,\nu_n^{-1}E).
  \end{align}
  The last isomorphism is given by the inverse Serre functor $\bS^{-1}$ and recalling that $\nu_n^{-1}:=\bS^{-1}[n]$ and
  $\Hom_E(E,\nu_n^{-1}E) \cong \Hom_{\Db(E)}(E,\nu_n^{-1}E)$, since $\nu_n^{-1}E$ is a module, since $E$ is $n$-representation infinite.
\end{remark}

\begin{remark}\label{rem:rolled-up}
    Assume $E$ is the opposite algebra of the endomorphism algebra of a full geometric  exceptional sequence $\mathbb{E}$ of type $(n, d)$ in a triangulated category $\cd$ as in Setup \ref{Set:fintype}. Then $E$ is $(d-1)$-representation-infinite by Lemma \ref{L:GeomHelicesandhigherreprinfinite} and its $d$-preprojective algebra $\Pi_d(E)$ is known to be isomorphic to the rolled-up helix algebra of the helix $\H$ generated by $\E$ as defined in Bridgeland--Stern \cite[Section 3.2.]{BridgelandStern}. 
    
    Indeed, firstly, there is an isomorphism of algebras
    $\Pi_d(E) \cong  \bigoplus_{i\geq 0} \Hom_E(E, \nu_{d-1}^{-i}(E)) $, cf. e.g. \cite[Lemma 2.13]{IyamaOppermann}.
    Note that the algebra on the right hand side is the endomorphism algebra of $E$ in the orbit category of $\Db(E)$ with respect the autoequivalence $\nu_{d-1}^{-1}$. In particular, if $f \in \Hom_E(E, \nu_{d-1}^{-j} E)$ and $g \in \Hom_E(E, \nu_{d-1}^{-i} E)$, then $fg:= \nu_{d-1}^{-i}(f)g \in \Hom_E(E, \nu_{d-1}^{-i-j} E)$. By tilting there is an algebra isomorphism $\bigoplus_{i\geq 0} \Hom_E(E, \nu_{d-1}^{-i}(E)) \cong \bigoplus_{i\geq 0} \Hom_\cd(\E, \nu_{d-1}^{-i}(\E))$.

    Next, consider the \emph{helix algebra} as studied by Bondal \& Polishchuk \cite{bondal-polishchuk}
    \begin{align}
       A(\H) = \bigoplus_{k \geq 0} \prod_{j-i=k} \Hom_\cd(\sE_i, \sE_j) 
    \end{align}
     of the helix $\H = (\sE_i)_{i \in \Z}$ of type $(n, d)$ generated by $\E=(\sE_1, \ldots, \sE_n)$. Applying powers of the Serre functor $\bS_\cd$ induces a $\Z$-action on $A(\H)$. The \emph{rolled-up helix algebra} $B(\H)$  of $\H$ is the subalgebra of elements in $A(\H)$ that are invariant under this action. One can check that
     \begin{align}
         \Hom_\cd(\sE_i, \nu_{d-1}^{-l}(\sE_j))  &\to \prod_{b-a=ln+j-i} \Hom_\cd(\sE_a, \sE_b) \\
         f &\mapsto (\bS_{\cd}^m(f))_{m \in \Z}
          \end{align}
     induces an algebra isomorphism
     \begin{align}
     \bigoplus_{l\geq 0} \Hom_\cd(\E, \nu_{d-1}^{-k}(\E))  \to B(\H):= \left(\bigoplus_{k \geq 0} \prod_{j-i=k} \Hom_\cd(\sE_i, \sE_j)\right)^\Z.  
     \end{align}
     This finishes the argument showing that the $(d-1)$-preprojective algebra $\Pi_{d-1}(E)$ is isomorphic to the rolled up helix algebra $B(\H)$.
\end{remark}

\subsection{The split case}\label{subsec:split-case}
The results in this section apply to $A=\End(\sG)$ provided 
 the assumptions of Proposition \ref{prop:algebra-extensions} hold and if, in addition, $\varphi$ in \eqref{eq:ses-G} is a split $\mathbb k$-algebra epimorphism (this is for example satisfied for projective cones over certain Fano varieties, cf. Theorem \ref{thm:projcone-split}). 

\medskip

Let $\varphi\colon A \twoheadrightarrow E$ be a $\mathbb k$-algebra epimorphism with kernel $I:=\ker \varphi$. If $I^2=0$, there is a natural $E$-bimodule structure on the two-sided ideal $I \subset A$. Indeed, for $\epsilon \in E$ and $x \in I$, set $\epsilon x:=\alpha x$, for some $\alpha\in \varphi^{-1}(\epsilon)$. One can check that this does not depend on the choice of $\alpha$, using the assumption $I^2=0$. The right action of $E$ on $I$ is defined analogously. Moreover, these actions define an $E$-bimodule structure, since $I$ is a two-sided ideal in $A$ and hence an $A$-bimodule.  

Assume that $\varphi$ splits, i.e. there is a $\mathbb k$-algebra homomorphism $\sigma\colon E \to A$, such that $\sigma\varphi=\id_E$. In particular, we have an isomorphism of vector spaces 
 $ E \oplus I \xrightarrow{\psi:=(\sigma,  \ \mathrm{incl.})}A$. We can use this isomorphism to define a $\mathbb k$-algebra structure on $E \oplus I$ such that $\psi$ becomes an $\mathbb k$-algebra homomorphism. Indeed, using that $\sigma$ is a $\mathbb k$-algebra homomorphism and that $I^2=0$, one can check that this $\mathbb k$-algebra multiplication on $ E \oplus I$ is given by
\begin{equation} \label{eq:split multiplication}
    (\epsilon, x)(\epsilon', x')=(\epsilon \epsilon', x \epsilon' + \epsilon x')
\end{equation}
where we have the multiplication in $E$ in the first component and the $E$-bimodule multiplication on $I$ in the second component.

Let $M$ be any $E$-bimodule. Then there is a $\mathbb k$-algebra structure on $E \oplus M$ by defining a multiplication analogous to \eqref{eq:split multiplication}.
If $\beta\colon M \to I$ is an isomorphism of $E$-bimodules, then there is an induced isomorphism of $\mathbb k$-algebras $E \oplus M \xrightarrow{(\id_E, \beta)} E \oplus I$. 

Let $T_E(M):=E \oplus M \oplus (M \otimes_E M) \oplus \ldots $ be the tensor algebra. Note that $M^{\geq 2}:=(M \otimes_E M) \oplus (M \otimes_E M \otimes_E M) \oplus \ldots$ defines a two-sided ideal in $T_E(M)$ (which is generated by $M \otimes_E M$).
By construction, there is a  $\mathbb k$-algebra isomorphism 
\begin{align}
    T_E(M)/M^{\geq 2} \to E \oplus M
\end{align}

Summing up we have explained the following result.

\begin{lemma}\label{L:Quotient-Tensor-Alg}
    Let $A \xrightarrow{\varphi} E$ be an epimorphism of $\mathbb k$-algebras, such that $(\ker \varphi)^2=0$. In particular, $I:=\ker \varphi$ has a natural $E$-bimodule structure.
    Let $M$ be an $E$-bimodule fitting into a short exact sequence
    \begin{align}
    0 \to M \xrightarrow{i} A \xrightarrow{\varphi} E \to 0,
\end{align}

such that $i\colon M \to \im \ i = \ker \varphi =I$ is an $E$-bimodule isomorphism. If the $\mathbb k$-algebra homomorphism $\varphi$ splits, then there is an isomorphism of $\mathbb k$-algebras
\begin{align}
    A \xleftarrow{\cong} T_E(M)/M^{\geq 2}
\end{align}
\end{lemma}

In combination with Proposition \ref{prop:algebra-extensions}, Lemma \ref{L:GeomHelicesandhigherreprinfinite} and Definition \ref{D:HigherpreprojAlg}, this implies the following result.

\begin{corollary}\label{C:QuotientPreproj}
    Assume that the conditions in Proposition \ref{prop:algebra-extensions} hold and let $E=\End(\sE)$. Assume additionally, that the $\mathbb k$-algebra homomorphism $\varphi$ in \eqref{eq:ses-G} splits.
    Then  
    \begin{align}\label{E:TensorModSquare}
    A \xleftarrow{\cong} \Pi_{d-1}(E)/M^{\geq 2},
\end{align}
is an isomorphism of $\mathbb k$-algebras,
where $\Pi_{d-1}(E)=T_E(M)$ is the $(d-1)$-preprojective algebra of $E$, i.e. the tensor algebra of the $E$-bimodule $M=\Hom(E, \nu_{d-2}^{-1}(E))$.
\end{corollary}
\begin{proof}
We apply Lemma \ref{L:Quotient-Tensor-Alg} to the homomorphism $\varphi$ in \eqref{eq:ses-G}. 
 Let $M=\Hom(E, \nu_{d-2}^{-1}(E))$. In  Proposition \ref{prop:algebra-extensions} (2) we show that 
$M \cong \ker \varphi$ as $E$-bimodules. By the assumptions on $E$ and Lemma \ref{L:GeomHelicesandhigherreprinfinite}, $E$ is $(d-2)$-representation infinite and thus $T_E(M)$ is the $(d-1)$-preprojective algebra $\Pi_{d-1}(E)$ of $E$ (cf. Definition \ref{D:HigherpreprojAlg}). 
\end{proof}

In many cases, Corollary \ref{C:QuotientPreproj} reduces the description of $A=\End(\sG)$ to the description of the preprojective algebras $\Pi_{d-1}(E)$.

\subsubsection{Description of some $n$-preprojective algebras}

We explain how to describe $n$-preprojective algebras for $n\leq 2$ (using known constructions in representation theory) and for all truncated Beilinson algebras based on work of Hanihara \cite{Hanihara2} (there is an alternative approach in the latter case in the setting of the appendix).

We start with the semi-simple case, i.e. $n=0$.

\begin{lemma}\label{L:preprojn0}
    Let $E=\mathbb{k}^t=\mathbb{k} \times \cdots \times \mathbb{k}$. Then $E$ is $0$-representation infinite and there is an isomorphism of $\mathbb k$-algebras $\Pi_1(E) \cong \mathbb{k}[x]^t = \mathbb{k}[x] \times \cdots \times \mathbb{k}[x]$
    
\end{lemma}
\begin{proof}
    The Serre functor $\mathsf{S}_E$ is the identity functor and hence so is $\nu_{0,E}=\mathsf{S}_E[-0]$. This implies that $E$ is $0$-representation infinite and also that $\Pi_1(E) \cong \mathbb{k}[x]^t$.
\end{proof}

The following relates to categorical absorption of nodal surface singularities as discussed in \cite{kks, ks}.

\begin{corollary}\label{cor:nodal-2d}
    In the setting and notations of Corollary \ref{C:QuotientPreproj}, assume that $d=2$, then there is an isomorphism of $\mathbb{k}$-algebras $A \cong \mathbb{k}[x]/(x^2)$.
\end{corollary}
\begin{proof}
   Combining Corollary \ref{C:QuotientPreproj} with Lemma \ref{L:preprojn0} shows $A \cong (\mathbb{k}[x]/(x^2))^t$. In the geometric setting of Proposition \ref{prop:algebra-extensions}, we have $t=1$. Indeed, by Corollary \ref{cor:veriderloc-negative-dg} (which is applicable since $\pi_*\colon \Db(Y) \to \Db(X)$ is a Verdier localization by e.g. \cite[Lemma 2.32]{pavic-shinder} since we are in dimension $2$ so the fibers of $\pi$ are at most $1$-dimensional), there is an admissible semiorthogonal decomposition
   \begin{align}
       \Db(X)=\langle \Db(A), \PP \rangle.
   \end{align}
   By work of Orlov, there is an equivalence of singularity categories $\Dsg(X)\cong \Dsg(A)$. This implies that $X$ has $t$ nodal singularities. So $t=1$ since we assume in Proposition \ref{prop:algebra-extensions} that $X$ has a unique singular point.
\end{proof}

The next case is $n=1$, which geometrically corresponds to singular $3$-folds $X$ and is well-studied in representation-theory.

\begin{theorem}\label{T:Preprojn1}
 Let $\mathbb k$ be a field and let $Q$ be a finite connected quiver that has no oriented cycles and has at least one arrow\footnote{Otherwise, we are in the situation of Lemma \ref{L:preprojn0}}. 

Then the path algebra $\mathbb{k}Q$ of $Q$ is $n$-representation infinite, if and only if $n=1$ and $Q$ is not a Dynkin quiver.

Moreover, in this case, the $2$-preprojective algebra\footnote{Also, known as the preprojective algebra of $Q$} $\Pi_2(Q):=\Pi_2(\mathbb{k}Q)$ is isomorphic to $\mathbb{k}\overline{Q}/(\rho)$. Here, the quiver $\overline{Q}$ has the same vertices as $Q$ and for every arrow $\alpha\colon s(\alpha) \to t(\alpha)$ in $Q$
there is an additional arrow $\alpha^*\colon t(\alpha) \to s(\alpha)$ in the opposite direction. Finally, 
\begin{align}
    \rho=\sum_{\textrm{arrows } \alpha \text{ in } Q} \alpha \alpha^* - \alpha^* \alpha
\end{align}
  
\end{theorem}
\begin{proof}
The first part is classical, cf. e.g. \cite[Example 3.11 (a)]{IyamaICM}.

The second part can, for example, be found in \cite[Theorem A]{RingelPreprojective}.
\end{proof}

\begin{example}\label{E:prprojhered}
Let $Q$ be the quiver given by the black arrows. It is not of Dynkin type. Then $\overline{Q}$ is obtained from $Q$ by adding the red arrows:
    \begin{equation*}
\small
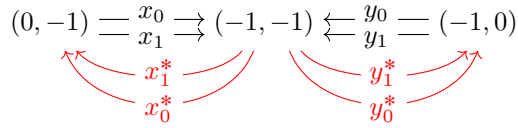
\begin{tikzpicture}[description/.style={fill=white,inner sep=2pt}]
 \matrix (n) [matrix of math nodes, row sep=2em,
                 column sep=2em, text height=1.5ex, text depth=1ex,
                 inner sep=2pt, nodes={inner xsep=0.3333em, inner
ysep=0.3333em}] at (-6, 0)
    {   (0,-1)  && (-1,-1) && (-1,0) \\
          };
    \draw[->] ($(n-1-1.east) + (0,1mm)$)  to  node[fill=white, yshift=0.7mm, scale=1] [midway]{$x_{0}$}($(n-1-3.west) + (0mm,1mm)$) ;
    \draw[->] ($(n-1-1.east) + (0,-1mm)$) to node[fill=white, yshift=-0.7mm, scale=1] [midway]{$x_{1}$} ($(n-1-3.west) + (0mm,-1mm)$);
 \draw[<-] ($(n-1-3.east) + (0,1mm)$)  to  node[fill=white, yshift=0.7mm, scale=1] [midway]{$y_{0}$}($(n-1-5.west) + (0mm,1mm)$) ;
    \draw[<-] ($(n-1-3.east) + (0,-1mm)$) to node[fill=white, yshift=-0.7mm, scale=1] [midway]{$y_{1}$} ($(n-1-5.west) + (0mm,-1mm)$);

    \draw[<-, red] ($(n-1-5.south west) + (5mm,-0mm)$) .. controls +(-3.7mm,-4.5mm)
and +(+3.7mm,-4.5mm) .. node[fill=white, scale=1, yshift=0.5mm] [midway]{$y_1^*$} ($(n-1-3.south east) + (-3mm,0mm)$);

\draw[<-, red] ($(n-1-5.south west) + (6.5mm,-0mm)$) .. controls +(-3.7mm,-9.5mm)
and +(+3.7mm,-9.5mm) .. node[fill=white, scale=1, yshift=-0.5mm] [midway]{$y_0^*$} ($(n-1-3.south east) + (-4.5mm,0mm)$);

   \draw[->, red] ($(n-1-3.south west) + (5mm,-0mm)$) .. controls +(-3.7mm,-4.5mm)
and +(+3.7mm,-4.5mm) .. node[fill=white, scale=1, yshift=0.5mm] [midway]{$x_1^*$} ($(n-1-1.south east) + (-3mm,0mm)$);

\draw[->, red] ($(n-1-3.south west) + (6.5mm,-0mm)$) .. controls +(-3.7mm,-9.5mm)
and +(+3.7mm,-9.5mm) .. node[fill=white, scale=1, yshift=-0.5mm] [midway]{$x_0^*$} ($(n-1-1.south east) + (-4.5mm,0mm)$);
\end{tikzpicture}
\end{equation*}
By multiplying with the idempotents at the different vertices of $\overline{Q}$, the ideal $(\rho)$ in Theorem \ref{T:Preprojn1} can be rewritten as 
\begin{align}
    (\rho)=(x_0^*x_0 +x_1^*x_1, x_0x_0^* + x_1x_1^* + y_0y_0^* + y_1y_1^*  , y_0^*y_0 +y_1^*y_1) 
\end{align}
\end{example}

\begin{corollary}\label{C:QuotientPreproj2}
    In the setting and notations of Corollary \ref{C:QuotientPreproj} and Theorem \ref{T:Preprojn1}, assume that $d=3$, then there is a quiver $Q$ without oriented cycles and such that all of its connected components are not of Dynkin type and there is an an isomorphism of $\mathbb{k}$-algebras \[A \cong \mathbb{k}\overline{Q}/((\rho) + R^2)\]
    where $R$ is the two-sided ideal generated by all arrows $\alpha^* \in \overline{Q}$.
\end{corollary}

\begin{example}
    Applying Corollary \ref{C:QuotientPreproj2} in Example \ref{E:prprojhered} shows that $A$ has the relations defined by $\rho$ and additionally any path containing at least two red arrows is zero in $A$. 
\end{example}

We continue with $n=2$, which geometrically corresponds to singular fourfolds $X$. 

\begin{theorem}\label{thm:3preproj}
    Let $E\cong \mathbb{k}Q/I$ be a finite dimensional $\mathbb{k}$-algebra that is $2$-representation infinite. 
    
    For each pair of vertices $i, j$ in $Q$, let $r_{1}(ij), \ldots, r_{m_{ij}}(ij)$ be a $k$-basis of $D\Ext^2_E(S_i, S_j)$, where $S_i$ denotes the simple $E$-module at vertex $i$. Let $\widetilde{Q}$ be the quiver obtained from $Q$ by adding arrows $\rho_r\colon j \to i$ for each pair of vertices $i, j$ and all elements $r \in \{r_l(ij)\mid 1 \leq l \leq m_{ij}\}=:\mathcal{R}(ij).$ Set $\mathcal{R}:= \bigcup_{i, j} \mathcal{R}(ij)$ 

    Then the $3$-preprojective algebra $\Pi_3(E)$ is isomorphic to the Jacobian-algebra \begin{align}
        J(\widetilde{Q},W) :=\mathbb{k}\widetilde{Q}/(\partial_a W \mid a \text{ arrow in } \widetilde{Q}),
            \end{align}
   where $W$ is the potential given by
   \begin{align}
       W=\sum_{r \in \mathcal{R}} r\rho_r
   \end{align}
    
\end{theorem}
\begin{proof}
    By \cite[Section 6.9 \& Theorem 6.10]{KellerCY} applied to $c=0$ and $W'=0$, the derived $3$-prepreprojective algebra ${\bf\Pi}_3(E)$ of $E$ is quasi-isomorphic to the Ginzburg dg-algebra $\Gamma_3(\widetilde{Q}, W)$ of $(\widetilde{Q}, W)$. In particular, the zeroth cohomology of both algebras are isomorphic.
    By definition of $\Gamma_3(\widetilde{Q}, W)$, there is an isomorphism $H^0(\Gamma_3(\widetilde{Q}, W)) \cong J(\widetilde{Q}, W)$. Since $E$ is $2$-representation infinite, by definition, ${\bf\Pi}_3(E)$ is concentrated in degree $0$ and coincides with 
    the $3$-prepreprojective algebra ${\Pi}_3(E)$ of $E$, cf. e.g. \cite[Part (ii) of proof of Theorem 4.36]{HIO} together with Remark \ref{R:careful}. In particular,
     $H^0({\bf\Pi}_3(E)) \cong {\bf\Pi}_3(E) \cong \Pi_3(E)$. This completes the proof.
\end{proof}

\begin{corollary}\label{C:QuotientPreproj3}
    In the setting and notations of Corollary \ref{C:QuotientPreproj} and Theorem \ref{thm:3preproj}, assume that $d=4$, then there is a quiver $Q$ and a potential $W$ for $\mathbb{k}\widetilde{Q}$ such that there is an an isomorphism of $\mathbb{k}$-algebras \[A \cong \mathbb{k}\widetilde{Q}/((\partial_a W \mid a \text{ arrow in } \widetilde{Q}) + R^2)\]
    where $R$ is the two-sided ideal generated by all arrows $\rho_r \in \widetilde{Q}$ for $r \in \mathcal{R}$.
\end{corollary}

\begin{remark}\label{rmk:product-of-proj}
By work of Thibault \cite{Thibault} and Iyama \& Grant \cite{IG}, explicit descriptions of $(g+1)$-preprojective algebras $\Pi_{g+1}(E)$ in terms of quivers with potentials can also be obtained more generally for Koszul algebras $E=\End(\sE)$ of global dimension $g>2$. Which by Corollary \ref{C:QuotientPreproj} yields explicit descriptions of the split square-zero extensions $A$. 

This applies for example to anticanonical cones over products $Z=\P^{n_1} \times \cdots \times \P^{n_t}$ (we will see an explicit description for $Z=\P^{d-1}$ obtained using different methods below) and more generally to products $Z=Z_1 \times \cdots \times Z_n$ of Fanos $Z$ satisfying the assumptions of Theorem \ref{thm:projcone-split} and such that all $E_i=\End(\E_i)$ are Koszul -- indeed, tensor products of Koszul algebras are Koszul by \cite[Ex. 1.6]{MartinezVillaIntrotoKoszul} and truncations of Koszul algebras in sources of the quiver are again Koszul.

Note however, that not all our examples are Koszul, e.g. for $Z=\P^1 \times \Bl_p(\P^2)$, there are algebras $\End(\sE)$ with relations of length $3$ (cf. \cite{King}), so they cannot be Koszul.     \end{remark}

\begin{example}
Let $E=\mathbb{k}Q/I$ be given by the black quiver below. Then the $3$-preprojective algebra $\Pi_3(E)$ of $E$ is given by the following quiver $\widetilde{Q}$ consisting of red and black arrows
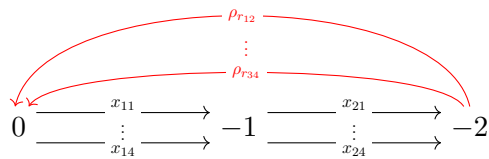
\begin{equation*}
\begin{tikzpicture}[description/.style={fill=white,inner sep=2pt}]
 \matrix (n) [matrix of math nodes, row sep=4em,
                 column sep=2em, text height=1.5ex, text depth=0.25ex,
                 inner sep=2pt, nodes={inner xsep=0.3333em, inner
ysep=0.3333em}] at (-5, 0)
    { 
    0 &&& -1 &&& -2 \\ 
      };

    \draw[->] ($(n-1-4.east) + (0,2mm)$)  to  node[scale=0.6, fill=white, yshift=1.7mm] [midway]{$x_{21}$}($(n-1-7.west) + (0mm,2mm)$) ;
    \draw[white] ($(n-1-4.east) + (0,0mm)$)  to  node[fill=white, scale=0.6, yshift=0.0mm] [midway]{$\bf\textcolor{black}{\vdots}$}($(n-1-7.west) + (0mm,0mm)$) ;
    \draw[->] ($(n-1-4.east) + (0,-2mm)$) to node[fill=white, scale=0.6, yshift=-1.7mm] [midway]{$x_{24}$} ($(n-1-7.west) + (0mm,-2mm)$);

     \draw[->] ($(n-1-1.east) + (0,2mm)$)  to  node[scale=0.6, fill=white, yshift=1.7mm] [midway]{$x_{11}$}($(n-1-4.west) + (0mm,2mm)$) ;
      \draw[white] ($(n-1-1.east) + (0,0mm)$)  to  node[fill=white, scale=0.6, yshift=0.0mm] [midway]{$\bf\textcolor{black}{\vdots}$}($(n-1-4.west) + (0mm,0mm)$) ;
    \draw[->] ($(n-1-1.east) + (0,-2mm)$) to node[fill=white, scale=0.6, yshift=-1.7mm] [midway]{$x_{14}$} ($(n-1-4.west) + (0mm,-2mm)$);

      \draw[white] ($(n-1-7.north west) + (5mm,0mm)$) .. controls +(-3.7mm,11mm)
and +(+3.7mm,11mm) .. node[fill=white, scale=0.6] [midway]{$\bf\textcolor{red}{\vdots}$}($(n-1-1.north east) + (-3mm,0mm)$);  

   \draw[->, red] ($(n-1-7.north west) + (3.8mm,-0mm)$) .. controls +(-3.7mm,16mm)
and +(+3.7mm,16mm) .. node[fill=white, scale=0.6] [midway]{$\rho_{r_{12}}$} ($(n-1-1.north east) + (-2.8mm,0mm)$);

      \draw[->, red] ($(n-1-7.north west) + (3mm,-0mm)$) .. controls +(-3.7mm,6mm)
and +(+3.7mm,6mm) .. node[fill=white, scale=0.6] [midway]{$\rho_{r_{34}}$}($(n-1-1.north east) + (-1mm,0mm)$);

\end{tikzpicture}\end{equation*}
where for $1\leq k<l\leq 4$, we have $r_{kl}:=r_{kl}(0 -2):=x_{2k}x_{1l}-x_{2l}x_{1k}$.

Hence, the potential is given by 
\begin{align}
W= \sum_{1\leq k<l\leq 4} r_{kl}\rho_{r_{kl}} 
= \sum_{1\leq k<l\leq 4}(x_{2k}x_{1l}-x_{2l}x_{1k}) \rho_{r_{kl}}
\end{align}
Taking cyclic derivatives of $W$ yields the following relations 
\begin{align}
    \partial_{\rho_{r_{kl}}}W &=r_{kl}=x_{2k}x_{1l}-x_{2l}x_{1k} \label{eq:R1}\\
    \partial_{x_{1j}}W&=\partial_{x_{1j}}\left(\sum_{1\leq k<j}
    (x_{2k}x_{1j}-x_{2j}x_{1k}) \rho_{r_{kj}}\right) + \partial_{x_{1j}}\left(\sum_{ j<l \leq 4}
    (x_{2j}x_{1l}-x_{2l}x_{1j}) \rho_{r_{jl}}\right) \\
    &=\sum_{1\leq k<j}
    x_{2k} \rho_{r_{kj}} +
    \sum_{ j<l \leq 4}
    -x_{2l} \rho_{r_{jl}} \label{eq:R2}\\
    \partial_{x_{2j}}W&=\sum_{1\leq k<j}
    -x_{1k} \rho_{r_{kj}} +
    \sum_{ j<l \leq 4}
    x_{1l} \rho_{r_{jl}} \label{eq:R3}.
\end{align}
By Theorem \ref{thm:3preproj}, the $3$-preprojective algebra $\Pi_3(E)$ is the quotient of $\mathbb{k}\widetilde{Q}$ by the ideal generated by all the relations\footnote{One can check that they generate the same ideal as the relations exhibited in Proposition \ref{P:Hanihara}.} \eqref{eq:R1}, \eqref{eq:R2}, \eqref{eq:R3}. By Corollary \ref{C:QuotientPreproj3}, the algebra $A$ is obtained from $\Pi_3(E)$ by imposing the additional relations that any paths containing any two of the $\rho$-arrows is zero.

\end{example}

We finish with an infinite family of $d$-representation infinite algebras for all $d\geq 1$. The following result is the special case $d=n$ of Hanihara's \cite[Theorem 4.12]{Hanihara2}.

\begin{proposition} \label{P:Hanihara}
Let $E=\End(\sE)$ be the truncated Beilinson algebra of $\P^{d-1}$. Then the $(d-1)$-preprojective algebra $\Pi_{d-1}(E)$ of $E$ is given by the following quiver $Q_d$ with relations
\begin{equation*}\begin{tikzpicture}[description/.style={fill=white,inner sep=2pt}]

    \matrix (n) [matrix of math nodes, row sep=2em,
                 column sep=2em, text height=1.5ex, text depth=0.25ex,
                 inner sep=2pt, nodes={inner xsep=0.3333em, inner
ysep=0.3333em}] at (-5, 0)
    {   1  && 2 && \cdots && d-1 \\
          };

    \draw[->] ($(n-1-1.east) + (0,2mm)$)  to node[yshift=-1.2mm, scale=0.8] [midway]{$\vdots$} node[yshift=1.7mm, scale=0.8] [midway]{$x_{11}$}($(n-1-3.west) + (0mm,2mm)$) ;
    \draw[->] ($(n-1-1.east) + (0,-2mm)$) to node[yshift=-1.7mm, scale=0.8] [midway]{$x_{1d}$} ($(n-1-3.west) + (0mm,-2mm)$);

    \draw[->] ($(n-1-3.east) + (0,2mm)$)  to node[yshift=-1.2mm, scale=0.8] [midway]{$\vdots$} node[yshift=1.7mm, scale=0.8] [midway]{$x_{21}$}($(n-1-5.west) + (0mm,2mm)$) ;
    \draw[->] ($(n-1-3.east) + (0,-2mm)$) to node[yshift=-1.7mm, scale=0.8] [midway]{$x_{2d}$} ($(n-1-5.west) + (0mm,-2mm)$);
    
       \draw[->] ($(n-1-5.east) + (0,2mm)$)  to node[yshift=-1.2mm, scale=0.8] [midway]{$\vdots$} node[yshift=1.7mm, scale=0.8] [midway]{$x_{(d-2)1}$}($(n-1-7.west) + (0mm,2mm)$) ;
    \draw[->] ($(n-1-5.east) + (0,-2mm)$) to node[yshift=-1.7mm, scale=0.8] [midway]{$x_{(d-2)d}$} ($(n-1-7.west) + (0mm,-2mm)$);

     \draw[->] ($(n-1-7.south west) + (5mm,-0mm)$) .. controls +(-3.7mm,-16mm)
and +(+3.7mm,-16mm) .. node[fill=white, scale=0.8] [midway]{$z_{(d-1)d}$} ($(n-1-1.south east) + (-3mm,0mm)$);

      \draw[->] ($(n-1-7.south west) + (3mm,-0mm)$) .. controls +(-3.7mm,-6mm)
and +(+3.7mm,-6mm) .. node[yshift=-3mm] [midway]{$\vdots$} node[yshift=-3.5mm, xshift=7mm, scale=0.7] [midway]{$z_{ij, \ i<j}$} node[fill=white, scale=0.8] [midway]{$z_{12}$}($(n-1-1.south east) + (-1mm,0mm)$);

\node[scale=0.8] at (2.27, 3mm){$x_{(i+1)k}x_{ij} - x_{(i+1)j}x_{ik}=0 $};

\node[scale=0.8] at (2.2, -2mm){$x_{1j}z_{kl} + x_{1l}z_{jk} - x_{1k}z_{jl}=0 $};

\node[scale=0.8] at (1.5, -7mm){$z_{kl} x_{(d-2)j} + z_{jk}x_{(d-2)l} - z_{jl} x_{(d-2)k}=0$};

\node[scale=0.8] at (1.28, -1.8){for all $1\leq j < k<l \leq d$ and  $1\leq i \leq d-3$ };

\end{tikzpicture}\end{equation*}
    
\end{proposition}

We independently obtained the following consequence using direct computations on weighted projective stacks $\mathcal{P}(1^d, d)$. Here we avoid these computations and deduce the result from Hanihara's.  

\begin{corollary} \label{Cor:description_R_d}
We use the notation of Proposition \ref{P:Hanihara}.

Then $A$ is given as a quotient of $\Pi_{d-1}E$ by the two-sided ideal $(\mathbb{k}Q_d \{ z_{ab} \mid 1 \leq a<b \leq d \} \mathbb{k}Q_d)^2$. 

In particular, it has an explicit description as a quiver $Q_d$ with relations as in Proposition \ref{P:Hanihara} with additional relations $(\mathbb{k}Q_d \{ z_{ab} \mid 1 \leq a<b \leq d \} \mathbb{k}Q_d)^2=0$.    
\end{corollary}
\begin{proof}
This follows from Proposition \ref{P:Hanihara} together with Corollary \ref{C:QuotientPreproj}.
\end{proof}

\subsection{The non-split case}\label{subs:non-split}

The following is the main result of this subsection. Let $X_1$ and $X_2$ be projective varieties as in Proposition \ref{prop:algebra-extensions}, such that $\widehat{\mathcal{O}}_{X_1,p} \cong \widehat{\mathcal{O}}_{X_2,p}$ is a cone singularity (over a strong Fano $Z$).
Then there is a categorical absorption $\Db(A_1 )$ of $p\in X_1$ that is a deformation of a categorical absorption $\Db(A_2)$ of $p \in X_2$. 
More precisely, there is a flat family $\AA$ over $\mathbb{A}^1$, with fibers $A_1$ and $A_2$ over $0 \in \mathbb{A}^1$ and $1 \in \mathbb{A}^1$, respectively.

\begin{theorem}\label{T:Deformations} In the setup of Proposition \ref{prop:algebra-extensions}, we fix the isomorphism class of the endomorphism algebra $E=\End(\sE)$.
    Let $A$ and $A'$ be finite dimensional algebras that appear as middle terms in the short exact sequence \eqref{eq:ses-G}. Then there exists a $\mathbb{k}[t]$-algebra $\mathcal{A}$ that satisfies:
    \begin{enumerate}[label=(\alph*)]
        \item $\mathcal{A}$ is a free (and hence flat) $\mathbb{k}[t]$-module.
        \item $\mathcal{A}/(t) \cong A$
        \item $\mathcal{A}/(t-1) \cong A'$
    \end{enumerate}
    In other words, $A$ and $A'$ are fibers in a flat family of algebras $\mathcal{A}$. In particular, any such algebra $A$ is a deformation of the split algebra.

    Moreover, if $\gldim E \leq 1$, then $A$ and $A'$ are isomorphic to the split extension. 
\end{theorem}

It is well-known that this follows from work of Hochschild \cite{Hochschild} on square-zero extensions. Indeed, there is the following result that is much more general than Theorem \ref{T:Deformations}.

\begin{proposition}\label{P:DeformGeneral}
    Let $E$ be a finite dimensional $\mathbb{k}$-algebra and let $M$ be an $E$-bimodule.
    For $j=1, 2$, consider short exact sequences
    \begin{align}\label{eq:squarezeroextension}
        0 \to M \xrightarrow{i_j} A_j \xrightarrow{\varphi_j} E \to 0
    \end{align}
    such that $\varphi_j$ is an algebra homomorphism and the two-sided ideal $I_j:=i_j(M)$ satisfies $I_j^2=0$. Then there exists a $\mathbb{k}[t]$-algebra $\mathcal{A}$ that satisfies:
    \begin{enumerate}[label=(\alph*)]
        \item $\mathcal{A}$ is a free (and hence flat) $\mathbb{k}[t]$-module of finite rank.
        \item $\mathcal{A}/(t) \cong A_1$
        \item $\mathcal{A}/(t-1) \cong A_2$
    \end{enumerate} 
    Moreover, if $\gldim E \leq 1$, then the $A_j$ are isomorphic to the split extension $E \oplus M$.
\end{proposition}

For the convenience of the reader, we explain this result in several steps.

\begin{proposition}[{\cite[Propositions 6.2 \& 6.3]{Hochschild}}]\label{P:Hochschildcocylce}
    Let $A:=A_j$ and $I:=I_j$ be as in \eqref{eq:squarezeroextension}. In particular, $A \cong E \oplus I$ as vector spaces. Then the multiplication \eqref{eq:split multiplication} on $A$ is modified as follows
\begin{equation} \label{eq:non split multiplication}
    (\epsilon, \sigma)(\epsilon', \sigma')=(\epsilon \epsilon', \sigma \epsilon' + \epsilon \sigma' + C(\epsilon \otimes \epsilon'))
\end{equation}
    where $C \in \Hom_\mathbb{k}(E \otimes_\mathbb{k} E, I)$ is a $\mathbb k$-linear map,  
satisfying the cocycle condition
\begin{align}\label{eq:cocycle}
xC(y \otimes z)-C(xy \otimes z)+C(x \otimes yz) -C(x \otimes y)z=0
\end{align}
for all $x, y, z \in E$. Conversely, every bilinear form $B$ satisfying \eqref{eq:cocycle} gives rise to an associative algebra structure on $A$ by replacing $C$ with $B$ in \eqref{eq:non split multiplication}.

Moreover, two cocycles $C, C'$ define isomorphic algebra structures on $A$ if and only if they define the same class in Hochschild cohomology $\mathrm{HH}^2(E, I)$.

Finally, $A$ is isomorphic to the split extension \eqref{eq:split multiplication} if and only if the class of $C$ vanishes in $\mathrm{HH}^2(E, I)$.
\end{proposition}

\begin{remark}
    For later use, we note that a cocycle $C$ as in \eqref{eq:non split multiplication} is constructed as follows. Let $\psi \colon E \to A$ be a $\mathbb{k}$-linear map (not necessarily a ring homomorphism) that splits $\varphi$, then $C(x \otimes y):=\psi(x)\psi(y)-\psi(xy)$. In particular, we can choose $\psi$ such that $\psi(1)=1$ and thus can assume $C(1 \otimes 1)=0$.
\end{remark}

\begin{proof}[Proof of Proposition \ref{P:DeformGeneral}]
    Let $S:=E \oplus M$ be the split extension (so the multiplication is as in \eqref{eq:split multiplication}). Set $E[t]:=E \otimes_\mathbb{k} \mathbb{k}[t]$, $M[t]:=M \otimes_\mathbb{k} \mathbb{k}[t]$ and $S[t]:=E[t] \oplus M[t] \cong S \otimes_\mathbb{k} \mathbb{k}[t]$. Then $M[t]$ is an $E[t]$-bimodule and $S[t]$ is a split extension.

    We apply Proposition \ref{P:Hochschildcocylce} to $A_j$ to obtain cocycles $C_j \in \Hom_\mathbb{k}(E \otimes_\mathbb{k} E, I_j) \cong \Hom_\mathbb{k}(E \otimes_\mathbb{k} E, M)$ such that the multiplication on $A_j$ is given by \eqref{eq:non split multiplication} with $C$ replaced by $C_j$ and such that $C_j(1 \otimes 1)=0$. We can extend $C_j$ to a cocycle $C_j[t] \in \Hom_\mathbb{k}(E[t] \otimes_\mathbb{k} E[t], M[t])$ by setting $C_j[t]((e \otimes p) \otimes (e' \otimes p'))=C_j(e \otimes e') \otimes pp'$. One can check that the linear combination
\begin{align}
    L:=(1 \otimes t) C_2[t] + (1 \otimes (1-t)) C_1[t]
\end{align}
    is again a cocycle. Hence, by Proposition \ref{P:Hochschildcocylce}, we get a new algebra structure $\mathcal{A}$ on $S[t]$: 
    \begin{equation} 
    (\epsilon \otimes p, \sigma \otimes p)(\epsilon'\otimes p', \sigma'\otimes p')=(\epsilon \epsilon' \otimes pp', \sigma \epsilon' \otimes pp' + \epsilon \sigma' \otimes pp' + L((\epsilon \otimes p) \otimes (\epsilon' \otimes p')))
\end{equation}
  Since $C_j(1 \otimes 1)=0$, there is an ring monomorphism 
  \begin{align}\label{eq:kt-alg}
      \mathbb{k}[t] \to \mathcal{A} \\
      p \mapsto (1 \otimes p, 0)
  \end{align}
  with image in the center of $\mathcal{A}$. Hence, $\mathcal{A}$ is a $\mathbb{k}[t]$-algebra. 
  
  By definition there is a short exact sequence
   \begin{align}
   0 \to M[t] \xrightarrow{i \otimes \id} \mathcal{A} \xrightarrow{\pi \otimes \id} E[t] \to 0
    \end{align}
where $i\colon M \to S$ is the canonical inclusion and $\pi\colon S \to E$ is the canonical projection. One can check that $i \otimes \id$ and  $\pi \otimes \id$ are $\mathbb{k}[t]$-module homomorphisms, where we view $\mathcal{A}$ as a $\mathbb{k}[t]$ module via \eqref{eq:kt-alg}. Since $E[t]=E \otimes_\mathbb{k} \mathbb{k}[t]$ is a free $\mathbb{k}[t]$-module, this sequence splits and $\mathcal{A} \cong E[t] \oplus M[t]$ is a free $\mathbb{k}[t]$-module of finite rank, showing (a).

Using the definition of $L$, one can check that 
\begin{align*}
    &\mathcal{A}/((1 \otimes t, 0)) \xrightarrow{\cong} A_1 \\
     &\overline{(e \otimes p, m \otimes p)} \mapsto  (e p(0), m p(0)) \\ \\
     &\mathcal{A}/((1 \otimes (1-t),0)) \xrightarrow{\cong} A_2 \\
     &\overline{(e \otimes p, m \otimes p)} \mapsto  (e p(1), m p(1))  
\end{align*}
define algebra isomorphisms, showing (b) and (c), where we identify $t$ and $(1 \otimes t, 0)$ using the embedding $\mathbb{k}[t] \to \mathcal{A}$. 

The last part, follows by combining Proposition \ref{P:Hochschildcocylce} with 
\begin{align}
 \mathrm{HH}^2(E, M)=\Ext^2_{E^{\mathop{ev}}}(E, M)=0,   
\end{align}
which follows from $\projdim_{E^{ev}}(E)\leq 1$. The latter holds since $\gldim E \leq 1$, cf. e.g. \cite[Equation (1.3.)]{ButlerKing}.
\end{proof}

\begin{remark}
    There is only a contribution of the cocycle $C$ in \eqref{eq:non split multiplication} if $\epsilon \neq 0 \neq \epsilon'$. In particular, in the setting of Proposition \ref{prop:algebra-extensions} we can use the description of $A_X:=\End_X(\sG)$ in the split case (see subsection \ref{subsec:split-case} in particular Corollary \ref{C:QuotientPreproj}) to describe (most of) $A_{X'}:=\End_{X'}(\sG)$ in a non-split case. Indeed, the quivers of $A_X$ and $A_{X'}$ are the same  and all relations $\sum_i \lambda_i p_i=0$ where all paths $p_i \in I$ are the same in $A_X$ and $A_{X'}$. For relations $\sum_j \mu_j q_j =0$ in $A_{X'}$ where at least one of the $q_j$ is not in $I$, we know that modulo $I$ this corresponds to a relation in $E=\End(\sE)$.
\end{remark}

We illustrate this remark in the following special case: 

\begin{prop}\label{prop:quotsing-nonsplit}
Let $X$ be a $d$-dimensional projective variety with a single singularity $p \in X$. Assume $p$ is complete locally isomorphic to the singularity of the  cone over the Veronese embedding of $\P^{d-1}$ of degree $d$.
Suppose that the blow-up $\pi \colon \tilde{X} \to X$ of $X$ at $p$ satisfies  Assumption \ref{A:KeyAssumptions}\footnote{The geometric exceptional sequence is the truncated Beilinson exceptional collection $\OO (-d+1) , \ldots , \OO (-1)$ in $\Db( \P^{d-1}) $, which is geometric of type $(d-1, d-1)$.}. The algebra $R_X=\End(\sG)$ is given by the following quiver $Q_d$ with relations:
\begin{equation*}\begin{tikzpicture}[description/.style={fill=white,inner sep=2pt}]

    \matrix (n) [matrix of math nodes, row sep=2em,
                 column sep=2em, text height=1.5ex, text depth=0.25ex,
                 inner sep=2pt, nodes={inner xsep=0.3333em, inner
ysep=0.3333em}, ampersand replacement=\&] at (-5, 0)
    {   1  \&\& 2 \&\& \cdots \&\& d-1 \\
          };

    \draw[->] ($(n-1-1.east) + (0,2mm)$)  to node[yshift=-1.2mm, scale=0.8] [midway]{$\vdots$} node[yshift=1.7mm, scale=0.8] [midway]{$x_{11}$}($(n-1-3.west) + (0mm,2mm)$) ;
    \draw[->] ($(n-1-1.east) + (0,-2mm)$) to node[yshift=-1.7mm, scale=0.8] [midway]{$x_{1d}$} ($(n-1-3.west) + (0mm,-2mm)$);

    \draw[->] ($(n-1-3.east) + (0,2mm)$)  to node[yshift=-1.2mm, scale=0.8] [midway]{$\vdots$} node[yshift=1.7mm, scale=0.8] [midway]{$x_{21}$}($(n-1-5.west) + (0mm,2mm)$) ;
    \draw[->] ($(n-1-3.east) + (0,-2mm)$) to node[yshift=-1.7mm, scale=0.8] [midway]{$x_{2d}$} ($(n-1-5.west) + (0mm,-2mm)$);
    
       \draw[->] ($(n-1-5.east) + (0,2mm)$)  to node[yshift=-1.2mm, scale=0.8] [midway]{$\vdots$} node[yshift=1.7mm, scale=0.8] [midway]{$x_{(d-2)1}$}($(n-1-7.west) + (0mm,2mm)$) ;
    \draw[->] ($(n-1-5.east) + (0,-2mm)$) to node[yshift=-1.7mm, scale=0.8] [midway]{$x_{(d-2)d}$} ($(n-1-7.west) + (0mm,-2mm)$);

     \draw[->] ($(n-1-7.south west) + (5mm,-0mm)$) .. controls +(-3.7mm,-16mm)
and +(+3.7mm,-16mm) .. node[fill=white, scale=0.8] [midway]{$z_{(d-1)d}$} ($(n-1-1.south east) + (-3mm,0mm)$);

      \draw[->] ($(n-1-7.south west) + (3mm,-0mm)$) .. controls +(-3.7mm,-6mm)
and +(+3.7mm,-6mm) .. node[yshift=-3mm] [midway]{$\vdots$} node[yshift=-3.5mm, xshift=7mm, scale=0.7] [midway]{$z_{ij, \ i<j}$} node[fill=white, scale=0.8] [midway]{$z_{12}$}($(n-1-1.south east) + (-1mm,0mm)$);

\node[scale=0.8] at (3.8, -4mm){
\begin{minipage}{100pt}
    {\begin{align*}
    x_{(i+1)k}x_{ij} + C_X(x_{(i+1)k} \otimes x_{ij}) - x_{(i+1)j}x_{ik} - C_X(x_{(i+1)j} \otimes x_{ik})&=0 \\
    x_{1j}z_{kl} + x_{1l}z_{jk} - x_{1k}z_{jl}&=0 \\
    z_{kl} x_{(d-2)j} + z_{jk}x_{(d-2)l} - z_{jl} x_{(d-2)k}&=0 \\
    \text{for all } 1\leq j < k<l \leq d \text{ and } 1\leq i \leq d&-3 \\ \\
    \text{and additionally }(\mathbb{k}Q_d \{ z_{ab} \mid 1 \leq a<b \leq d \} \mathbb{k}Q_d)^2&=0
\end{align*}}
\end{minipage}
}; 
\end{tikzpicture}\end{equation*}
where  $C_X \in \Hom_{\mathbb k}(\End(\sE) \otimes_{\mathbb k} \End(\sE), I_{\sF})$ is a $\mathbb k$-linear map that satisfies the cocycle condition \eqref{eq:cocycle} and $\End(\sE)$ is the truncated Beilinson algebra (the latter only depends on the complete local information of our singularity $p$ and not on the global geometry of $X$).
\end{prop}
\begin{proof}
    Corollary \ref{eq:cone-absorption} and Example \ref{ex:projspace-quadric}(a) describe $R_X$ in the split case. Using this, the general (non-split) case follows from Proposition \ref{P:Hochschildcocylce}. 
\end{proof}

\begin{remark}
    Recall that if the characteristic of $\mathbb{k}$ is coprime to $d$, then the singualrity  $p \in X$ of the cone over the Veronese embedding of $\P^{d-1}$ of degree $d$ can be identified with the quotient singularity of type $\frac{1}{d} (1^d)$.
\end{remark}

\section{The derived category of cone singularities over Fano varieties} \label{sec:cone-Fano}

One of the main results of this section (Corollary \ref{eq:cone-absorption}) is the construction of a Kawamata type semiorthogonal decomposition  
\[
\Db(X) \cong \langle \Db(A) , \CC \rangle ,
\]
where $X$ is a $d$-dimensional projective variety with  a unique  singular point $p \in X$ which is an $\mathrm{ACS}$-singularity (see Definition \ref{def:ACS}), and such 
that Assumption \ref{A:KeyAssumptions} holds, where $ \Bl_p (X)$ plays the role of $Y$ in \ref{A:KeyAssumptions}.
Moreover, the finite dimensional $\mathbb{k}$-algebra $A$ fits into a short exact sequence
\begin{equation}\label{eq:sequence-summary}
           0\to \Hom^\bullet ( \sE , \bS^{-1} \sE [d-2] ) \to A \xrightarrow{q} \End (  \sE ) \to 0 ,
\end{equation}
where $q$ is an algebra homomorphism and $\bS_\sE^{-1}(\sE)[d-2]$ is an $\End(\sE)$-bimodule.
Recall that $\sE$ is a geometric exceptional sequence of type $(m, d-1)$, see Assumption \ref{A:KeyAssumptions}.
In section \ref{sec:algebra-description},
we explain how to describe the algebra $A$ in terms of quivers with relations, assuming that we know such a description for $\End(\sE)$ and if, additionally, one of the following holds: 
\begin{enumerate}
    \item  $\dim(X) \leq 3$, i.e. $\gldim \End(\sE) \leq 1$ (see Corollary \ref{cor:nodal-2d} and Corollary \ref{C:QuotientPreproj2}).
    \item The sequence \eqref{eq:sequence-summary} splits (for example if $X$ is a certain projective cone, see Theorem \ref{thm:projcone-split}) and one of the following holds:
    \begin{enumerate}
        \item $\dim X\leq 4$ , i.e. $\gldim \End(\sE) \leq 2$ (see Corollary \ref{C:QuotientPreproj3}).
    \item $\End(\sE)$ is obtained by truncating the Beilinson algebra by removing the vertex corresponding to $\OO$ (Corollary \ref{Cor:description_R_d}).
    \end{enumerate}
\end{enumerate}
See also Remark \ref{rmk:product-of-proj} for additional cases.

\subsection{Cone singularities and localization}\label{subsec:cone-and-localization}
Throughout this section,  $\mathbb k$ is an algebraically closed field of characteristic $0$.
Recall that a variety $X $ together with a closed embedding $X\subset \P^r$ is called \emph{projectively normal}, if the corresponding homogeneous coordinate ring of $X \subset \P^r$ is  is an integrally closed domain.

Let $Z$ is a smooth $(d-1)$-dimensional Fano variety and let $\sL$ be a very ample line bundle on $Z$, such that the corresponding embedding $Z\subset \P^r$ is projectively normal.
In this section, we assume that $X$ is a $d$-dimensional projective variety such that $p \in X$ is the singularity of the cone over $Z$ with respect to  $Z\subset \P^r$.
Moreover, we assume that $\omega_Z \cong \sL^{-m}$, holds, for some $m$.
Let $\pi : \tilde{X} \to X$ be the blow-up of $X$ at $p$.
This blow-up $\pi\colon \tilde{X} \to X$ is resolves the singularity $p\in X$.
The exceptional divisor $\iota \colon E \hookrightarrow \tilde{X}$ is isomorphic to $Z$ and we have an isomorphism of line bundles $\iota^\ast \OO_{\tilde{X}} (E) \cong \sL^{-1}$.
Note that such $p\in X$ are rational singularities.
We have the following result.

\begin{theorem}[Efimov]\label{thm:bo-localization-and-kernel}
    Let $X$ and  $p \in X$ be as above, and assume that $X$ is smooth away from $p \in X$. 
    Then the pushforward functor $\pi_\ast : \Db (\tilde{X}) \to \Db(X)$ is a Verdier localization.
    That is, there is an induced equivalence 
    \[
    \pi_\ast \colon \Db ( \tilde{X}) / \ker ( \pi_\ast ) \xrightarrow{\cong} \Db(X) .
    \]
    Moreover, the kernel category $\ker ( \pi_\ast ) \subset \Db(\tilde{X}) $ is generated by $\iota_\ast (\OO_E^{\perp})$. 
\end{theorem}

\begin{proof}
    Recall from \cite[Definition 5.3]{ks} that a normal isolated singularity $p \in U$ is called \emph{acyclic projectively normal}, if
    \begin{enumerate}
        \item\label{eq:apn-1} the canonical morphism $\mm^n_{p} / \mm^{n+1}_{p} \to \bH^0 ( E , \OO_E (-nE) )$ is an isomorphism, and if
        \item\label{eq:apn-2}  the vanishing $\bH^p ( E , \OO_E (-n E ) ) = 0$ holds for all $i > 0$,
    \end{enumerate}
    for all $n \geq 0$. 
    Here, $\mm_{p} \subset \OO_U$ denotes the maximal ideal corresponding to $p \in U$, and $E $ is the exceptional divisor of the blow-up $\Bl_p (U)$. 
    We note that the singularity $p\in X$ as in the statement of the theorem
    is a acyclic projectively normal isolated singularity. 
    Indeed, $p \in X$ being a acyclic projective variety is a complete local question. Hence we can assume that $p \in X$ is the affine cone over the Fano $Z$ via the embedding $Z \subset \P^r$, and $\tilde{X}$ is the blow-up at the vertex.
    Since the corresponding embedding $Z \subset \P^r$ is projectively normal, together with \cite[Chapter II, Exercise 5.14]{hartshorne} we have that condition \eqref{eq:apn-1} is satisfied.
    For condition \eqref{eq:apn-2}, we note that there is an isomorphism
    \[
    \bH^p ( E , \OO_E ( -n E) ) \cong \bH^p ( Z , \sL^n )  = 0, 
    \]
    for all $n \geq 0$ and $p > 0$, where we used Kodaira vanishing in the second equality (and that $\omega_Z \cong \sL^{-m}$).
    It follows then by \cite[Corollary 5.6]{ks} that $\pi_\ast : \Db(\tilde{X} ) \to \Db(X)$ is a Verdier localization and that $\ker  \pi_\ast $ is generated by 
       $\iota_\ast ( \OO_E^{\perp} ) $.
\end{proof}

We say that a smooth Fano variety $Z$ has \emph{projectively normal anticanonical embedding}, if
 the anticanonical bundle  $\omega_Z^{-1}$ is very ample, and such that the corresponding  embedding $Z \subset \P^r$ given by $\omega_Z^{-1}$ is projectively normal.

\begin{corollary}\label{cor:bo-localization-anticanonical}
    Suppose in addition to the assumptions of Theorem \ref{thm:bo-localization-and-kernel} that $p \in X$ is the singularity of the cone over a smooth Fano variety $Z$ with projectively normal anticanonical embedding.
    Assume further that $\Db(Z)$ admits a full exceptional collection 
    \[
    \sL_1 , \ldots , \sL_m , \OO_Z .
    \]
    Then the blow up $\pi \colon \tilde{X} \to X$ at $p \in X$ is a crepant resolution and $\pi_\ast \colon \Db( \tilde{X} ) \to \Db(X)$ is a Verdier localization.
    Furthermore, the kernel category $\ker \pi_\ast$ is generated by $d$-spherical objects
    \[
    \iota_\ast \sL_1 , \ldots , \iota_\ast \sL_m .
    \]
\end{corollary}

\begin{proof}
        The fact that $\pi\colon \tilde{X} \to X$ is crepant is well-known.
        \footnote{Locally the blow up of $p\in X$ is isomorphic to the total space of $\omega_Z$, since the anticanonical embedding is projectively normal.
        By the formula for the canonical bundle of the total space it follows that the canonical bundle of the local blow-up is trivial.
        As the singularity is rational, this together with \cite[Theorem 5.10]{kollar-mori} implies that $p\in X$ is Gorenstein and we have that $\omega_{\tilde{X}}\cong \pi^\ast \omega_X $.}\label{footnote5}
        Moreover, by Theorem \ref{thm:bo-localization-and-kernel} $\pi_\ast \colon \Db ( \tilde{X} ) \to \Db (X)$ is a Verdier localization and the kernel $\ker \pi_\ast$ is generated by the objects
        \[
        \iota_\ast\sL_1 , \ldots , \iota_\ast\sL_m .
        \]
        It remains to show that the $\iota_\ast \sL_j$ are $d$-spherical objects in $\Db(\tilde{X} )$.
    Since the blow-up $\pi : \tilde{X} \to X$ is a crepant resolution and $\iota_\ast\sL_j$ is in $\ker \pi_\ast$,  Lemma \ref{lem:dCY-kernel} shows $\bS_Y ( \iota_\ast \sL_j ) \cong \iota_\ast \sL_j [d]$ in $\Db ( \tilde{X} )$.
Moreover, 
applying $\Hom^\bullet (  \sL_j , - )$ onto the distinguished triangle 
\[
 \sL_j \to \iota^! \iota_\ast\sL_j \to \sL_j \otimes \OO_E(E)[-1]
\]
in $\Db( \tilde{X})$, we obtain a long exact sequence of 
$\Hom$-spaces
\[
\ldots \Hom ( \sL_j , \sL_j [s] ) \to \Hom ( \sL_j , \iota^! \iota_\ast \sL_j [s] ) \cong \Hom ( \iota_\ast\sL_j , \iota_\ast \sL_j [s] ) \to \Hom ( \sL_j , \sL_j \otimes \omega_Z [s-1] ) \to \ldots ,
\]
where we used adjunction in the second component and the isomorphism of line bundles $\OO_E (E) \cong \omega_Z$ in the third component.
Using Serre duality on $\Db(Z)$ in the third component  
and that $\sL_j$ is an exceptional object in $\Db(Z)$ in both the first and third component, we have conclude isomorphisms
\begin{align*}
\Hom (\iota_\ast\sL_j , \iota_\ast\sL_j [s] ) \cong  \begin{cases} \mathbb{k} , \quad & \text{if } s =   0 , d ,  \\ 
0 \quad & \text{else. }\end{cases}
\end{align*}
This finishes the proof of the corollary.
\end{proof}

\begin{example}\label{ex:delpezzo-projnormal}
    Here is a list of some Fano varieties with projectively normal anticanonical embedding:
    \begin{enumerate}
        \item The Veronese embedding $\P^{d-1} \subset \P^N$ of any degree is projectively normal (cf. \cite[Chapter II, Exercise 5.14]{hartshorne}). 
        In particular, the anticanonical embedding $\P^{d-1} \subset \P^r$ is projectively normal.
        \item Smooth $(d-1)$-dimensional quadrics $Q \subset \P^{d}$ are projectively normal.
        In particular, the anticanonical embedding $Q\subset \P^r$ is projectively normal (cf. \cite[Chapter II, Exercise 5.14]{hartshorne}).
        \item Smooth del Pezzo surfaces of degree at least $3$ have very ample anticanonical bundle.
        The corresponding anticanonical embedding is projectively normal, by \cite[Theorem 8.3.4]{dolgachev}.
        \item Smooth $d-1 \geq 3$ dimensional del Pezzo varieties varieties of degree $5$ have projectively normal anticanonical embedding (see e.g. \cite[Corollary 2.8 and Definition 0.1]{Gallego-Purnaprajna}).
    \end{enumerate}
    In particular, Corollary \ref{cor:bo-localization-anticanonical} can be applied to projective varieties $X$, such that $p\in X$ is the cone singularity over these Fano varieties (with respect to  the anticanonical embedding).
\end{example}

\begin{remark}\label{rmk:projnormal-product}
    If $Z_1 $ and $Z_2$ are Fano varieties with projective normal anticanonical embedding $Z_1 \subset \P^{r_1}$ and $Z_2 \subset \P^{r_2}$, then their product $Z = Z_1  \times Z_2$ is a Fano with projectively normal anticanonical embedding.
    Indeed, the anticanonical bundle of $Z$ is very ample, since the  anticanonical bundles of $Z_1$ and $Z_2$ are very ample, and the anticanonical embedding of $Z$ is given by the Segre embedding.
    Moreover, the Segre embedding (of $\P^{r_1} \times \P^{r_2}$) is projectively normal, and thus the anticanonical embedding of $Z$ is projectively normal (using e.g. \cite[Chapter II, Exercise 5.14 (d)]{hartshorne}).
\end{remark}

\subsection{Cone singularities and categorical absorptions}
The following theorem discusses an equivalent criterion to strong adherence in the case  singularities of cones over Fano varieties with projectively normal anticanonical embedding:

\begin{theorem}\label{thm:strong-adherence-equiv}
    Let $X$ be a $d$-dimensional projective variety with a single isolated singularity at $p$, and let $p\in X$ be the singularity of the cone over a smooth Fano variety 
    with projectively normal anticanonical embedding. 
     Let
    $\pi \colon \tilde{X} \to X$ be the blow-up of $X$ at $p$ and $\iota \colon E \subset \tilde{X}$ its exceptional locus.
    Assume
    there is an exceptional collection $\E = ( \sE_1 , \ldots , \sE_m ) $ in $\Db ( \tilde{X} )$, such that the image of $\E$ via the right adjoint $\iota^! \colon \Db ( \tilde{X} ) \to \Db(E)$ to the pushforward functor $\iota_\ast$ is contained in $\langle \OO_E \rangle^{\perp}$.
    
    Then there is an autoequivalence $\bT$ on $\Db ( \tilde{X} )$ such that $\pi_\ast$ is invariant under $\bT$ and such that $\E$ and $\bT$ are strongly adherent, if and only if 
      the functor $\iota^! \colon \Db ( \tilde{X} ) \to \Db(E)$ restricts to an equivalence
    \[
    \iota^! \colon \thick ( \E ) \xrightarrow{\cong} \langle \OO_E \rangle^\perp .
    \]

\end{theorem}

\begin{proof}
    The ``only if'' direction follows from Proposition \ref{prop:defs-agree}.
    For the ``if'' direction,
    set $\L = ( \sL_1 , \ldots , \sL_{m} )$ with $\sL_j : = \iota^!  \sE_j  [1]$. 
    This collection $\L$ forms a full exceptional collection of $\langle \OO_E \rangle^\perp$. 
    By Corollary \ref{cor:bo-localization-anticanonical}, the $\iota_\ast\sL_j$ are spherical, hence
    by Proposition \ref{prop:spherical} \ref{prop:spherical-a}, the functor $\pi_\ast $ is invariant under the autoequivalence
    \[\bT_\L : = \bT_{\iota_\ast\sL_1} \circ \ldots \circ \bT_{\iota_\ast\sL_m}\] 
    on $\Db( \tilde{X}).$
    We want to show that $\E$ and $\bT_{\L}$
    are strongly adherent.
    For this we start by computing the following graded  $\Hom$-spaces:
    \begin{equation}
        \Hom^\bullet ( \iota_\ast \sL_j , \sE_i ) \xrightarrow[\cong]{\ad} \Hom^\bullet ( \sL_j , \iota^! \sE_i ) = \Hom^\bullet ( \iota^! \sE_j , \iota^! \sE_i ) [-1]
    \end{equation}
    In particular, if $i=j$ there is a morphism $\epsilon_j \colon \iota_\ast\sL_j \to \sE_j [1] $ corresponding to the identity morphism $\id \colon \iota^!\sE_j \to \iota^! \sE_j$ via adjunction.
    Moreover, by standard properties of the adjunction, we have that the commutativity of morphisms
    $\iota^! (-)  = \ad (-)\cdot \epsilon_j$ between graded vector spaces in $\Db(\mathbb{k})$ holds.
    Thus since $\ad (-)$ is an isomorphism, as well as $\iota^! (-)$ on the morphism of $\E$ by assumption, we obtain that
    \begin{equation}\label{eq:some-hom-prop}
        (-)\cdot \epsilon_j \colon \Hom^\bullet (\sE_j , \sE_i ) \xrightarrow[\cong]{(-)\cdot\epsilon_j} \Hom^\bullet (   \iota_\ast \sL_j, \sE_i )[1]
    \end{equation}
    is an isomorphism in $\Db(\mathbb{k})$.
    In particular, we have shown that \eqref{eq:assump1}  of Proposition \ref{prop:spherical} \ref{prop:spherical-b} is satisfied.
    Moreover, there is a distinguished triangle
    \begin{equation}\label{eq:spherical-iso-2}
          \iota_\ast \sL_i [-1] \xrightarrow{\epsilon_i} \sE_i \to \bT_{\sL_i } \sE_i \to \iota_\ast \sL_i    .
    \end{equation}
    Applying $\Hom^\bullet ( \iota_\ast \sL_n , - )$ with $n < i$ to this triangle, we get a triangle 
    \begin{equation}\label{eq:another-dist-triangle}
     \Hom^\bullet ( \iota_\ast \sL_n , \bT_{\iota_\ast \sL_i}\sE_i ) \to \Hom^\bullet ( \iota_\ast \sL_n , \iota_\ast \sL_i ) \xrightarrow{\epsilon_i \cdot (-)}\Hom^\bullet ( \iota_\ast \sL_n , \sE_i  )[1]
    \end{equation}
    in $\Db(\mathbb{k})$.
    The same computation as in  \eqref{eq:K-L-iso} in the proof of Proposition \ref{prop:defs-agree} shows that there are isomorphisms
    \[
    \Hom^\bullet ( \iota_\ast \sL_n , \iota_\ast \sL_i  ) \xrightarrow[\cong]{\ad(-)} \Hom^\bullet ( \sL_n , \iota^! \iota_\ast \sL_i ) \xleftarrow[\cong]{  \ad_{\sL_i} \cdot (-) }  \Hom^\bullet ( \sL_n ,  \sL_i ) 
     \]
     for $n<i$.
     Here the morphism $\ad_{\sL_i} \colon \sL_i \to \iota^! \iota_\ast \sL_i$ is the morphism in $\Db(E)$ corresponding to the identity $\id \colon \sL_i \to \sL_i$ via adjunction.
     These isomorphisms extend to the following commutative diagram
     \begin{equation}
\xymatrix{
   \Hom^\bullet ( \sL_n , \iota^! \iota_\ast \sL_i )  \ar@{<-}[rr]^-{\ad_{\sL_i} \cdot (-)}_\cong  \ar@{<-}[rrd]_{\ad(-)}^\cong & &   \Hom^\bullet (   \sL_n ,  \sL_i ) \ar@{->}[d]^{\iota_\ast (-)} \ar@{->}[rr]^-{ \ad(\epsilon_i )  \cdot (-) = \id  }_\cong & & \Hom^\bullet ( \sL_n , \iota^!\sE_i [1] ) \ar@{<-}[d]^{\ad(-)}_{\cong} \\
  & & \Hom^\bullet ( \iota_\ast \sL_n , \iota_\ast \sL_i ) \ar@{->}[rr]^{\epsilon_i \cdot (-)} & & \Hom^\bullet (  \iota_\ast \sL_n , \sE_i [1] )},
\end{equation}
where the commutativity of both the triangle and the square follows from the definition of adjunction.
In particular, we read off from this diagram that 
the bottom morphism $\epsilon_i \cdot (-)$ is an isomorphism.
Thus, we see that Assumption \eqref{eq:assump2} of Proposition \ref{prop:spherical} \ref{prop:spherical-b} is satisified.
We can thus apply Proposition \ref{prop:spherical} \ref{prop:spherical-b} and we get
ismorphisms
$\bT_{\iota_\ast \sL_i} \sE_i \cong \bT_\L \sE_i$ and that $\bT_\L \sE_i$ fits into a distinguished triangle
   \begin{equation}\label{eq:main-triangle}
            \iota_\ast \sL_i[-1] \xrightarrow{\epsilon_i} \sE_i \to \bT_{\L} \sE_i \to \iota_\ast \sL_i 
    \end{equation}
    in $\Db( \tilde{X} )$, for all $i$.

    Next we show that $\E$ and $\bT_\L$ are adherent.
   For this, we need to show that the collection $\mathbb{E} ,  \bT_{{\L}}   \E $ is exceptional, and that $\tilde{ \AA } = \langle \E , \bT_{\L}  \E  \rangle $ contains $\ker\pi_\ast$.
    Applying $\Hom^\bullet ( - , \sE_j )$ onto the triangle \eqref{eq:main-triangle}, we obtain
    a distinguished triangle 
    \[
    \Hom^\bullet ( \iota_\ast \sL_i , \sE_j ) \to \Hom^\bullet ( \bT_\L \sE_i , \sE_j ) \to \Hom^\bullet ( \sE_i , \sE_j  ) \xrightarrow[\cong]{(-)\cdot \epsilon_i} \Hom^\bullet ( \iota_\ast \sL_i , \sE_j )[1]
    \]
    in $\Db(\mathbb{k})$, where by \eqref{eq:some-hom-prop} the morphism $(-)\cdot \epsilon_i$ is an isomorphism in $\Db(\mathbb{k})$.
    It follows that $\Hom^\bullet  ( \bT_{\L} \sE_i , \sE_j  ) $ vanishes for all $i$ and $j$ and thus the collection $\E , \bT_{\L}  \E $ is exceptional.
    Moreover, from the distinguished triangle \eqref{eq:main-triangle} we read of that $\iota_\ast \sL_i$ is contained in $\tilde{\AA} $, for all $i$.
    Since $\L$ is a full exceptional collection of $\langle \OO_E\rangle^\perp$ (see above), we see by Theorem \ref{thm:bo-localization-and-kernel} that $\iota_\ast \sL_1 , \ldots , \iota_\ast \sL_{m} $ generates $\ker \pi_\ast$.
    This means that $\ker\pi_\ast$ is contained in $\tilde{\AA}$ and we have shown that $\E$ and $\bT_\L$ are adherent.
    Finally, $\E$ and $\bT_\L$ are also strongly adherent, as the cone of the morphism $\sE_i \to \bT\sE_i $ is of the form $\iota_\ast \sL_i$ in $\Db ( \tilde{X} )$ (see \eqref{eq:main-triangle}) and since $\L$ forms an exceptional collection of $\langle \OO_E \rangle^\perp$.
    This finishes the proof of the Theorem.
\end{proof}

\begin{remark}
     To prove Theorem \ref{thm:strong-adherence-equiv}, we only use the assumption that $p \in X$ is the singularity of the cone over a smooth Fano variety with projectively normal anticanonical embedding to deduce
    \begin{align}\label{eq:Efimov}
     \ker\pi_\ast = \iota_\ast ( \OO_E^\perp ),   
    \end{align}
  via Theorem \ref{thm:bo-localization-and-kernel}.
    We do not know if \eqref{eq:Efimov} holds more generally for projective isolated Gorenstein singularities $p\in X$ with a crepant resolution.
\end{remark}

\begin{corollary}\label{eq:cone-absorption}
     Let $X$ be a $d$-dimensional projective variety over $\mathbb k$ with a single   singularity at $p$, and assume that $p\in X$ is the singularity of the cone over a Fano variety 
    with projectively normal anticanonical embedding.
    Let
    $\pi \colon \tilde{X} \to X$ be the blow-up of $X$ at $p$.
    Denote by $\iota \colon E \subset \tilde{X}$ its exceptional divisor.
    If there is a geometric exceptional collection $\E$ of type $(m,k)$ with $k \leq d-1$ and such that there is an equivalence 
     \[
    \iota^! \colon \thick ( \E ) \xrightarrow{\cong} \langle \OO_E \rangle^\perp .
    \]
    Then there is an admissible semiorthogonal decomposition 
    \begin{equation}\label{eq:cone-absorp}
        \Db(X) \cong \langle \DD_{\mathrm{fg}} (A^\bullet ) , \BB \rangle 
    \end{equation}
    where $\BB \cong \pi_\ast {^{\perp}\tilde{\AA}} \subset \Dperf (X)$ with $\tilde{\AA} = \langle \E , \bT_{\L } \E \rangle$ and $A^\bullet $ is a finite-dimensional $\dg$-algebra with cohomology concentrated in degrees $0$ and $k+1-d$ and isomorphic (as $\Hom$-spaces) to $\End ( \sE )$ and $\Hom ( \sE , \bS^{-1} \sE [k-1] )$, respectively.
    In particular, if $\sE$ is of type $(m, d-1 )$, then $A^\bullet$ is quasi-isomorphic to a finite dimensional $\mathbb{k}$-algebra $A$ fitting into a short exact sequence
    \begin{equation}\label{eq:ses-algebra-A}
        0\to \Hom^\bullet ( \sE , \bS^{-1} \sE [d-2] ) \to A \xrightarrow{q} \End (  \sE ) \to 0 ,
    \end{equation}
    where $\sE = \bigoplus_i \sE_i$, and $q$ is a morphism of $\mathbb{k}$-algebras.
\end{corollary}

\begin{proof}
    The corollary follows immediately from Theorem \ref{thm:strong-adherence-equiv} and Corollary \ref{cor:veriderloc-negative-dg}.
\end{proof}

\begin{remark}\label{rmk:k-geq-d-1}
   Suppose, in addition to the assumptions in Corollary \ref{eq:cone-absorption} (with $k \leq d-1$), that the full exceptional collection $\iota^! \E , \OO_E $ of $\Db(E)$ is strong.
   Then $\Db(E) $ has a tilting object $\sT$ and we set $A = \End ( \sT )$. 
   Moreover, we note by Corollary \ref{cor:geometric-subcollection} that the full exceptional collection $\iota^! \E ,  \OO_E$ is geometric of type $( m+1 , k+1 )$.
   We have the following chain of inequalities 
   \[
    \gldim A \geq \rqdim \, \Db(A) = \rqdim \, \Db(E) \geq \dim E = d-1 ,
   \]
   where $\rqdim \, \DD$ denotes the Rouquier dimension of a triangulated category $\DD$, where we used \cite{krause-kussin} in the first inequality and \cite[Proposition 7.17]{rouquier} in the third inequality, and where the second equality comes from the tilting equivalence.
   Now, since $\sT$  arises as the direct sum of a
geometric exceptional collection of type $(m+1, k+1)$,
then $k \geq  \gldim A$, see Remark \ref{R:GldimHigherHered}.
We thus have an equality $k = d-1$ in this case.
In particular, in all our examples in subsection \ref{subsec:ProjCones} the corresponding exceptional collection $\E$ will be geometric of type $(m , d-1)$ \footnote{In particular, the blow-up $\pi \colon \tilde{X} \to X$ of $X$ at $p$ satisfies Assumption \ref{A:KeyAssumptions}) 
and the singularity $p \in X$ is an $\mathrm{ACS}$-singularity. }

We are not aware of any $E$ admitting  (only non-strong) exceptional collection $\L , \OO_E$, such that $\L$ is geometric of type $(m,k)$, with $k < d-1$.
\end{remark}

\subsection{Projective cones}\label{subsec:ProjCones}
In this last subchapter, we show that Corollary \ref{eq:cone-absorption} applies to the corresponding projective cones and we examine the semiorthogonal decomposition \eqref{eq:cone-absorp} in this case in more detail.
More concretely, we show that $\BB$ in \eqref{eq:cone-absorp} is generated by two exceptional line bundles, and that the finite-dimensional $\mathbb{k}$-algebra $A$ in the tilting case is isomorphic to the split algebra corresponding to the short exact sequence \eqref{eq:ses-algebra-A}, as discussed in Subsection \ref{subsec:split-case}.

Let $X$ be the projective cone over a smooth Fano variety $Z$ with projectively normal anticanonical embedding. 
That is $\omega_Z^{-1}$ is assumed to be a very ample line bundle, and we consider the projective cone of the embedding given by $\omega_Z^{-1}$.
Denote by $\pi \colon \tilde{X} \to X$ the blow-up at the singularity and let $\iota\colon E\hookrightarrow \tilde{X}$ be the exceptional divisor.
Recall that projection away from the vertex of $X$ is resolved by the blow-up $\tilde{X}$ of the vertex, and there is a $\P^1$-bundle structure $q \colon \tilde{X} \to Z$ 
given by the rank $2$ vector bundle $\omega_Z^{-1} \oplus \OO_Z$,
and, in particular, there is a section of $q$ inducing an isomorphism $Z \cong E$.
We have the following.

\begin{proposition}\label{prop:full-exc-ProjCones}
    Let $X $ be the projective cone over a smooth Fano variety $Z$ with projectively normal anticanonical embedding $Z\subset \P^r$ and suppose there is a full exceptional collection 
    \[
    \L = (\sL_1 , \ldots , \sL_m )
    \]
    of $ \OO_Z^\perp \subset\Db(Z)$.
        Moreover, let $q \colon \tilde{X} \to  Z$ be the $\P^1$-bundle structure of the blow-up $\pi\colon \tilde{X} \to X$ of $X $ at the vertex as above. Set
    \[
    \E := q^\ast \L ( - E ) = q^\ast \L \otimes \OO_{\tilde{X}} (-E) .
    \]
Then $\E $ and $\bT:=\bT_{\L}$ are strongly adherent to each other and there is are full exceptional collections 
\begin{equation}\label{eq:full-exc-coll-E-TE}
    \sE_1 , \ldots , \sE_m , \bT \sE_1 , \ldots , \bT \sE_m , \pi^\ast\OO_{X} , \pi^\ast \OO_{X} (1) 
\end{equation}
and 
\begin{equation}\label{eq:full-exc-coll-E-F}
    \sE_1 , \ldots , \sE_m , \sF_1 , \ldots , \sF_m , \pi^\ast\OO_{X} , \pi^\ast \OO_{X} (1) 
\end{equation}
of $\Db( \tilde{X} )$.
Here, $\OO_X(1)$ is the restriction of $\OO_{\P^{r+1}}(1)$ via $X\subset \P^{r+1}$ and $\F  : = \bS^{-1}_{\bT q^\ast \L (-E)} \bT \E$ (cf. Lemma \ref{lem:hom-K-F}).
\end{proposition}

\begin{proof}
The fact that $\E$ and $\bT$ are strongly adherent follows readily by applying $\iota^! (-) \cong \iota^\ast(-)(E)[-1]$ onto $\E$ and using Theorem \ref{thm:strong-adherence-equiv}, and using that $Z \cong E \xhookrightarrow{\iota} \tilde{X}$ is a section of $q \colon \tilde{X} \to Z$.

    It remains to show that the right orthogonal of $\tilde{\AA} = \langle \E , \bT ( \E ) \rangle $ is generated by $\OO_{\tilde{X}} , \pi^\ast\OO_{X} (1)$.
    First, by projective bundle formula applied to $q\colon \tilde{X} \to X$ there is a semiorthogonal decomposition
    \begin{align}\label{eq:decomp-antican}
        \Db ( \tilde{X} ) & \cong \langle q^\ast \Db(Z) , q^\ast \Db(Z) \otimes \OO(1) \rangle \notag \\
        & \cong \langle  q^\ast \sL_1 , \ldots , q^\ast \sL_m , \OO_{\tilde{X}} , \OO (1) ,  q^\ast \sL_1 ( -K_Z )\otimes \OO(1) , \ldots , q^\ast \sL_m ( -K_Z )\otimes \OO(1) \rangle \notag \\
        & \cong \langle  q^\ast \sL_1 ( -K_Z ) \otimes \OO(-1) , \ldots , q^\ast \sL_m ( -K_Z )\otimes \OO(-1) , q^\ast \sL_1 , \ldots , q^\ast \sL_m , \OO_{\tilde{X}} , \OO_{\tilde{X}} (1) \rangle ,
    \end{align}
where $K_Z$ denotes the canonical divisor of $Z$ and $\OO(1)$ is the tautological line bundle of $q \colon \tilde{X} \to Z$.
Note that we use in the second equivalence that $\Db(Z)$ has a full exceptional collection of the form 
\[
 \OO_Z ,   \sL_1 ( -K_Z ) , \ldots , \sL_m ( -K_Z ) ,
\]
and we use Serre duality in the third equivalence, where we note that the canonical bundle of $\tilde{X}$ is of the form $\omega_{\tilde{X}} \cong \OO(-2)$ by the canonical bundle formula for projective bundles.
Note further that $E$ is the section of $\tilde{X} \to Z $ corresponding to $\OO_Z$, that is $E = \P ( \OO_Z ) \subset \tilde{X}$.
We define the section $H: = \P ( \omega_Z^{-1} ) \subset \tilde{X}$, and we note that 
there are  isomorphisms of line bundles
\begin{equation}\label{eq:lbundle-eqn}
    \OO_{\tilde{X}} (-E ) \cong \OO(-1) \otimes q^\ast \OO_Z ( -K_Z )  \in \Pic ( \tilde{X} )   \quad \text{and} \quad \OO_{\tilde{X}} (H) \cong \OO_{\tilde{X}} (1)\in \Pic ( \tilde{X} ) .
\end{equation}
Since $E = \P( \OO_Z )$ and $H = \P ( \omega_Z^{-1} )$ are disjoint to each other, the image of $H $ in $X$ via $\pi $ (which we again denote by $H$, by abuse of notation) is Cartier.
One can check that there is an isomorphism of line bundles $\OO_X (H) \cong \OO_X (1)$, and thus $\OO_{\tilde{X}} (H) \cong \pi^\ast \OO_X(1)$.
Now, using this isomorphism of line bundles and \eqref{eq:lbundle-eqn}  
in the semiorthogonal decomposition \eqref{eq:decomp-antican}, we obtain the semiorthogonal decomposition
\begin{equation}\label{eq:decomp-antican2}
    \Db ( \tilde{X} ) \cong \langle  q^\ast \sL_1 ( -E ) , \ldots , q^\ast \sL_m ( -E ) , q^\ast \sL_1 , \ldots , q^\ast \sL_m , \pi^\ast\OO_{X} , \pi^\ast\OO_{X} (1) \rangle  .
\end{equation}

Next, we show that there is an isomorphism of objects $\bT_\L  \sE_i  \cong \sE_i (E)$ for all $i = 1 , \ldots , m$.
Indeed, we tensor the short exact sequence of sheaves
\[
0 \to \OO_{\tilde{X}} \to \OO_{\tilde{X}} (E) \to \OO_E (E) \to 0 
\]
by $\sE_i$ in $\Db ( \tilde{X} )$ and we obtain a distinguished triangle
\[
\sE_i \to \sE_i (E) \to \iota_\ast \sL_i \xrightarrow{\alpha_i} \sE_i [1] 
\]
in $\Db ( \tilde{X} )$, where we used that $\sE_i $ is defined as $q^\ast \sL_i(-E)$.
It is clear that $\alpha_i $ is a nontrivial morphism in $\Db ( \tilde{X} )$, as $\sE_i (E)$ is an exceptional (and thus indecomposable) object in $\Db( \tilde{X} )$.
Furthermore, we have a distinguished triangle 
\[
\sE_i \to \bT_{\L} \sE_i  \to \iota_\ast \sL_i \xrightarrow{\epsilon_i } \sE_i [1]
\]
see \eqref{eq:main-triangle}, and there is only one morphism $\epsilon_i $ from $\iota_\ast \sL_i \to \sE_i [1]$ (up to scalar) by Lemma \ref{lem:hom-computation} \eqref{eq:hom-computation-iso1}.
It follows thus that $\epsilon_i $ and $\alpha_i $ agree (up to scalar) and so comparing these two distinguished triangles implies that there is an isomorphism $\bT_\L  \sE_i  \cong \sE_i (E)$, for all $i$.
We can yet again rewrite the decomposition \eqref{eq:decomp-antican2}
as
\begin{equation}\label{eq:exc-coll-E-TE}
    \Db( \tilde{X} ) \cong \langle \E , \bT_{\L} \E , \pi^\ast \OO_{X} , \pi^\ast\OO_{X} (1) \rangle ,
\end{equation}
which proves \eqref{eq:full-exc-coll-E-TE}. Finally,
Since $\F$ is a full exceptional collection of $\thick\bT_{\L} ( \E )$, the full exceptional collection \eqref{eq:full-exc-coll-E-F} follows from \eqref{eq:full-exc-coll-E-TE}. 
This concludes the proof of the proposition.
\end{proof}

\begin{corollary}\label{eq:projcone-absorption}
    Let $X$ be the projective cone over a smooth Fano variety $Z$ with projectively normal anticanonical embedding and suppose there is a full geometric exceptional collection $\L$ of type $(m,k)$
    of $ \OO_Z^\perp \subset\Db(Z)$ with $k \leq d-1$.
    Then there exists an admissible semiorthogonal decomposition 
    \[
    \Db (X) \cong \langle \DD_{\mathrm{fg}} ( A^\bullet  ) , \OO_X , \OO_X (1) \rangle , 
    \]
    where 
    $A^\bullet $ is a finite-dimensional $\dg$ $\mathbb{k}$-algebra with cohomology concentrated in degrees $0$ and $k+1-d$ and isomorphic (as $\Hom$-spaces) to $\End ( \sL )$ and $\Hom ( \sL , \bS^{-1} \sL [k-1] )$, respectively, where $\sL = \bigoplus_i \sL_i$.
\end{corollary}

\begin{proof}
This follows from Proposition \ref{prop:full-exc-ProjCones} and Corollary \ref{eq:cone-absorption}.
 \end{proof}

Let us now discuss the tilting case (i.e. the case $k=d-1$ in Corollary \ref{eq:projcone-absorption}) in more detail.
We need the following lemma.

\begin{lemma}\label{lem:help}
We use the same notation as in Proposition \ref{prop:full-exc-ProjCones}.
For all $1\leq i \leq m$, the exceptional object $\sF_{i} [d-2]$ fits into the distinguished triangle 
\begin{equation}\label{eq:left-mutation}
    \sF_{i}[d-2] \to \OO_{\tilde{X}} \otimes \Hom^\bullet ( \OO_{\tilde{X}} , q^\ast \sL_i (-K_Z ) ) \xrightarrow{\mu_i} q^\ast \sL_i (-K_Z ) 
\end{equation}
    in $\Db( \tilde{X} )$,
    where $\mu_i$ is the evaluation map.
    In other words, $\sF_i [d-2]$ is isomorphic to the left mutation of $q^\ast \sL_i (-K_Z )$ through $\OO_{\tilde{X}}$.
\end{lemma}

\begin{proof}
We consider the full exceptional collection 
\[
   q^\ast \sL_1 ( -E )  , \ldots , q^\ast \sL_m ( -E ) ,  \sF_1 , \ldots , \sF_m , \OO_{\tilde{X}} ,  \OO_{\tilde X} (1) .
\]
of $\Db ( \tilde{X} )$, see \eqref{eq:full-exc-coll-E-F}.
We prove the lemma inductively.
More concretely, let us denote by $\sB_j$ the right mutation of $\sF_{j} [d-2]$ through $\OO_{\tilde{X}}$ in $\Db( \tilde{X})$, for all $j$.
We show by induction that $\sB_j $ is isomorphic to $\sL_j (-K_Z)$ in $\Db( \tilde{X} )$, for all $j$.
    Assume thus by induction hypothesis that there are isomorphism $ \sB_m \cong q^\ast \sL_m (-K_Z) , \ldots , \sB_{j+1} \cong q^\ast \sL_{j+1} (-K_Z)$.
    We obtain thus a full exceptional collection
    \[
     q^\ast \sL_1 ( -E )  , \ldots , q^\ast \sL_m ( -E ) ,  \sF_1 , \ldots , \sF_{j-1} , \OO_{\tilde{X}} , \sB_j , q^\ast \sL_{j+1} (-K_Z)  , \ldots  q^\ast \sL_m (-K_Z)  , \OO_{\tilde X} (1) 
    \]
     by mutating the exceptional collection above.
     In the following, we show that $\sB_j \cong q^\ast \sL_j (-K_Z)$.

    Let us first show that $q^\ast \sL_j (-K_Z)$ is contained in $\langle \sB_j\rangle $, for all $j$.
    To see this, it is enough to show that the following graded $\Hom$-spaces in $\Db(\mathbb{k})$ vanish
    \begin{equation}\label{eq:projbundleprop1}
        \Hom^\bullet ( q^\ast \sL_j (-K_Z) , q^\ast \sL ( -E ) ) = 0 \quad \text{and} \quad \Hom^\bullet ( q^\ast \sL_j (-K_Z) , \sF_i ) = 0
    \end{equation}    
    for all $1\leq i \leq j -1$,
    as well as that the vanishing of graded $\Hom $-spaces 
    \begin{equation}\label{eq:projbundleprop2}
        \Hom^\bullet ( q^\ast \sL_{l} (-K_Z) , q^\ast \sL_j (-K_Z) ) = 0 \quad \text{and} \quad \Hom^\bullet ( \OO_{\tilde X}(1) , q^\ast \sL_j (-K_Z) ) = 0
    \end{equation}
    holds,
    for all $j+1 \leq l\leq m$.
    For the vanishing of  
    the first term of \eqref{eq:projbundleprop1},
    we note that $q^\ast \L ( -E ) $ is contained in $q^\ast \Db(Z) \otimes \OO(-1)$ (by the isomorphism of line bundles $\OO_{\tilde{X}}(-E) \cong \OO(-1) \otimes q^\ast \OO_Z (-K_Z )$, see \eqref{eq:lbundle-eqn}) and $ q^\ast \sL_j (-K_Z)$ is contained in $q^\ast \Db(Z)$.
    Thus by semiorthogonality between $q^\ast \Db(Z) \otimes \OO(-1)$ and $q^\ast \Db(Z) $ (projective bundle formula), we get the desired vanishing of graded $\Hom$-spaces.
    Similarly, 
    the the vanishing of 
    the second term of \eqref{eq:projbundleprop2}
    follows from the isomorphism of line bundles $\OO_{\tilde{X}}(H) \cong \OO(1)$ and the semiorthogonality between $q^\ast \Db(Z)$ and $q^\ast \Db(Z) \otimes \OO(1)$.
    The vanishing of 
    the first term of \eqref{eq:projbundleprop2}
    follows easily from the fact that $\L$ is an exceptional collection.
    
    It remains to show the vanishing of the second term of \eqref{eq:projbundleprop1}.
        Let us show more generally that for all $1\leq i \leq j \leq m$ the following isomorphism 
    \begin{equation}\label{eq:hom-vanishing}
      \Hom^\bullet ( q^\ast \sL_j (-K_Z) , \sF_{i}[d-2] ) \cong \begin{cases}
        \mathbb{k}[-1],& \text{if $j=i$} \\
        0.& \text{if $j>i$}
    \end{cases}  
    \end{equation}
    in $\Db(\mathbb{k})$ holds.
    To show this, we note that tensoring the short exact sequence 
    \[
    0\to \OO_{\tilde{X}} \to \OO_{\tilde{X}} (E) \to \iota_\ast \OO_Z ( K_Z ) \to 0
    \]
    by $p^\ast \sL_j ( -K_Z )$ in $\Db( \tilde{X} )$ 
    yields a distinguished triangle 
    \[
    p^\ast \sL_j ( -K_Z ) \to p^\ast \sL_j ( 1 ) \to \iota_\ast \sL_j 
    \]
    in $\Db( \tilde{X} )$,
    where we used the isomorphism of line bundles $\OO_{\tilde{X}}(E) \cong \OO_{\tilde{X}}( 1 ) \otimes p^\ast \OO_Z(K_Z)$ in the second term.
    Applying $\Hom^\bullet ( - , \sF_i [d-2])$, we obtain a distinguished triangle 
    \begin{equation}\label{eq:Lj-Fi}
        \Hom^\bullet ( \iota_\ast\sL_j , \sF_i [d-2] ) \to \Hom^\bullet ( p^\ast \sL_j ( 1 ) , \sF_i [d-2] ) \to \Hom^\bullet ( p^\ast\sL_i (-K_Z ) , \sF_i [d-2] )
    \end{equation}
    in $\Db(\mathbb{k})$.
    The first term vanishes for $j>i$ and is isomorphic to $\mathbb{k}[-2]$ by Lemma \ref{lem:hom-K-F} \eqref{lem:hom-K-F-a}.
    For the second term, we note that $\sF_i $ is contained in $\thick q^\ast \L  \subset q^\ast \Db(Z) $ (recall that there is an identification $\bT  (\E ) \cong q^\ast \L$).
    On the other hand, $p^\ast \sL_j ( 1 )$ is contained in $q^\ast \Db(Z) \otimes \OO(1)$, and hence by semiorthogonality between $q^\ast \Db(Z)$ and $q^\ast \Db(Z) \otimes \OO(1)$, we have that the second term of \eqref{eq:Lj-Fi} vanishes.
    We conclude \eqref{eq:hom-vanishing}.

We have thus shown \eqref{eq:projbundleprop1} and \eqref{eq:projbundleprop2} and thus that $q^\ast\sL_j (-K_Z) $ is contained in $\langle \sB_j\rangle $.
Since $q^\ast\sL_j (-K_Z) $ is exceptional, we have that $\sB_{j} \cong q^\ast\sL_j (-K_Z) [N]$ in $\Db(\tilde{X})$, for some $N$ in $\Z$.
    Since $\sB_{j}$ is by definition the right mutation of $\sF_{j}[d-2]$ through $\OO_{\tilde{X}}$, there is a defining distinguished triangle of $\sB_{j}$
    \[
    \sF_{j}[d-2] \to \OO_{\tilde{X}} \otimes \Hom^\bullet ( \sF_{j}[d-2] , \OO_{\tilde{X}} )^\ast \to \sB_{j} 
    \]
    given by the co-evaluation morphism.
    Applying $\Hom^\bullet ( q^\ast\sL_j (-K_Z) , - )$ onto this distinguished triangle, we obtain
    \begin{align*}
        \Hom^\bullet ( q^\ast\sL_j (-K_Z) ,\sF_{j}[d-2] ) & 
        \to \Hom^\bullet ( q^\ast\sL_j (-K_Z) , \OO_{\tilde{X}} ) \otimes \Hom^\bullet ( \sF_{j}[d-2] , \OO_{\tilde{X}} )^\ast \\
        &\to \Hom^\bullet ( q^\ast\sL_j (-K_Z) , \sB_{j} )
    \end{align*}
    in $\Db(\mathbb{k})$.
    Using \eqref{eq:hom-vanishing} for the case $j=i$ in the first term, together with the fact that the second term $\Hom^\bullet ( q^\ast\sL_j (-K_Z) , \OO_{\tilde{X}} )$ vanishes (as $\sL_j$ is contained in $\langle\OO_Z\rangle^\perp$), we get that $\Hom^\bullet ( q^\ast\sL_j (-K_Z) , \sB_{j}  ) \cong \mathbb{k}[0]$.
    In particular, this implies that $N=0$.

This completes the inductive proof
        that the right mutation of $\sF_{j}$ through $\OO_{\tilde{X}}$ is $q^\ast\sL_j (-K_Z)$, for all $j$.
    But this also means that the left mutation of $q^\ast\sL_j (-K_Z)$ through $\OO_{\tilde{X}}$ is just $\sF_{j}$, for all $j$.
    The proof of the lemma is thus complete.
\end{proof}

\begin{theorem}\label{thm:projcone-split}
    Let $X$ be the projective cone over a smooth $(d-1)$-dimensional very strong Fano variety $Z$ (cf. Definition \ref{def:verystrong-Fano}) with respect to the anticanonical embedding and suppose the corresponding (full geometric) exceptional collection $\L  = ( \sL_1 , \ldots , \sL_m ) $ of $\OO_Z^\perp $ 
    is given by coherent sheaves.
    Then there exists an admissible semiorthogonal decomposition 
    \begin{equation}\label{eq:porjconesod}
        \Db (X) \cong \langle \Db ( A  ) , \OO_X , \OO_X (1) \rangle , 
    \end{equation}
    where 
    $A $ is the finite-dimensional split $\mathbb{k}$-algebra 
    \[
A \cong  \End ( \sL  ) \oplus  \Hom^\bullet ( \sL , \bS^{-1} \sL [d-2] ) 
    \]
    as described in \eqref{E:TensorModSquare}.
\end{theorem}

\begin{proof}
Recall that, by definition of a very strong Fano variety, the exceptional collection $\L$ is geometric of type $(m, d-1)$.
 Let $\pi \colon \tilde{X}\to X$ be the blow-up of $X$ at the singularity and let $q\colon \tilde{X} \to Z$ the corresponding $\P^1$-bundle structure on $\tilde{X}$.
By Corollary \ref{eq:projcone-absorption}, there exists a semiorthogonal decomposition of the form \eqref{eq:porjconesod}, with a finite-dimensional $\mathbb{k}$-algebra $A$ fitting into an exact sequence
\[
0\to \Hom ( \sL , \bS^{-1} \sL [d-2] ) \to A \xrightarrow{q} \End ( \sL  ) \to 0 
    \]
 where $q$ is a morphism of $\mathbb{k}$-algebras.
 Here, $A=\End (\sG )$, where $\sG = \bigoplus_i \sG_i$ in $\Db( \tilde{X})$ is as constructed in Proposition \ref{prop:G} with respect to the geometric exceptional collection $\E = ( q^\ast \sL_1 (-E) , \ldots , q^\ast \sL_m (-E) )$ of type $(m, d-1 )$. 

To prove the theorem, we need to show that the short exact sequence of $\End ( \sL )$ bi-modules (see Proposition \ref{prop:algebra-extensions} \eqref{prop:algebra-extensions-2}) above splits.
We are going to show the splitting by constructing a split morphism of algebras $\psi \colon \End ( \sL ) \to A$ below.
    Let us first recall the construction of the $\sG_i$.
    We consider the full exceptional collection \eqref{eq:full-exc-coll-E-F} of $\Db( \tilde{X} )$ (together with the isomorphism of line bundles $\OO_{\tilde{X}} (1 ) \cong \pi^\ast \OO_X(1)$):
    \[
    q^\ast \sL_1 (-E) , \ldots , q^\ast \sL_m (-E) , \sF_1 , \ldots, \sF_{m} , \OO_{\tilde{X}} , \OO_{\tilde{X}}(1) .
    \]
     Further, 
     there is a distinguished (and unique up to scalar, since $\L$ is geometric and of type $(m, d-1 )$) morphism $\nu_i \colon q^\ast \sL_i (-E) \to \sF_i [d-1]$ for all $i$ and we call $\sG_i$ its cocone (cf. Proposition \ref{prop:G}).

    Let us now give an alternative description of 
    $\sG_i$, exploiting the fact that $X$ is a projective cone (i.e. the description of the $\sG_i$ below might not work for general $X$).
Namely, we have that $\sF_{i}[d-2]$ fits into a distinguished triangle
    \[
    \sF_{i}[d-2] \to \OO_{\tilde{X}} \otimes \mathrm{H}^0 ( \sL_i (-K_Z) ) \xrightarrow{\mu_i} q^\ast \sL_1 (-K_Z) 
    \]
    in $\Db( \tilde{X} )$ by Lemma \ref{lem:help},
    where $\mu_i$ is the evaluation map, for all $1\leq i \leq m$.
    We use here that $\L , \OO_Z$, and thus $\OO_Z , \L(-K_Z)$, is a geometric exceptional collection of type $(m+1,d)$ by Corollary \ref{cor:geometric-subcollection} and, in particular, there is no higher sheaf cohomology of $\sL_i(-K_Z)$.
    Moreover, we consider the following commutative diagram.
    \begin{equation}\label{eq:octa-Gi}
\xymatrix{
    q^\ast \sL_i (-E) \ar@{<--}[d]^{\beta_i}\ar@{->}[r]^{\cong} & q^\ast \sL_i (-E) \ar@{<-}[d]^{0 \oplus \id } &   \\
    C_i  \ar@{<--}[d]^{\alpha_i}\ar@{->}[r] & \OO_{\tilde{X}} \otimes \mathrm{H}^0 ( \sL_i (-K_Z) ) \oplus q^\ast \sL_i (-E) \ar@{<-}[d]^{(\id , 0 )} \ar@{->}[r]^-{\mu_i \oplus w } &  q^\ast \sL_i (-K_Z) \ar@{<-}[d]^{=} \ar@{->}[r] & C_i [1] \ar@{<--}[d] \\
 \sF_{i}[d-2] \ar@{->}[r] &\OO_{\tilde{X}} \otimes \mathrm{H}^0 ( \sL_i (-K_Z) ) \ar@{->}[r]^-{\mu_i} &  q^\ast \sL_i (-K_Z ) \ar@{->}[r] & \sF_{i}[d-1]} 
\end{equation}
 Here $w$ denotes the section in $\mathrm{H}^0 ( \OO ( 1 ) )$ corresponding to the section $H = \P( \omega_Z^{-1} )$ and the morphism $w \colon q^\ast\sL_i (-E) \to q^\ast \sL_i (-K_Z)$ is induced by $w \in \mathrm{H}^0 ( \OO ( 1 ) )$ after identifying $q^\ast\sL_i (-E) \cong q^\ast \sL_i (-K_Z) \otimes \OO_{\tilde{}}(-1)$ (via the isomorphism of line bundles $\OO_{\tilde{X}}(-E) \cong q^\ast \OO_Z (-K_Z) \otimes \OO_{\tilde{X}}(-1)$).  
Moreover, the object $C_i[1]$ in $\Db( \tilde{X} )$ is the cone of the morphism $\mu_i \oplus w$ and the dotted arrow $\alpha_i$ is a morphism induced by the commutative diagram of the horizontal distinguished triangles.
By the octahedral axiom, the cone of $\alpha_i$ is isomorphic to $q^\ast \sL_i (-E)$ in $\Db( \tilde{X} )$.
Since $C_i$ is non-trivial (i.e. not isomorphic to the direct sum of $\OO_{\tilde{X}} \otimes \mathrm{H}^0 ( \sL_i (-K_Z) ) \oplus q^\ast \sL_i (-E)$ and $q^\ast\sL_i(-K_Z)[-1]$) and since $\Hom(\sL_i (-E) , \sF_{i}[d-1])\cong \mathbb{k}$, we get that $C_i\cong \sG_i$ in $\Db( \tilde{X} )$.
Thus $\sG_i$ fits into a distinguished triangle
\begin{equation}\label{eq:two-term}
    \sG_i \to \OO_{\tilde{X}} \otimes \mathrm{H}^0 ( \sL_i (-K_Z) ) \oplus q^\ast\sL_i(-E)  \xrightarrow{\mu_i \oplus w} q^\ast\sL_i (-K_Z) .
\end{equation}
Since the second and the third object in this distinguished triangle are coherent sheaves, as well as the morphism $\mu_i \oplus w$ is a morphism of sheaves, the cone complex of $\mu_i \oplus w$ is (by definition of the cone complex) the following two-term complex 
and thus there is an isomorphism
\begin{equation}\label{eq:two-term-complex}
    \sG_i \cong \{ \OO_{\tilde{X}} \otimes \mathrm{H}^0 ( \sL_i (-K_Z) ) \oplus q^\ast\sL_i(-E)  \xrightarrow{\mu_i \oplus w} q^\ast\sL_i (-K_Z) \}
\end{equation}
in $\Db(\tilde{X})$.
Note that the term $q^\ast\sL_i (-K_Z)$ is concentrated in degree $0$.

Let us now define the algebra morphism $\psi\colon \End(\sL) \to A$ using the description \eqref{eq:two-term-complex} of $\sG_i$.
Namely, for any morphism $f$ in $\Hom ( \sL_i , \sL_j )$ we define the assignment $\psi \colon \Hom ( \sL_i , \sL_j ) \to \Hom ( \sG_i , \sG_j )$ by\footnote{Note that we are abusing notation here. Indeed, $\id \otimes f $ is $\id\otimes f (-K_Z)$, $f$ in the lower right component of the matrix is $q^\ast f (-E)$, and the right-most $f$ is $q^\ast f(-K_Z)$.}:

    \begin{equation*}\label{E:farrows}\begin{tikzpicture}[description/.style={fill=white,inner sep=2pt}, baseline=(current  bounding  box.center)]

    \matrix (n) [matrix of math nodes, row sep=1.5em,
                 column sep=2em, text height=1.5ex, text depth=0.25ex,
                 inner sep=2pt, nodes={inner xsep=0.3333em, inner
ysep=0.3333em}] at (0, 0)
    {   
    && \{ \OO_{\tilde{X}} \otimes \mathrm{H}^0 ( \sL_i (-K_Z) ) \oplus q^\ast\sL_i(-E) && q^\ast\sL_i (-K_Z) \}  \\
    f  \\
    &&\{ \OO_{\tilde{X}} \otimes \mathrm{H}^0 ( \sL_j (-K_Z) ) \oplus q^\ast\sL_j (-E) && q^\ast\sL_j (-K_Z) \}  \\};
 
  \draw[->] ($(n-3-3.east) + (0,0mm)$) to node[midway, yshift=8]{$\mu_j \oplus w$} ($(n-3-5.west) + (0mm,0mm)$);
\draw[->] ($(n-1-3.east) + (0,0mm)$) to node[midway, yshift=8]{$\mu_i \oplus w$} ($(n-1-5.west) + (0mm,0mm)$);

\draw[->] ($(n-1-5.south) + (0,0mm)$) to node[midway, xshift=8]{$f$} ($(n-3-5.north) + (0mm,0mm)$);
\draw[->] ($(n-1-3.south) + (0,0mm)$) to node[midway, xshift=38]{$\begin{pmatrix} \id \otimes f & 0 \\ 0 & f \end{pmatrix} $} ($(n-3-3.north) + (0mm,0mm)$);

\draw[|->] ($(n-2-1.east) + (5mm,0mm)$) to ($(n-2-1.east) + (20mm,0mm)$);

\end{tikzpicture}\end{equation*}
It is clear that this assignment sends $f$ to a morphism of complexes.
Moreover, it is straightforward to check that this is a $\mathbb{k}$-algebra homomorphism (i.e. behaves well with respect to composition of morphisms).
Finally, we check that $\psi(f)$ evaluated at $q= (\ast\cdot \beta_i )^{-1} ( \beta_j \cdot (-) )$ (see Proposition \ref{prop:G} \eqref{lem:G-c}).
By \eqref{eq:octa-Gi}, we see that the 
morphism $\beta_j $ in $\Db( \tilde{X} )$ is induced by the morphism of complexes
\begin{equation}
\xymatrix{
     \OO_{\tilde{X}} \otimes \mathrm{H}^0 ( \sL_j (-K_Z) ) \oplus q^\ast \sL_j (-E) \ar@{->}[d]^{( 0 , \id )} \ar@{->}[r]^-{\mu_j \oplus w } &  q^\ast \sL_j (-K_Z) \ar@{->}[d]   \\
  q^\ast \sL_j (-E)  \ar@{->}[r] &  0 } 
\end{equation}
for all $j$.
From this description of $\beta_j$ it follows that there is a commutative diagram
    \begin{equation*}
\xymatrix{
     \sG_i \ar@{->}[r]^-{\beta_i} \ar@{->}[d]_{\psi(f)} & q^\ast \sL_i (-E) \ar@{->}[d]_{f}\\
 \sG_j \ar@{->}[r]^-{\beta_j} & q^\ast \sL_j (-E)} ,
\end{equation*}
in $\Db( \tilde{X} )$.
Hence, we have the chain of equalities
\[
q( \psi (f) ) = (\ast\cdot \beta_i )^{-1} ( \beta_j \cdot \psi (f) )  = (\ast\cdot \beta_i )^{-1} ( f \cdot \beta_i ) = f \in \End ( \sL ) .
\]
In other words, $q \colon \End (\sG ) \to \End (\sL)$ is a split homomorphism of $\mathbb{k}$-algebras.
This completes the proof of the theorem.
\end{proof}

\begin{remark}
    The category $\OO_Z^\perp$ admits geometric collections of line bundles of type $(d-1, d-1)$ for $Z = \P^{d-1}$ (see \cite[Example 2.9]{bondal-polishchuk} or \cite[Example 3.3 (a)]{BridgelandStern}, and e.g. Proposition \ref{P:dBlockGeometricRed}) and of type (3,2) $Z = \P^1\times \P^1$ (see \cite[Example 3.3 (b)]{BridgelandStern} or Example \ref{ex:p1xp1} and  Proposition \ref{P:dBlockGeometricRed}), respectively.
    Hence we can apply Theorem \ref{thm:projcone-split} using Example \ref{ex:delpezzo-projnormal}.
    Moreover, one can check that the morphism $\mu_i \oplus w$ appearing in the description of the $\sG_i$ as a two-term complex
    \[
    \sG_i \cong \{ \OO_{\tilde{X}} \otimes \mathrm{H}^0 ( \sL_i (-K_Z) ) \oplus q^\ast\sL_i(-E)  \xrightarrow{\mu_i \oplus w} q^\ast\sL_i (-K_Z) \}
    \]
    in the proof of Theorem \ref{thm:projcone-split} (see \ref{eq:two-term-complex}) is surjective, for all $i$.
    This implies for all $i$ that $\sG_i$ is the kernel of $\mu_i \oplus w$ and, in particular, a locally free sheaf.
\end{remark}

The following example is different from all other examples we consider. Namely, these geometric sequences  don't come from $d$-block exceptional sequences!

\begin{example}\label{ex:del-pezzo}
Let us  discuss examples of Theorem \ref{thm:projcone-split} for explicit $Z$.
Let $Z$ be a del Pezzo surface of degree $8$, or $7$.
That is, $Z$ is a  blow-up of $\mathbb{P}^2$ in $1$ and $2$ points, respectively.
We note first that the anticanonical embedding is projectively normal, see Example \ref{ex:delpezzo-projnormal}.
We explain further that $Z$ admits a full geometric exceptional collection of type $(4, 3)$ respectively $(5,3)$ of the form $\L , \OO_Z$, where $\L$ is a geometric exceptional collection of line bundles of type $(3, 2)$, respectively $(4,2)$.
We use the exceptional collections in \cite[Propositions 6.1 \& 6.2]{King}. The descriptions of the endomorphism algebras in loc. cit. implies that $\End_Z(\sL)=\mathbb{k}Q_Z$ for a non-Dynkin quivers $Q_Z$ depicted in black below. Thus $\L$ is geometric by Lemma \ref{L:GeomHelicesandhigherreprinfinite}, since $\mathbb{k}Q_Z$ is $1$-representation infinite by Theorem \ref{T:Preprojn1}. Now $\End_Z(\sL \oplus \mathcal{O}_Z)$ has a quiver without oriented cycles (as endomorphism algebra of an exceptional collection) and differs from the quiver of $\End_Z(\sL)$ by a single (source) vertex. This implies that 
\begin{align}
   \gldim \End_Z(\sL \oplus \mathcal{O}_Z) \leq  \gldim \End_Z(\sL) +1 =2=\dim(Z),
\end{align}
which shows that $\L, \mathcal{O}_Z$ is geometric (Proposition \ref{prop:BHI}).

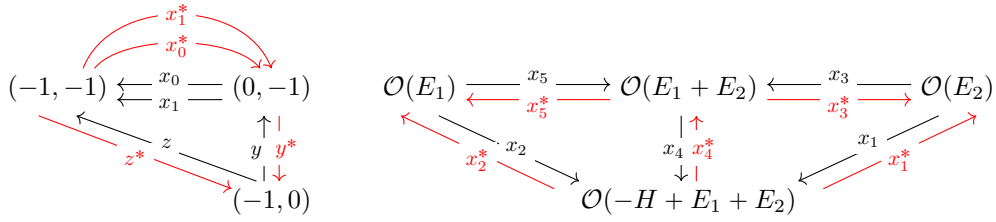
\begin{equation*}
\small
\begin{tikzpicture}[description/.style={fill=white,inner sep=2pt}]
 \matrix (n) [matrix of math nodes, row sep=2.5em,
                 column sep=2em, text height=1.5ex, text depth=1ex,
                 inner sep=2pt, nodes={inner xsep=0.3333em, inner
ysep=0.3333em}] at (-6, 0)
    {   (-1,-1)  && (0,-1) \\ 
     && (-1,0) \\
          };
    \draw[<-] ($(n-1-1.east) + (0,1mm)$)  to  node[fill=white, yshift=0.7mm, scale=0.8] [midway]{$x_{0}$}($(n-1-3.west) + (0mm,1mm)$) ;
    \draw[<-] ($(n-1-1.east) + (0,-1mm)$) to node[fill=white, yshift=-0.7mm, scale=0.8] [midway]{$x_{1}$} ($(n-1-3.west) + (0mm,-1mm)$);

 \draw[<-] ($(n-1-3.south) + (-1mm,0mm)$)  to  node[fill=white, xshift=-1mm,  yshift=-0.5mm, scale=0.8] [midway]{$y$}($(n-2-3.north) + (-1mm,0mm)$);
  \draw[->, red] ($(n-1-3.south) + (1mm,0mm)$)  to  node[fill=white, xshift=1mm, scale=0.8] [midway]{$y^*$}($(n-2-3.north) + (1mm,0mm)$);

  \draw[<-] ($(n-1-1.south) + (2.5mm,0mm)$)  to  node[fill=white, yshift=0.7mm, scale=0.8] [midway]{$z$}($(n-2-3.north) + (-2mm,0mm)$);
   \draw[->, red] ($(n-1-1.south) + (-2.5mm,0mm)$)  
to node[fill=white, yshift=-0.7mm, scale=0.8] [midway]{$z^*$}($(n-2-3.north) + (-5.3mm,-1mm)$);

   \draw[<-, red] ($(n-1-3.north west) + (5mm,-0mm)$) .. controls +(-3.7mm,+4.5mm)
and +(+3.7mm,+4.5mm) .. node[fill=white, scale=0.8, yshift=-0.5mm] [midway]{$x_0^*$} ($(n-1-1.north east) + (-3mm,0mm)$);

\draw[<-, red] ($(n-1-3.north west) + (6.5mm,-0mm)$) .. controls +(-3.7mm,+9.5mm)
and +(+3.7mm,+9.5mm) .. node[fill=white, scale=0.8] [midway]{$x_1^*$} ($(n-1-1.north east) + (-4.5mm,0mm)$);

 \matrix (m) [matrix of math nodes, row sep=2.5em,
                 column sep=2em, text height=1.5ex, text depth=1ex,
                 inner sep=2pt, nodes={inner xsep=0.3333em, inner
ysep=0.3333em}] at (1, 0)
    {  \mathcal{O}(E_1) && \mathcal{O}(E_1+ E_2) && \mathcal{O}(E_2) \\ 
    && \mathcal{O}(-H+E_1 +E_2)  \\ 
            };
    \draw[<-] ($(m-1-3.west) + (0,1mm)$)  to  node[fill=white, yshift=1mm, scale=0.8] [midway]{$x_{5}$}($(m-1-1.east) + (0mm,1mm)$) ;
        \draw[->, red] ($(m-1-3.west) + (0,-1mm)$)  to  node[fill=white, yshift=-1mm, scale=0.8] [midway]{$x_{5}^*$}($(m-1-1.east) + (0mm,-1mm)$) ;

    \draw[<-] ($(m-1-3.east) + (0,1mm)$)  to  node[fill=white, yshift=1mm, scale=0.8] [midway]{$x_{3}$}($(m-1-5.west) + (0mm,1mm)$) ;
        \draw[->, red] ($(m-1-3.east) + (0,-1mm)$)  to  node[fill=white, yshift=-1mm, scale=0.8] [midway]{$x_{3}^*$}($(m-1-5.west) + (0mm,-1mm)$) ;

        \draw[->] ($(m-1-3.south) + (-1mm,0mm)$)  to  node[fill=white, xshift=-1mm, scale=0.8, yshift=-0.5mm] [midway]{$x_{4}$}($(m-2-3.north) + (-1mm,0mm)$) ;    
        \draw[<-, red] ($(m-1-3.south) + (1mm,0mm)$)  to  node[fill=white, xshift=1mm, scale=0.8] [midway]{$x_{4}^*$}($(m-2-3.north) + (1mm,0mm)$) ;

              \draw[->] ($(m-1-1.south) + (2.5mm,0mm)$)  to  node[fill=white, xshift=1mm, scale=0.8] [midway]{$x_{2}$}($(m-2-3.north west) + (1mm,0mm)$) ; 
              \draw[<-, red] ($(m-1-1.south) + (-2.5mm,0mm)$)  to  node[fill=white, yshift=-1mm, scale=0.8] [midway]{$x_{2}^*$}($(m-2-3.north west) + (-2.3mm,-1mm)$) ;

                    \draw[->] ($(m-1-5.south) + (-2.5mm,0mm)$)  to  node[fill=white, yshift=1mm, scale=0.8] [midway]{$x_{1}$}($(m-2-3.north east) + (-1mm,0mm)$) ; 
\draw[<-, red] ($(m-1-5.south) + (2.5mm,0mm)$)  to  node[fill=white, yshift=-1mm, scale=0.8] [midway]{$x_{1}^*$}($(m-2-3.north east) + (2.3mm,-1mm)$) ;
\end{tikzpicture}\end{equation*}
The notation of the vertices comes from the $\P^1$-bundle structure and from the blow up structure, respectively (recall here that we consider opposite algebras).
Then the $\mathbb{k}$-algebra $A$ appearing in Theorem \ref{thm:projcone-split} is isomorphic to $\mathbb{k} \overline{Q}/I$ by Corollary \ref{C:QuotientPreproj2}.
Here $\overline{Q}$ is the double quiver obtained by adding the red arrows, and the two-sided ideal $I$ is generated by all paths $p$ in $\overline{Q}$ with at least two ${^\ast}$-arrows, together with  $\sum_{a\in Q_1} [ a^\ast , a ]$.
\end{example}

\begin{example}\label{ex:projspace-quadric}
We list examples of very strong Fano varieties of dimension $(d-1)$ -- in particular, Theorem \ref{thm:projcone-split} yields Kawamata type semiorthogonal decompositions over their projective cones (with respect to the anticanonical embedding). In all these cases the anticanonical embedding is projectively normal by Example \ref{ex:delpezzo-projnormal}.

\begin{enumerate}[label=(\alph*)]
    \item projective spaces $\P^{d-1}$ (by \cite{beilinson} combined with Propositions \ref{prop:BHI}  \& \ref{P:dBlockGeometricRed}). The split algebra $A$ from Theorem \ref{thm:projcone-split} is described as an explicit quiver algebra in Corollary \ref{Cor:description_R_d}. 
    \item smooth quadric hypersurfaces (by \cite{Kap}, together with 
  Propositions \ref{prop:BHI}  \& \ref{P:dBlockGeometricRed} -- in even dimensions $(d-1)$ one block of the $d$-block exceptional sequence consists of the two spinor bundles).
  \item del Pezzo surfaces of degree $8$ and $7$ (see Example \ref{ex:del-pezzo})
  \item del Pezzo surfaces of degree $6$ and $5$ (see \cite[Proof of Proposition 4.2, case (3) and (4)]{karpov-nogin} and \cite[Example 8.6 (c)]{BridgelandStern} combined with Proposition \ref{P:dBlockGeometricRed})
\item smooth del Pezzo threefolds of degree  $5$ (by \cite{orlov-V5}, together with
  Propositions \ref{prop:BHI} \&  \ref{P:dBlockGeometricRed})
  \item products $Z=Z_1 \times \cdots \times Z_t$ of very strong Fanos $Z_i$ are again very strong, by Lemma \ref{L:productsofgeom}. This is already interesting for products of projective spaces $\P^{n_1} \times \cdots \times \P^{n_t}$, see Example \ref{E:prprojhered} for a split algebra $A$ for $Z=\P^1 \times \P^1$ and Example \ref{ex:P2xP1-cone} for $Z=\P^2 \times \P^1$.
\end{enumerate} 
\end{example}

\begin{example}\label{ex:P2xP1-cone}
We discuss  a particular case of $Z$ being a product of projective spaces, see Example \ref{ex:projspace-quadric}.
More concretely, let $Z$ in Theorem \ref{thm:projcone-split} be the product space $\P^2\times \P^1$.
In this case,  a $\mathbb{k}$-algebra $A$ appearing in Theorem \ref{thm:projcone-split} can be described by the quiver
    
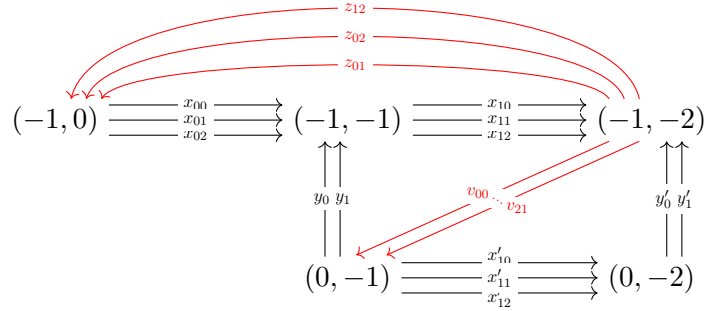
\begin{equation*}\begin{tikzpicture}[description/.style={fill=white,inner sep=2pt}]
 \matrix (n) [matrix of math nodes, row sep=4em,
                 column sep=2em, text height=1.5ex, text depth=0.25ex,
                 inner sep=2pt, nodes={inner xsep=0.3333em, inner
ysep=0.3333em}] at (-5, 0)
    { 
    (-1, 0) &&& (-1,-1) &&& (-1,-2) \\
     &&& (0,-1) &&& (0,-2)  \\ 
      };

    \draw[->] ($(n-1-4.east) + (0,2mm)$)  to  node[scale=0.6, fill=white] [midway]{$x_{10}$}($(n-1-7.west) + (0mm,2mm)$) ;
    \draw[->] ($(n-1-4.east) + (0,-2mm)$) to node[fill=white, scale=0.6] [midway]{$x_{12}$} ($(n-1-7.west) + (0mm,-2mm)$);
    \draw[->] ($(n-1-4.east) + (0,0mm)$)  to  node[fill=white, scale=0.6] [midway]{$x_{11}$}($(n-1-7.west) + (0mm,0mm)$) ;

     \draw[->] ($(n-1-1.east) + (0,2mm)$)  to  node[scale=0.6, fill=white] [midway]{$x_{00}$}($(n-1-4.west) + (0mm,2mm)$) ;
    \draw[->] ($(n-1-1.east) + (0,-2mm)$) to node[fill=white, scale=0.6] [midway]{$x_{02}$} ($(n-1-4.west) + (0mm,-2mm)$);
    \draw[->] ($(n-1-1.east) + (0,0mm)$)  to  node[fill=white, scale=0.6] [midway]{$x_{01}$}($(n-1-4.west) + (0mm,0mm)$) ;
    
    \draw[->] ($(n-2-4.east) + (0,2mm)$) to node[scale=0.6, fill=white, yshift=1.5mm] [midway]{$x'_{10}$}($(n-2-7.west) + (0mm,2mm)$) ;
    \draw[->] ($(n-2-4.east) + (0,-2mm)$) to node[fill=white, scale=0.6, yshift=-1mm] [midway]{$x'_{12}$} ($(n-2-7.west) + (0mm,-2mm)$);
    \draw[->] ($(n-2-4.east) + (0,0mm)$)  to  node[fill=white, scale=0.6] [midway]{$x'_{11}$}($(n-2-7.west) + (0mm,0mm)$) ;
    
     \draw[->] ($(n-2-4.north) + (-1mm,0mm)$)  to  node[scale=0.6, xshift=1.9, fill=white] [midway]{$y_{1}$}($(n-1-4.south) + (-1mm,0mm)$) ;
        \draw[->] ($(n-2-4.north) + (-3mm,0mm)$)  to  node[scale=0.6,xshift=-1.9, fill=white] [midway]{$y_{0}$}($(n-1-4.south) + (-3mm,0mm)$) ;

     \draw[->] ($(n-2-7.north) + (4mm,0mm)$)  to  node[scale=0.6, xshift=1.9, fill=white] [midway]{$y'_{1}$}($(n-1-7.south) + (4mm,0mm)$) ;
        \draw[->] ($(n-2-7.north) + (2mm,0mm)$)  to  node[scale=0.6,xshift=-1.9, fill=white] [midway]{$y'_{0}$}($(n-1-7.south) + (2mm,0mm)$) ;

 \draw[->, red] ($(n-1-7.south west) + (3mm,0mm)$) to  node[scale=0.6, yshift=3.5, xshift=-1mm, fill=white] [midway]{$v_{00}$}($(n-2-4.east) + (-6mm,3mm)$) ;
    \draw[->, red] ($(n-1-7.south west) + (7mm,0mm)$) to node[fill=white, yshift=-3.5, scale=0.6, xshift=1mm] [midway]{$v_{21}$} ($(n-2-4.east) + (-2mm,3mm)$);
    \draw[white] ($(n-1-7.south west) + (5mm,-0mm)$)  to  node[fill=white, scale=0.3 ] [midway]{\bf $\textcolor{red}{\ddots}$}($(n-2-4.east) + (-4mm,3mm)$) ;

   \draw[->, red] ($(n-1-7.north west) + (7mm,-0mm)$) .. controls +(-3.7mm,16mm)
and +(+3.7mm,16mm) .. node[fill=white, scale=0.6] [midway]{$z_{12}$} ($(n-1-1.north east) + (-5mm,0mm)$);
   
      \draw[->, red] ($(n-1-7.north west) + (5mm,-0mm)$) .. controls +(-3.7mm,11mm)
and +(+3.7mm,11mm) .. node[fill=white, scale=0.6] [midway]{$z_{02}$}($(n-1-1.north east) + (-3mm,0mm)$);  
    
      \draw[->, red] ($(n-1-7.north west) + (3mm,-0mm)$) .. controls +(-3.7mm,6mm)
and +(+3.7mm,6mm) .. node[fill=white, scale=0.6] [midway]{$z_{01}$}($(n-1-1.north east) + (-1mm,0mm)$);

\end{tikzpicture}\end{equation*}
where the notation of the vertices is given by the corresponding line bundles on $\P^2 \times \P^1$.
Moreover,
the relations of $A$ are given by all paths containing at least two red arrows together with relations that arise as cyclic derivatives of a potential $W$, see Corollary \ref{C:QuotientPreproj3}.
The potential is given by 
\begin{align*}
W=  \sum_{i=0}^2\sum_{j=0}^1 v_{ij} \rho_{v_{ij}} +   \sum_{0 \leq i <j \leq 2} z_{ij} \rho_{z_{ij}},
\end{align*}
where the $\rho_\bullet$ are the `minimal relations' of $\End(\sL)$, i.e.
\begin{align*}
 \rho_{v_{ij}} =y'_jx'_{1i}- x_{1i}y_j, \quad  \rho_{z_{ij}}=x_{1i}x_{0j}-x_{1j}x_{0i}.
\end{align*}
\end{example}

\appendix
\section{Tilting on weighted projective spaces \texorpdfstring{$\P(1^d, m)$}{P(d, m)\textasciicircum wps}, \texorpdfstring{\\ by Martin Kalck, Yujiro Kawamata and Nebojsa Pavic}{}}\label{sec:appendix}

\subsection{Overview}
\noindent
 Let $m,d > 1$ be integers and let $\mathbb{k}$ be a field\footnote{Not necessarily algebraically closed.} with $\chr( \mathbb{k})$ coprime to $m$. We construct explicit tilting objects $\sT_m^d$ for all weighted projective spaces\footnote{See Subsection \ref{sub:prelim} below for the definition.} of the form $X_m^d=\P(1^d, m)$. 

This generalizes the threefold case $\P(1, 1, 1, 3)$ studied in \cite{kawamata3}. We note that for all weighted projective surfaces $\P(a, b, c)$, semiorthogonal decompositions into derived categories of finite dimensional local algebras have been constructed in \cite{kks}. The tilting result below gives an (a priori) different description for the derived categories surfaces $\P(1, 1, m)$.
Also, for the very special case $\P(1^d, d)$, which is a projective cone over $\P^{d-1}$, we construct similar explicit semiorthogonal decompositions in Subsection \ref{subsec:ProjCones}. The tilting objects  $\sT_m^d$ constructed below can also be decomposed to obtain semiorthogonal decompositions into derived categories of finite dimensional algebras. Actually, using this explicit description of $\sT_d^d$ and noncommutative Gr{\"o}bner bases was our original approach to describe the algebras $R_d$ in Corollary \ref{Cor:description_R_d}. 

The space $X_m^d$ has a unique singular point.  
It is a singularity of the cone over the Veronese embedding of $\P^{d-1}$ of degree $m$.
This singularity is Gorenstein if and only if $m$ divides $d$ and  
the blow-up of the singularity is a crepant resolution
if and only if $m=d$. In particular, the results in this appendix yield first tilting results for derived categories of higher dimensional \emph{non}-Gorenstein varieties.

As a consequence, we obtain derived equivalences with module categories of finite dimensional algebras, which in turn imply triangle equivalences between the corresponding Buchweitz--Orlov singularity categories, which we call \emph{singular equivalences}. This extends the singular equivalences for all cyclic quotient surface singularities obtained in \cite{kalck-karmazyn} to higher dimensional cyclic quotient singularities of the form $\frac{1}{m}(1^d)$. 

Note that, in the Gorenstein case, such singular equivalences have independently been obtained by Hanihara \cite{Hanihara1, Hanihara2} (via graded singularity categories) and the case  $d=m=3$
also follows from \cite{BIRSc} in combination with a Morita-type Theorem for cluster categories \cite{KellerReiten08}.

For $d=3$ and $m=2$, 
Corollary \ref{C:SingEq} generalizes the main result of \cite{KalckNewSingular} to all fields of characteristic different from $2$. 
For all other non-Gorenstein singularities of the form $\frac{1}{m}(1^d)$ in dimension $d\geq 3$  these singular equivalences are new.

\subsection{Preliminaries and definitions}\label{sub:prelim} 
Following \cite{dolgachev} and \cite{kawamata0}, we recall basic notation and properties of weighted projective spaces $\P(1^d, m):=\P (1 , \ldots , 1 , m)$, where the weight $1$ appears $d$ times.

Consider the graded algebra $S= \mathbb{k}[x_0 , \ldots , x_d ]$, where the variables $x_0 , \ldots , x_{d-1} $ have degree one and $x_d$ has degree $m$ and define
\[
 \P (1^d , m ) = \Proj ( S ).
\] 
The weighted projective space $\P ( 1^d , m )$ can also be viewed as a the quotient variety $\P^d / \Z_m $, where the action of $\Z_m$ on $\P^d = \P^d_{[z_0 : \ldots : z_d ]}$ via 
\[
[z_0 : \ldots : z_d ] \mapsto [z_0 : \ldots : z_{d-1} : \zeta z_d ],
\]
with $\zeta $ a fixed primitive $m$-th root of unity.
We denote the $\Z_m$-equivariant quotient morphism by $\pi : \P^d \to \P^d / \Z_m$.
Recall that there is a reflexive sheaf $\OO (i)$ (up to isomorphism) of rank one on $\P (1^d , m)$ for $i \in  \Z$, which corresponds to the $i$-th multiple of the generator of the class group $\Cl ( X_m^d )\cong \Z$.

We recall further that the quotient stack $Y_m^d : =  [ \P^d / \Z_m ]$ (see \cite[Section 5]{kawamata0} for a definition) corresponding to the above action comes with  morphisms $\sigma : [ \P^d / \Z_m ] \to \P^d / \Z_m$ and $\pi' : \P^d \to [ \P^d / \Z_m ] $, which compose to $\pi$.
Moreover, we have the identification
\[
\Db ( [\P^d / \Z_m ] ) \cong \Db_{\Z_m} ( \P^d ) ,
\]
where the right hand side is the derived category of $\Z_m$-equivariant coherent sheaves on $\P^d$. 
Denote by $D_i$, for $0\leq i \leq n$ the divisor on $\P^d$ given by $z_i = 0$ and let $\sD_i$ be the image $\sD_i = \pi'_\ast ( D_i )$ in $Y:=Y_m^d$.
We set $\OO_Y (1) : = \OO_Y ( \sD_0 ) $.

\begin{definition}
Let $X$ be a projective variety (or a quotient stack).

An object $T$ of $\Dperf(X)$ is called a \emph{tilting object}, if it satisfies the following conditions:
\begin{enumerate}
    \item $T$ generates $\DD(\Qcoh(X))$, i.e. for all $A$ in $\DD(\Qcoh(X))$, if $\Hom^\bullet_{X} (T , A) = 0 $ then $A=0$. 
    \item $\Hom (T , T [i]) = 0$, for all $i \neq 0$. 
\end{enumerate}
\end{definition}

It is well-known that Beilinson's Theorem generalizes to weighted projective stacks $[\P^d / \Z_m ]$.

\begin{theorem}[{see \cite[Section 5]{kawamata0}, or \cite[Proposition 2.7 and Theorem 2.12]{auroux-katzarkov-orlov}}]\label{thm:stack-exc-coll}
   The category $\Db(Y_m^d) \cong \Db_{\Z_m} ( \P^d )$ has a full and strong exceptional collection (where $Y:=Y_m^d$)
\[
  \OO_Y (-m+1) , \OO_Y ( -m+2 ) ,\ldots , \OO_Y , \OO_Y (1) , \ldots , \OO_Y (d) .
\]
Moreover, we have the following isomorphisms for all $1\leq i , j \leq m+d$:
\begin{equation}\label{eq:equiv-except-hom}
    \Hom_Y ( \OO_Y (-m + i ) , \OO_Y (-m+j) [p] ) \cong \begin{cases}
        S_{j-i},& \text{if $i\leq j$ and $p=0$} \\
        0,& \text{else}
    \end{cases}
\end{equation}
where $S_l$ denotes the degree $l$ component of the graded polynomial ring $S= \mathbb{k}[x_0, \ldots, x_d]$.
In particular, $\Db (Y_m^d)$ has a tilting object
\begin{equation}\label{eq:tilting-wps}
    \sB_m^d :=\OO_Y (-m+1) \oplus \OO_Y (-m+2) \oplus \ldots \oplus \OO_Y \oplus \OO_Y (1)  \oplus \ldots \oplus \OO_Y (d) .
\end{equation}
\end{theorem}

\subsection{Main result}
The following theorem is the main result of this appendix.

\begin{theorem}\label{T:MainGenerald}
Let $m, d >1$ be integers.
Then the $d$-dimensional weighted projective space $X_{m}^d =\P(1^d, m)$ 
has a tilting object $\sT_{m}^d$. In particular, there is a triangle equivalence
\begin{align}\label{eq:tiltingAppendix}
\Db(X_{m}^d ) \cong \Db ( \End_{X_m^d}(\sT_{m}^d )\op). 
\end{align}
The tilting object is given explicitly as 
\begin{align}
\sT_m^d = \bigoplus_{l=1}^d \sG_l \oplus \OO ,
\end{align}
where the $\sG_l:=\sG_{l,d,m}$ are given as the following two-term complexes
\begin{align}
\sG_{l,d,m}:=\begin{cases} &\OO \otimes S_l \oplus \OO (r_l) \xrightarrow{\displaystyle \mu \oplus x_d^{k_l}} \OO(l) \quad \text{if } r_l \neq 0,  \\ &\OO \otimes S_l  \xrightarrow{\displaystyle \qquad \qquad \mu    \qquad} \OO(l) \quad \text{else. }\end{cases}
\end{align}
Here, $\mu$ denotes the natural multiplication map, the first term of the complexes is concentrated in degree $0$. Moreover, for $l \in \Z_{>0}$ the integers $k_l$ and $r_l$ are given by
\begin{align}
l=k_l \cdot m + r_l, \quad \text{where} \quad  k_l \geq 1 \quad \text{and} \quad  -m < r_l \leq 0. 
\end{align}
Moreover, decomposing the tilting object $\sT_m^d$
yields a semiorthogonal decomposition\footnote{
{with $\Db ( \End ( \sG )\op ) \subseteq \Db ( X_m^d )$ admissible.}}
    \begin{align}\label{eq:SODwpsgeneral}
        \Db ( X_m^d ) \cong \langle \Db ( \End ( \sG )\op ) , \OO  \rangle,
    \end{align}
    where $\sG = \bigoplus_{l=1}^d \sG_{l , d , m}$.
     If $m$ divides $d$, or, equivalently, when $X_m^d$ is Gorenstein, then this semiorthogonal decomposition is admissible. If $m=d$, we get an admissible semiorthogonal decomposition 
     \begin{align}\label{eq:SODwps2}
        \Db ( X_m^d ) \cong \left\langle \Db \left( \End\left(  \bigoplus_{l=1}^{d-1} \sG_{l , d , m} \right)\op \right) , \OO, \OO(d)  \right\rangle.
    \end{align}
\end{theorem}

\begin{proof}
By \cite[Theorem 7.6]{hille-van-den-bergh}\footnote{In \cite{hille-van-den-bergh} the base field is assumed to be $\C$ but the result is based on \cite{BondalVdB}, which holds over any field.}, the claimed triangle equivalence \eqref{eq:tiltingAppendix} follows once we show that $\sT_m^d$ is tilting.  
We first show that $\sT :=\sT_m^d$ is perfect. Since $\OO(am)$ is a line bundle and hence perfect for all $a \in \Z$, it suffices to show that 
\begin{align}\label{E:Cone}
\Cone(\sG_l \to  \OO\otimes S_l) = \Cone(\OO(r_l) \xrightarrow{x_d^{k_l}} \OO(l) )
\end{align}
for $r_l \neq 0$ is perfect. 
The latter is quasi-isomorphic to $\OO_{k_lZ}(l)$, where $Z \subset X$ is the zero locus of $x_d$. 
Since the unique singular point $[ 0:\ldots:0:1 ]$ does not lie in $Z$, we see that $\OO_{k_lZ}(l)$ is perfect.

We will show in Lemma \ref{L:generation} below, that the smallest thick triangulated subcategory $\mathsf{S}$ of $\Db(X)$ containing $\sT_m^d$
contains $\OO_Z(1), \ldots, \OO_Z(d)$, which generate $\Db(Z)$ by applying Beilinson's work to $Z=\P^{d-1}$. 
By definition, $\mathsf{S}$ also contains $\OO$.
Combining this,  we see that $\mathsf{S}$ contains all divisorial sheaves $\OO(m)$ on $X^d_m$. They generate $\DD(\Qcoh X)$ by \cite[Lemma 36.16.1.] {stacks-project}. Hence, $\sT_m^d$ generates $\DD(\Qcoh X)$ as desired.

It remains to show that $\Hom^\bullet (\sT_m^d , \sT_m^d )$ is concentrated in degree $0$. To show this, we apply the fully faithful functors
\begin{align}
\Perf(X_m^d) \xrightarrow{\pi^*} \Db(Y_m^d ) \xrightarrow{\beta:=\Hom_{Y_m^d}(\sB_m^d , -)}  \Db(\End_{Y_m^d}(\sB_m^d )\op) 
\end{align}
and compute the isomorphic $\Hom^\bullet_{B_m^d}(\beta\pi^*(\sT_m^d), \beta\pi^*(\sT_m^d ))$, where $B_m^d :=\End_{Y_m^d}(\sB_m^d )\op$ and $\sB_m^d$ is the tilting bundle in $\Db (Y_m^d)$ as defined in \eqref{eq:tilting-wps}.
We label the indecomposable projective $B_m^d$-module corresponding to $\OO_Y(i)$ by $P_i$. Then the image of $\sG_l:=\sG_{l,d,m}$ under $\beta\pi^*$ is given by
\begin{align}
\sP_{l,d,m}:=\begin{cases} &P_0 \otimes S_l \oplus P_{r_l} \xrightarrow{\displaystyle \mu \oplus x_d^{k_l}} P_l \quad \text{if } r_l \neq 0,  \\ &P_0 \otimes S_l  \xrightarrow{\displaystyle \qquad \quad \mu    \qquad} P_l \quad \text{else, }\end{cases}
\end{align}
where again $\mu$ is the evaluation map.
Since there are no non-zero morphisms $P_i \to P_j$ for $j<i$, looking at the homotopy category $K^b(\proj B_m^d ) \cong \Db(B_m^d)$, we see that $\Hom^\bullet_{B_m^d}(\sP_{l,d,m}, \sP_{l,d,m})$ is concentrated in positive degrees.
Moreover, by construction, every morphism $P_i \to P_l$, with $i\leq 0$ and $1 \leq l \leq d$ factors over the morphism defining
$\sP_{l,d,m}$. Applying Lemma \ref{L:Approximation} completes the proof. 

This also shows that $\Hom(\OO, \sG)=0$, which implies \eqref{eq:SODwpsgeneral}, by a result of Miyachi, cf. e.g. \cite[Proposition 3.2] {HocheneggerKPloogRepArxiv}.
Indeed, let $\Lambda:=\End(\sG \oplus \OO)\op$ and let $e$ be the idempotent corresponding to $\OO$. Since $\Hom(\OO, \sG)=0$, there are no arrows ending\footnote{We are looking at the opposite algebra here.} in the vertex $e$ in the quiver of $\Lambda$. This implies that $\Lambda e \Lambda= \Lambda e$ and hence $\Lambda/e:=\Lambda/\Lambda e \Lambda \cong \Lambda(1-e)$ is a projective left $\Lambda$-module and since $e\Lambda e \cong \End \OO \cong \mathbb{k}$ has finite global dimension, the idempotent $e$ is ``recollant'' in the sense of \cite[Definition 3.1] {HocheneggerKPloogRepArxiv}. Now \cite[Proposition 3.2] {HocheneggerKPloogRepArxiv} yields a semiorthogonal decomposition
$\Db(\Lambda)=\langle \Db(\Lambda/e), \Db(e \Lambda e)\rangle $, with $\Db(\Lambda/e) \subset \Db(\Lambda)$ admissible. The left hand side identifies with $\Db(X_m^d)$ by tilting and $\Lambda/e \cong (1-e)\Lambda(1-e) \cong \End(\sG)\op$, whereas $\Db(e \Lambda e) \cong \thick(\OO)$. Since $\OO$ is perfect, if $X_m^d$ is Gorenstein, the semiorthogonal decomposition is  admissible by \cite[Lemma 2.15]{KPS}.

Similarly, \eqref{eq:SODwps2} follows by replacing $\sG_d$ by $\OO(d)$. Since $\OO$ is a direct summand of $\sT_m^d$ this does not affect the generation and $\Hom(\OO(d), \bigoplus_{l=1}^{d-1} \sG_{l , d , m} )=0$ follows by degree reasons.
The existence of an admissible semiorthogonal decomposition can now be deduced as in \eqref{eq:SODwpsgeneral}. 
\end{proof}

\begin{lemma}\label{L:generation} 
Let $\mathsf{S}$ be the smallest thick triangulated subcategory of $\Db(X)$ containing $\sT_m^d$. Then $\mathsf{S}$
contains $\OO_Z(1), \ldots, \OO_Z(d)$, where $Z \subset X$ is the zero locus of $x_d$.
\end{lemma}
\begin{proof}
As above let 
\begin{align}
l=k_l \cdot m + r_l, \quad \text{where} \quad  k_l \geq 1 \quad \text{and} \quad  -m < r_l \leq 0. 
\end{align}
We first note that \eqref{E:Cone} shows $\OO_{k_lZ}(l) \in \mathsf{S}$ for all $1 \leq l \leq d$. 
We use induction on $k_l$ to show that $\OO_{Z}(l) \in \mathsf{S}$ for all $1 \leq l \leq d$ as claimed. For $k_l=1$, and thus when $0 \leq l \leq m-1 $, there is nothing to show. Assume that we have shown the statement for all $k_l<n$. We show the statement for $k_l=n$. Set $l_n = n \cdot m + r_l$. There are triangles for all $k$ and $l$:
\begin{align}\label{E:extensions}
\OO_Z(l- (k-1)m) \to \OO_{kZ}(l) \to \OO_{(k-1)Z}(l)
\end{align}
By induction hypothesis $\OO_Z(l_n -(n-1)m)$ is in $\mathsf{S}$ (note that $r<s$ implies $k_r<k_s$). By the discussion above $\OO_{n Z}(l_n ) \in \mathsf{S}$ and thus $\OO_{(n-1)Z}(l_n)$ is in $\mathsf{S}$. Hence using \eqref{E:extensions} 
repeatedly shows the claim.
\end{proof}

\begin{lemma}\label{L:Approximation}
Let $A$ be a ring. Let $\{\sP_i\}_{i=1}^4$ be finitely generated projective $A$-modules (possibly zero). Consider
complexes of length at most $2$: $C_1:=\sP_1 \xrightarrow{a} \sP_2$ and $C_2:=\sP_3 \xrightarrow{b} \sP_4$. Here $\sP_1$ and $\sP_3$  are both concentrated in the same degree $i$.
Moreover, assume the following:
\begin{itemize}
\item[(a)] If $j \neq 2$, then every morphism $\sP_j \to \sP_2$ factors over $a$. 
\item[(b)] If $k \neq 4$, then every morphism $\sP_k \to \sP_4$ factors over $b$. 
\end{itemize}
Let $C=C_1 \oplus C_2$. Then $\Hom_{K^b(\proj-A)}(C, C[i])=0$ for all $i>0$.
\end{lemma}

\begin{proof}
This follows from the definition of morphisms in $K^b(\proj-A)$. Indeed, using our assumptions every non-trivial morphism of complexes $C \to C[i]$ is null homotopic.
\end{proof}

\begin{remark}
Note that in the Gorenstein case, i.e. when $m$ divides $d$, the semiorthogonal decompositions of $\Db( \P (1^d , m ) )$ in  \eqref{eq:SODwpsgeneral} \& \eqref{eq:SODwps2} gives a \emph{finite dimensional categorical absorption of singularities} in the sense of \cite{kalck-klapproth-pavic}. If $m=d$, then one can show that the algebra $\End\left(  \bigoplus_{l=1}^{d-1} \sG_{l , d , d} \right)\op$ in \eqref{eq:SODwps2} is given by the quiver with relations in Proposition \ref{P:Hanihara} and Corollary \ref{Cor:description_R_d}.

Note further that the component $\Db (\End (\sG ) )$ in the semiorthogonal decomposition in \eqref{eq:SODwpsgeneral} contains exceptional objects if $m<d$ and $\Db ( \End ( \sG ) )$ will thus not be a minimal absorption. 
\end{remark}

\subsection{Application to singular equivalences}

We start by exhibiting a slightly different tilting object for $X_2^3$, which is better suited for the application below.

\begin{proposition}\label{p:tilt-on-x23}
We use the notation of Theorem \ref{T:MainGenerald} and consider the weighted projective threefold $X_2^3:=\P(1^3, 2)$. 

Then 
$\sT=\OO(-2) \oplus \sG_1 \oplus \sG_2 \oplus \OO$ is a tilting object. 
Moreover, the quiver $Q$ of $\End(\sT)\op$ has the following shape\footnote{In other words, the sequence of subcategories
$\thick(\OO(-2)), \thick(\sG_1), \thick(\sG_2), \thick(\OO)$ 
is semiorthogonal. It is not an exceptional sequence since $\sG_1$ is not exceptional as we will see below.}
\begin{equation} \label{eq:quivershape}
  \begin{tikzpicture}[description/.style={fill=white,inner sep=2pt}]

    \matrix (n) [matrix of math nodes, row sep=2em,
                 column sep=2em, text height=1.5ex, text depth=0.25ex,
                 inner sep=2pt, nodes={inner xsep=0.3333em, inner
ysep=0.3333em}] at (-5, 0)
    {   \OO(-2)  && \sG_1 && \sG_2 && \OO \\
          };

    \draw[<-, dashed] ($(n-1-1.east) + (0,0mm)$)  to  ($(n-1-3.west) + (0mm,0mm)$) ;

    \draw[<-, dashed] ($(n-1-3.east) + (0,0mm)$)  to  ($(n-1-5.west) + (0mm,0mm)$) ;

       \draw[<-, dashed] ($(n-1-5.east) + (0,0mm)$)  to  ($(n-1-7.west) + (0mm,0mm)$) ;

     \draw[->, dashed] ($(n-1-7.south west) + (2mm,-0mm)$) .. controls +(-3.7mm,-16mm)
and +(+3.7mm,-16mm) ..  ($(n-1-1.south east) + (-3mm,0mm)$);

 \draw[<-, dashed] ($(n-1-3.south west) + (2mm,-0mm)$) .. controls +(-3.7mm,-5mm)
and +(+3.7mm,-5mm) ..  ($(n-1-3.south east) + (-3mm,0mm)$);

   \draw[->, dashed] ($(n-1-7.north west) + (2mm,-0mm)$) .. controls +(-3.7mm,8mm)
and +(+3.7mm,8mm) ..  ($(n-1-3.north east) + (-3mm,0mm)$);

      \draw[->, dashed] ($(n-1-5.south west) + (3mm,-0mm)$) .. controls +(-3.7mm,-8mm)
and +(+3.7mm,-8mm) .. ($(n-1-1.south east) + (-1mm,0mm)$);
  
\end{tikzpicture}  
\end{equation}
where each arrow indicates that there can be $\alpha \geq 0$ arrows between these vertices in the quiver $Q$. Whereas no arrow between any two vertices indicates that there can be \emph{no} arrows between these vertices in $Q$.

\end{proposition}
\begin{proof}
    That the $\sG_i$ are perfect follows from Theorem \ref{T:MainGenerald}. In combination with the fact that  $\OO(-2)$ is a line bundle, this shows that $\sT$ is perfect.

    Moreover, $\thick(\sT)$ contains \begin{align}
        \Cone(\OO(i-2) \to \OO(i)) \cong \OO_Z(i)
    \end{align} 
    for $0 \leq i \leq 2$, which generate $\Db(Z)$ since $Z \cong \P^2$. Since $\OO$ is a summand of $\sT$, proceeding as in the proof of Theorem \ref{T:MainGenerald} shows that $\sT$ generates $\DD(\Qcoh(X_2^3))$.

    Note that the summands $\sG_1, \sG_2$ and $\OO$ of $\sT$ are also summands of the tilting object $\sT_2^3$ from Theorem \ref{T:MainGenerald}. So to show that $\sT$ is tilting it is enough to show vanishing of $\Hom(\OO(-2), \sT[s])$ and $\Hom(\sT, \OO(-2)[s])$ for all $s \neq 0$. As in the proof of Theorem \ref{T:MainGenerald}, the main idea is to do the computations in the homotopy category $K^b(\proj-B_2^3)$ together with Lemma \ref{L:Approximation}, the fact that there are no homomorphisms $P_i \to P_j$ for $i>j$ and the definitions of $\sG_i$ which guarantee that all homomorphisms $P_i \to P_l$ for $i\leq 0$ and factor over the morphisms defining $\sG_i$.

    Finally, the computations showing that the quiver of $\End(\sT )^{\op}$ has the claimed shape are again very similar. 
    We only explain why there are no morphisms  $\sG_2 \to \sG_1$ (and thus no arrows $\sG_1 \to \sG_2$ in the quiver of $\End(\sT)\op$). Since they are part of a tilting object, we already know that such a morphism has to be in degree $0$:
    \begin{align*}
        \xymatrix{
    \OO \otimes S_2 \ar[dd]_{{\begin{pmatrix} i \\ 0 \end{pmatrix}}} \ar[rr] && \OO(2) \ar[dd]^{0} \\ \\
    \OO \otimes S_1 \oplus \OO(-1) \ar[rr]^-{{  \begin{pmatrix} \mu & x_3 \end{pmatrix}}} && \OO(1) }
    \end{align*}
    Since $\OO$ is exceptional the entries of the matrix $i$ are scalars. In particular, if $i \neq 0$ the composition $$  \begin{pmatrix} \mu & x_3 \end{pmatrix} \circ \begin{pmatrix}
        i && 0
    \end{pmatrix}^T$$ 
    is non-zero and thus cannot be part of a morphism of complexes, since $\Hom(\OO(2), \OO(1))=0$. This completes the proof.
\end{proof}

For Noetherian (possibly noncommutative) algebras, the following definition goes back to Buchweitz \cite{Buchweitz}. For quasi-projective varieties the definition is due to Orlov \cite{orlov-sing-1}.

\begin{definition}
    Let $X$ be a quasi-projective variety. The \emph{singularity category} of $X$ is the triangulated quotient category
    \begin{align}
        \Dsg(X):=\frac{\Db(\Coh \, X)}{\Dperf(X)}. 
    \end{align}
    Let $R$ be a left Noetherian ring. The \emph{singularity category} of $R$ is the triangulated quotient category
    \begin{align}
        \Dsg(R):=\frac{\Db(R\mbox{-}\mod )}{\Dperf(R)} 
    \end{align}
\end{definition}

\noindent
We generalize
 the main result of \cite{KalckNewSingular} to all  
 fields $\mathbb{k}$ of characteristic different from two.

\begin{corollary}\label{C:SingEq}
   In the notation of Proposition \ref{p:tilt-on-x23}, there are triangle equivalences
   \begin{align}\label{eq:singequivalences}
       \Dsg\left(\frac{1}{2}(1,1,1) \right) \cong \Dsg(X_2^3) \cong \Dsg(\End(\sT)\op) \cong \Dsg\left(\frac{\mathbb{k}[z_1, z_2, z_3]}{(z_1, z_2, z_3)^2}\right),
   \end{align}
   where $\frac{1}{2}(1,1,1)$ denotes the invariant ring of $\Z/2\Z$ acting on $\mathbb{k}\llbracket u_1, u_2, u_3\rrbracket$ with weights $(1, 1, 1)$.
   
   In particular, the Grothendieck groups of these triangulated categories are isomorphic to $\Z/4\Z$.
\end{corollary}

\begin{remark}
    The proof below is very different from \cite{KalckNewSingular}. In particular, this yields an alternative way to compute the Grothendieck group of the singularity category of $X_2^3$ and of the complete local coordinate ring of its singular point. In contrast, in loc. cit. we used knowledge of this Grothendieck group for algebraically closed fields of characteristic zero
\cite{pavic-shinder} to deduce the singular equivalence below.
\end{remark}

\begin{proof}
    By \cite[Theorem 7.6]{hille-van-den-bergh}, the tilting object $\sT$ from Proposition \ref{p:tilt-on-x23} induces a triangle equivalence
    \begin{align}
        \Db(X_2^3) \cong \Db(\End(\sT)\op),
    \end{align}
    which in turn induces a triangle equivalence on the quotient categories
    \begin{align}
        \Dsg(X_2^3) \cong \Dsg(\End(\sT)\op)).
    \end{align} This shows the equivalence in the middle of \eqref{eq:singequivalences}.
    Moreover, the category on the right is idempotent complete since $\End(\sT)\op$ is a finite dimensional $\mathbb{k}$-algebra \cite[Corollary 2.4]{chenradical}. Now, since the completion of the local ring at the singular point of $X_2^3$ is isomorphic to $\frac{1}{2}(1,1,1)$ (note that we use here that the charactersitic of $\mathbb{k}$ is coprime to $m=2$), the left equivalence in \eqref{eq:singequivalences} follows from \cite{OrlovIdemp}, since singularity categories of complete local rings are idempotent complete.

    It remains to show the equivalence on the right of \eqref{eq:singequivalences}. We observe that by iteratively removing sources or sinks (and all adjacent arrows) the quiver \eqref{eq:quivershape} of $\End(\sT)\op$ can be reduced to a quiver with a single vertex at $\sG_1$. Then \cite[Lemma 4.3.]{chenradical} implies
    \begin{align} 
        \Dsg(\End(\sT)\op) \cong \Dsg(\End(\sG_1)\op)
    \end{align}
    We finish the proof by showing 
    \begin{align}
        \End(\sG_1)\op \cong \frac{\mathbb{k}[z_1, z_2, z_3]}{(z_1, z_2, z_3)^2}.
    \end{align}
    A general endomorphism has the following form
    \begin{align} \label{eq:chainm}
        \xymatrix{
    \OO^3 \oplus \OO(-1) \ar[ddd]_{{\begin{pmatrix}
    * & * &  *  & a\cdot l_0\\ 
    * & * &  *  & b\cdot l_1 \\
    * & * &  *  & c\cdot l_2  \\
     0 & 0 & 0 & \mu \cdot \id  
    \end{pmatrix}}} \ar[rrr]^{\quad \begin{pmatrix} x_0 & x_1 & x_2  & x_3 \end{pmatrix}} &&& \OO(1) 
    \ar[ddd]^{\lambda \cdot \id} 
    \\ \\ \\
    \OO^3 \oplus \OO(-1) \ar[rrr]^{\quad \begin{pmatrix} x_0 & x_1 & x_2 & x_3 \end{pmatrix}} &&& \OO(1), 
    }
    \end{align}
    where the $l_i$ are linear forms in $x_0, x_1, x_2$.
    Commutativity of this diagram implies $\lambda=\mu$, that the submatrix of scalars $*$ equals $\lambda \cdot \id$ and that the composition
    \begin{align}\label{eq:compo}
        \begin{pmatrix} x_0 & x_1 & x_2 \end{pmatrix} \circ \begin{pmatrix}
        a\cdot l_0 & b\cdot l_1  & c\cdot l_2
    \end{pmatrix}^T 
        \end{align}
        vanishes. 
        Observe that \begin{align}
           \dim_k \Hom(\OO(-1), \OO^3) = 3 \cdot 3 =9 \\
           \dim_k \Hom(\OO(-1), \OO(1)) = \dim_k S_2=7
        \end{align}
        The compositions in \eqref{eq:compo} span a subspace of $\Hom(\OO(-1), \OO(1))$ of dimension $7-1=6$ (we hit all monomials of degree $2$ except for $x_3$). Summing up the subspace $S \subseteq  \Hom(\OO(-1), \OO^3)$ such that \eqref{eq:compo} vanishes is $9-6=3$ dimensional, this yields a $3$-dimensional subspace $R=<z_1, z_2, z_3> \subseteq \End(\sG_1)$. 
     Moreover, any two elements in $R$ compose to zero (since there are no non-zero maps from $\OO$ to $\OO(-1)$). 
Similarly, there are no non-trivial chain homotopies in \eqref{eq:chainm}. So $\End(\sG_1)=<\id, z_1, z_2, z_3>$ as vector spaces and the discussion above implies that $\End(\sG_1)\cong \frac{\mathbb{k}[z_1, z_2, z_3]}{(z_1, z_2, z_3)^2}$ is commutative. 

Finally, the statement about the Grothendieck groups follows since $\mathbb{k}[z_1, z_2, z_3]/(z_1, z_2, z_3)^2$ is a local algebra of $\mathbb{k}$-dimension $4$, cf. e.g. \cite[Proposition 5.2.]{kalck-karmazyn}\footnote{Note that in this example, we don't need $\mathbb{k}$ to be algebraically closed, since the simple module has $k$-dimension $1$.}.
\end{proof}

\begin{remark} 
    There is another family of singular equivalences between commutative rings \cite{YangPrivate, kalck-karmazyn} that can be obtained in an analogous way:
    \begin{align} \label{eq:SingEqu}
        \Dsg(X_n^2) \cong
        \Dsg\left(\frac{1}{n}(1,1)\right) \cong \Dsg\left(\frac{\mathbb{k}[z_1, \ldots, z_{n-1}]}{(z_1, \ldots, z_{n-1})^2}\right),
    \end{align}
    where $\mathbb{k}$ is a field with characteristic coprime to $n$. 

    Again the tilting object in Theorem \ref{T:MainGenerald} is modified in a way analogous to Proposition \ref{p:tilt-on-x23} 
    \begin{align}
        \sT_n:=\OO(-n) \oplus \sG_{1, 2, n} \oplus \OO.
    \end{align}
    
  Then the quiver of $\End(\sT_n)\op$ has a similar shape as in Proposition \ref{p:tilt-on-x23} and the same reductions as in the proof of Corollary \ref{C:SingEq} yield a singular equivalence with $\End(\sG_{1, 2, n})$. Similarly, to the computations in the proof of Corollary \ref{C:SingEq}, one checks that
  \begin{align}
        \dim_\mathbb{k} \End(\sG_{1, 2, n})=1 + \dim_\mathbb{k} \Hom(\OO(-n+1), \OO^2) - (\dim_\mathbb{k} \Hom(\OO(-n+1), \OO(1))-1)\\ =1 + 2n -(n+2-1) = n
  \end{align}
  and that all morphisms in the radical square to zero. This shows $\End(\sG_{1, 2, n})\cong \frac{\mathbb{k}[z_1, \ldots, z_{n-1}]}{(z_1, \ldots, z_{n-1})^2}$ and completes the sketch of proof of \eqref{eq:SingEqu}.
\end{remark}

\providecommand{\arxiv}[1]{{\tt{arXiv:#1}}}

\end{document}